\documentclass[letterpaper,10pt,oneside,reqno]{amsart}

\usepackage{amsmath,amssymb,amsthm,amsfonts}
\usepackage[colorlinks=true,linkcolor=blue,citecolor=red]{hyperref}
\usepackage{graphicx,color}
\usepackage{upgreek}
\usepackage[mathscr]{euscript}
\usepackage{bbm}
\usepackage{mathrsfs}
\usepackage{mathtools}
\allowdisplaybreaks
\numberwithin{equation}{section}

\usepackage{tikz}
\usetikzlibrary{shapes,arrows,positioning,decorations.markings}

\usepackage{array}
\usepackage{adjustbox}
\usepackage{cleveref}
\usepackage{enumerate}
\usepackage{caption}
\usepackage{verbatim}
\usepackage[DIV=12]{typearea}

\usepackage[shadow=true]{todonotes}

\newcommand{\1}{\mathbf{1}}
\newcommand{\R}{\mathcal{R}}
\newcommand{\sgn}{\text{sgn}}
\newcommand{\z}{\mathrm{z}}
\renewcommand{\Re}{\mathrm{Re}}
\newcommand{\LL}{\mathrm{L}}
\newcommand{\RR}{\mathrm{R}}
\newcommand{\Fic}{\mathbb{F}_{\mathrm{ic}}}
\newcommand{\ddbar}[2]{\frac{{\mathrm d}#1}{2\pi {\mathrm i}#2}}
\newcommand{\polylog}{\mathrm{Li}}
\newcommand{\dd}{{\mathrm d}}
\newcommand{\ii}{\mathrm i}
\newcommand{\cchi}{\mathrm{ch}}
\newcommand{\conf}{\mathcal{X}}
\newtheorem{proposition}{Proposition}[section]
\newtheorem{lemma}[proposition]{Lemma}
\newtheorem{corollary}[proposition]{Corollary}
\newtheorem{theorem}[proposition]{Theorem}
\theoremstyle{definition}
\newtheorem{definition}[proposition]{Definition}

\newtheorem{remark}[proposition]{Remark}
\newtheorem{assum}[proposition]{Assumption}
\newtheorem{notation}[proposition]{Notation}

\begin{document}

\title[Multi-point distribution of discrete time  periodic TASEP]
{Multi-point distribution of discrete time  periodic TASEP} 
\author{Yuchen Liao}

\date{}
\begin{abstract}
	We study the one-dimensional discrete time totally asymmetric simple exclusion process with parallel update rules on a spatially periodic domain. A multi-point space-time joint distribution formula is obtained for general initial conditions. The formula involves contour integrals of Fredholm determinants with kernels acting on certain discrete spaces. For a class of initial conditions satisfying certain technical assumptions, we are able to derive large-time, large-period limit of the joint distribution, under the relaxation time scale $t=O(L^{3/2})$ when the height fluctuations are critically affected by the finite geometry. The assumptions are verified for the step and flat initial conditions. As a corollary we obtain the multi-point distribution of discrete time TASEP on the whole integer lattice $\mathbb{Z}$ by taking the period $L$ large enough so that the finite-time distribution is not affected by the boundary. The large time limit for the multi-time distribution of discrete time TASEP on $\mathbb{Z}$ is then obtained for the step initial condition.
\end{abstract}
\maketitle

\section{Introduction}
\label{sec:intro}
Models in the one-dimensional Kardar-Parisi-Zhang (KPZ) universality class are expected to have the same limiting height fluctuations under the $t:t^{2/3}:t^{1/3}$ KPZ scaling for temporal correlations, spatial correlations and height fluctuations. Describing the universal limiting fluctuating field and proving the convergence of concrete models to the universal limiting field have been the main goals in the field and attracted active research over the last two decades. The convergence to the limiting fluctuating field (at one-point or multi-point level) have been obtained for a large class of models (mostly exactly solvable), either for the whole space models \cite{baik1999,johansson2000, borodin2007, borodin2008parallel, tw08, tw09,matetski2017, dauvergne2018,johansson2019multi,liu2019}, or for half space models \cite{baikbarraquand2018, barraquand2018, barraquand2020}.

 More recently similar limit theorems were obtained for models on periodic domains. Such models are defined on a finite ring of size $L$ instead of on the infinite or semi-infinite lattices. The crucial interest here is to understand the effect of this finite geometry on the height fluctuations.   For periodic models the one-point marginals of limiting height fluctuations were obtained for classical step, flat and stationary initial conditions in \cite{baik2018relax, liu2018}, and independently in \cite{prolhac2016} under the so-called relaxation time scale $t=O(L^{3/2})$ when the height fluctuations are critically affected by the finite geometry. Further extensions including formulas for multi-point space-time joint distributions were subsequently obtained in \cite{baik2019multi} for step initial condition and \cite{baik2019general} for more general initial conditions. These limiting formulas are expected to be universal for all related models on periodic domains. So far all limiting distribution formulas were obtained as scaling limits of finite-time joint distributions of a single model in the KPZ universality class, namely the continuous time periodic totally asymmetric simple exclusion process (TASEP on a ring). 
 \medskip

The goal of this paper is to study another classical model in the KPZ universality class, the discrete time totally asymmetric exclusion process with parallel updates, on periodic domains. On the infinite lattice $\mathbb{Z}$, this model has been well studied.  The one-point marginal distribution was obtained in \cite{johansson2000} for the equivalent geometric last passage percolation model and  joint distributions of several locations at equal time was obtained in \cite{borodin2008parallel}. Recently the joint distributions along the time direction have also been studied, in \cite{johansson2019two} for two-time case and \cite{johansson2019multi} for general multi-time joint distributions. However on the periodic domain there are fewer results concerning height fluctuations (see \cite{povolotsky2007} for some results on transition probability and stationary distributions), which are the main focuses of this paper.  The main results of this paper are summarized as follows: 
\begin{enumerate}
	\item For general initial conditions we obtain a finite-time multi-point (in both space and time) joint distribution formula for discrete time parallel periodic TASEP. The formula consists of an $m$-fold ($m$ is the number of space-time points being considered in the joint distribution) contour integral with integrand involving a Fredholm determinant, where the Fredholm determinant has kernels acting on certain discrete sets related to the roots of some polynomials. 
	\smallskip
	\item Under the relaxation time scale $t=O(L^{3/2})$ and the $1:2:3$ KPZ scaling, we obtain large-time, large-period limits for the multi-point joint distributions under certain assumptions on the initial condition, which are verified for the step and flat cases. These limiting formulas agree with those obtained in \cite{baik2019multi, baik2019general}, thus providing an evidence that the height fluctuations for periodic models in the KPZ class are in fact universal. 
	\smallskip
	\item We also obtain the same type of multi-point joint distribution formula for discrete time  TASEP on the infinite lattice $\mathbb{Z}$, by relating it with the corresponding periodic model with large enough period so that the finite-time distribution is not affected by the boundary. The large time limit for the multi-time distribution formula under KPZ scaling is then derived for step initial condition, which agrees with the one in \cite{liu2019}. 
\end{enumerate}

Comparing to the previous work \cite{baik2019multi,baik2019general} and \cite{liu2019} on the multi-point distributions of continuous time TASEP, on either periodic domain or $\mathbb{Z}$, our work consists of formulas with similar structures but involves an extra parameter $p$ describing the hopping probability, which makes the algebraic properties a bit more complicated. In particular the polynomial whose roots are related to the kernels in the periodic formulas now depends on the extra parameter $p$. We point out that all the formulas for continuous time TASEP can be obtained from our formulas by rescaling the time and taking $p\to 0$.

On the infinite lattice $\mathbb{Z}$,  the multi-time distribution of discrete time TASEP (more precisely the equivalent geometric last passage percolation model) has been considered recently in \cite{johansson2019multi}, where a formula involving contour integrals of Fredholm determinants was obtained. We expect our formulas to agree with the formulas in \cite{johansson2019multi}, both in finite time and for the large time limit. Currently we are not able to prove the equivalence of the two types of formulas though they share a lot of similarities.
\medskip
\subsection*{Outline of the paper} In Section~\ref{sec: models and results}  we describe the discrete time parallel TASEP models and state the main results involving several multi-point joint distribution formulas, for both periodic domain and infinite lattice $\mathbb{Z}$, finite time and large time limit. From Sections~\ref{sec: transition} to Section~\ref{sec: summation identity} we derive the main finite-time algebraic formulas for multi-point distribution of discrete time parallel periodic TASEP and we regard them as the main technical novelties in this paper. In particular in Section~\ref{sec: transition} we prove a novel transition probability formula for  discrete time parallel TASEP on the periodic domain. The formula involves integral of determinants and is proved using coordinate Bethe ansatz. In Section~\ref{sec: multitime} we derive the finite-time multi-point distribution formula by performing a multiple sum over the transition probabilities. The  key ingredients are certain Cauchy-type summation identities over the eigenfunctions of the generator, which might be of independent interests so we discuss the proof in Section~\ref{sec: summation identity}. In Section~\ref{sec: relaxation time limit} and Section~\ref{sec: proof_limittheorem} we discuss the large time, large period asymptotics for the multi-point distribution under the relaxation time scale $t=O(L^{3/2})$. Finally in Section~\ref{sec: infinite} we derive the multi-time distribution for infinite discrete time parallel TASEP and perform large time asymptotics under KPZ scaling for the step initial condition. 
\bigskip
\subsection*{Acknowledgements} I would like to thank Jinho Baik for suggesting this problem and for useful discussions. I would also like to thank Zhipeng Liu and Mustazee Rahman for helpful discussions. I am grateful to the two anonymous referees for catching up several typos and to their valuable suggestions which help improve the quality of the paper. This work is supported in part through Jinho Baik's NSF grants DMS-1664531.

\section{Models and main results}
\label{sec: models and results}
Let $N<L$ be positive integers. We consider discrete time TASEP with parallel updates with $N$ particles on a spatially periodic domain of size $L$. It is convenient to view the dynamics as particles moving to the right on the integer lattice $\mathbb{Z}$ while periodicity forces particle configurations to be identical copies of each other every $L$ sites. More precisely this means that the occupation functions $\eta_j(t)$ which equals $1$ if there is a particle at site $j\in \mathbb{Z}$ at time $t$ and equals $0$ otherwise should satisfy 
\begin{equation}
	\eta_{j}(t) = \eta_{j+kL}(t),\quad \text{for all } j, k\in \mathbb{Z}\text{ and } t\in \mathbb{N}.
\end{equation}
We fix a single period of size $L$ and denote the locations of totally $N$ particles in this period at time $t$ as 
\begin{equation*}
x_1(t)>x_2(t)>\cdots>x_N(t).
\end{equation*}
Here $x_i(t)\in \mathbb{Z}$ for $1\leq i\leq N$ and we index the particles from right to left. The locations of all the particles then satisfy $x_{i+kN}(t)= x_i(t)-kL$ for $1\leq i\leq N$ and $k\in \mathbb{Z}$ so that we have 
\begin{equation*}
\cdots>x_N(t)+L=x_0(t)>x_1(t)>x_2(t)>\cdots>x_N(t)>x_{N+1}(t)=x_1(t)-L>\cdots.
\end{equation*}
Thus the natural configuration space for the particles should be
	\begin{equation}\label{eq: conf L}
		\conf_{N}^{(L)}:= \{\vec x=(x_1,\cdots,x_N)\in \mathbb{Z}^N: x_N+L>x_1>x_2>\cdots>x_N\}.
	\end{equation}
The discrete time parallel periodic TASEP with $N$ particles, period $L$ and hopping probability $0<p<1$, which we denote by $\mathrm{dpTASEP}(L,N,p)$, is the following Markovian dynamics on particle configurations $\vec x(t)\in \conf_{N}^{(L)}$: at each time step, each particle in a single period hops to its right neighbor site independently with probability $p=1-q$ provided that the site is empty, otherwise it stays at its current position. As a special case, the discrete time parallel TASEP on $\mathbb{Z}$ which we will denote by $\mathrm{dTASEP}(p)$ corresponds to particles following the same evolution rules with the period $L\to \infty$. In this case the configuration space for the first $N$ particles (from right to left, we assume the existence of a right-most particle) becomes 
\begin{equation}\label{eq: conf infty}
	\conf_{N}^{(\infty)}:= \{\vec x=(x_1,\cdots,x_N)\in \mathbb{Z}^N: x_1>x_2>\cdots>x_N\}.
\end{equation}

\begin{notation}
	Throughout the paper there will be several very similar formulas and quantities corresponding to either discrete time parallel TASEP on a periodic domain with size $L$ or on the infinite lattice $\mathbb{Z}$, as well as their scaling limits. To avoid confusion we will always follow the convention used in \eqref{eq: conf L} and \eqref{eq: conf infty} by putting superscripts $(L)$ to any quantities  related to periodic model in finite-time and $(\infty)$ to any quantities related to model on the infinite lattice $\mathbb{Z}$ in finite-time. On the other hand quantities describing the large-time limits will have superscripts $(\mathrm{per})$ and $(\mathrm{kpz})$ for periodic and infinite models, respectively.
\end{notation}

The main results of the paper consist of several formulas for multi-point joint distributions of discrete time parallel TASEP on periodic domain and on $\mathbb{Z}$ as well as their large time scaling limits. All of them are expressed as contour integrals of Fredholm determinants. The main difference is that the formulas for periodic models are expressed in terms of discrete sets related to roots of certain algebraic equations, while the Fredholm determinant appearing in formulas for discrete time TASEP on $\mathbb{Z}$ has kernels acting on continuous contours.
\subsection{Multi-point distribution formulas for \texorpdfstring{$\mathrm{dpTASEP}(L,N,p)$}{Lg}} The first theorem is a finite-time multi-point joint distribution formula for discrete time TASEP on a periodic domain. This is the starting point of all other results in this paper.
\begin{theorem}[Finite-time multi-point joint distribution for $\mathrm{dpTASEP}(L,N,p,\vec y)$]\label{thm: finite-time}
	Let $N<L$ be integers. Consider discrete time periodic parallel TASEP with hopping probability $0<p<1$, $N$ particles and $L$ sites per period (denoted by $\mathrm{dpTASEP}(L,N,p,\vec y)$). Here $\vec y\in \conf_{N}^{(L)}$ is the initial condition, i.e., $x_i(0)=y_i$ for $1\leq i\leq N$. Set $\rho=N/L$.
	Fix a positive integer $m$ and let $(k_i,t_i)$, $1\leq i\leq m$ be $m$ distinct points in $\mathbb{Z}\times \mathbb{Z}_{\geq 0}$ satisfying $0\leq t_1\leq\cdots \leq t_m $. Then for any integers $a_1,\cdots,a_m$,
	\begin{equation}\label{eq: multipoint general}
	\mathbb{P}_{\vec y}^{(L)}\left(\bigcap_{i=1}^m \{x_{k_i}(t_i)\geq a_i\}\right)= \oint\cdots\oint \mathscr{C}_{\vec y}^{(L)}(\vec z)\mathscr{D}_{\vec y}^{(L)}(\vec z)\frac{\dd z_1}{2\pi \ii z_1}\cdots\frac{\dd z_m}{2\pi \ii z_m}.
	\end{equation}
	Hre we denote $\vec z:=(z_1,\cdots,z_m)$. The contours are over nested circles centered at the origin: $0<|z_m|<\cdots<|z_1|<\mathbbm{r}_c$, where $\mathbbm{r}_c$ is a parameter depending on $p$ and $\rho$ that is defined in \eqref{eq: critical radius}. Here and in all the remaining results we suppress the dependence of the integrand on the parameters $a_i,k_i,t_i$ as well as the hopping probability $0<p<1$. The function $\mathscr{C}_{\vec y}^{(L)}(\vec z)$ is defined in Definition \ref{def: C general}. The function $\mathscr{D}_{\vec y}^{(L)}(\vec z)$ is a Fredholm determinant 
	\begin{equation}
	\mathscr{D}_{\vec y}^{(L)}(\vec z) = \det(I-\mathscr{K}_{1}^{(L)}\mathscr{K}_{\vec y}^{(L)}),
	\end{equation}
	where the operators $\mathscr{K}_{1}^{(L)}$ and $\mathscr{K}_{\vec y}^{(L)}$ are defined in Definition \ref{def:D general}.
\end{theorem}

\begin{remark}
	Theorem~\ref{thm: finite-time} generalizes Theorem~3.1 of \cite{baik2019general} (and also Theorem 4.6 of \cite{baik2019multi} for the special step initial condition) on the finite-time multi-point distribution of continuous time periodic TASEP. In fact their formulas can be obtained from our formula \eqref{eq: multipoint general} by taking $p=\epsilon$, $\hat{t}= \epsilon t$ and letting $\epsilon\to 0$. 
\end{remark}
Next we state the theorem on the large-time, large-period scaling limit of \eqref{eq: multipoint general} under the relaxation time scale $t=O(L^{3/2})$. To emphasize the dependence on the initial condition, we add the subscript ``$\mathrm{ic}$'' for the terms in the limit which depend on the initial conditions. For convenience we make the following choice of labelling: we assume that $x_1(0)\le 0 < x_{0}(0)$. This is equivalent with assuming that the initial condition satisfies $y_1\le 0< y_N+L$.

\begin{theorem}[Relaxation time limit] \label{thm:main}
	Consider a sequence of Markovian dynamics $\mathrm{dpTASEP}(L, N, \vec y(L))$ depending on the period $L$,  where $\rho=\rho_L=N/L$ stays in a compact subset of $(0,1)$ and $y_1\leq 0<y_N+L$.
	Suppose that the sequence of initial conditions $\vec y(L)$ satisfies certain assumptions (see Assumption \ref{def:asympstab}).  
	Fix a positive integer $m$ and let $\mathrm{p}_j = (\gamma_j,\tau_j)$ be $m$ points in the region
	\begin{equation*}
		\mathrm{R}:= [0,1] \times \mathbb{R}_{> 0}
	\end{equation*}
	satisfying
	\begin{equation*}
		0< \tau_1<\tau_2<\cdots<\tau_m.
	\end{equation*}
	Then for every fixed $\mathrm{x}_1,\cdots,\mathrm{x}_m\in\mathbb{R}$ and  parameters $a_i, k_i, t_i$, $1\leq i\leq m$ given by 
	\begin{equation} \label{eq:parameters}
		t_i = c_1\tau_i L^{3/2}+O(1),\quad a_i= c_2 t_i+\gamma_i L+O(1),\quad k_i=c_3 t_i+c_4\gamma_i L+c_5\mathrm{x}_i L^{1/2}+O(1),
	\end{equation}
	we have
	\begin{equation} \label{eq:main_theorem}
		\begin{aligned}
			\lim_{L\to\infty} \mathbb{P}^{(L)}_{\vec y} \left( \bigcap_{j=1}^m \left\{ x_{k_i}(t_i)\geq a_i	\right\}\right) = \Fic^{\mathrm{per}} (\mathrm{x}_1,\cdots, \mathrm{x}_m; \mathrm{p}_1,\cdots, \mathrm{p}_m).
		\end{aligned}
	\end{equation}
	Here the constants $c_i$ depend explicitly on particle density $0<\rho<1$ and hopping probability $0<p<1$ and are given by 
	\begin{equation}\label{eq: scaling constants}
		\begin{aligned}
			&c_1 = \frac{1}{p(1-p)}\frac{\nu^{5/2}}{\rho^{1/2}(1-\rho)^{1/2}}, \quad c_2= \frac{p(1-2\rho)}{\nu},\\
			&c_3 = \frac{2\rho^2\cdot p(1-p)}{\nu(1+\nu-2p\rho)},\quad c_4=-\rho,\quad c_5= -\rho^{1/2}(1-\rho)^{1/2}\nu^{1/2},
		\end{aligned}
	\end{equation}
	where we set $\nu:=\sqrt{1-4p\cdot\rho(1-\rho)}$ for convenience. We recall that in \eqref{eq:main_theorem} $\mathbb{P}^{(L)}$ denotes the probability associated to $\mathrm{dpTASEP}(L, N, \vec y(L))$. The function $\Fic^{\mathrm{per}}(\cdot)$ agrees with the one defined in Section 6.4 of \cite{baik2019general} as the relaxation time limit of multi-point distribution of continuous time periodic TASEP. We recall the definition of $\Fic^{\mathrm{per}}(\cdot)$ in Section \ref{sec: limitingdistribution} for completeness. The convergence is locally uniform in $\mathrm{x}_j, \tau_j$, and $\gamma_j$. 
\end{theorem}

\medskip
	\subsection{Multi-point distribution formulas for $\mathrm{dTASEP}(p)$} The finite-time multi-point joint distribution for discrete time parallel TASEP on $\mathbb{Z}$ can be obtained from Theorem~\ref{thm: finite-time} as a corollary. The rough idea is to take the period $L$ sufficiently large so that the finite-time formula is unaffected by the boundary. The precise procedure is more delicate and will be explained in Section~\ref{sec: infinite}. For simplicity we will only state result on step initial condition but we point out that in \cite{liu2019} similar multi-time formulas for continuous time TASEP was obtained for general initial conditions and we believe their arguments can also be adapted to our model. The main result is the following:
\begin{theorem} [Finite-time multi-point joint distribution for $\mathrm{dTASEP}(p)$]
	\label{thm:main1}
	Assume $\vec y=(-1,-2,\cdots,-N)$. Consider discrete time parallel TASEP on $\mathbb{Z}$ with initial particle locations $x_i(0)=y_i$ for $1\le i\le N$. Let $m\ge 1$ be a positive integer and $(k_1,t_1),\cdots, (k_m,t_m)$ be $m$ distinct points in $\{1,\cdots,N\}\times \mathbb{Z}_{\geq 0}$. Assume that $0\le t_1\le\cdots\le t_m$. Then, for any integers $a_1,\cdots,a_m$,
	\begin{equation}
		\label{eq:thm1}
		\mathbb{P}_{\mathrm{step}}^{(\infty)} \left( 
		\bigcap_{\ell=1}^m \left\{ x_{k_\ell} (t_\ell) \ge a_\ell 
		\right\}
		\right) = \oint\cdots\oint \left[\prod_{\ell=1}^{m-1}\frac{1}{1-\theta_\ell}\right] \mathcal{D}_{\mathrm{step}}^{(\infty)}(\theta_1,\cdots,\theta_{m-1}) \frac{\dd \theta_1}{2\pi \ii \theta_1}\cdots \frac{\dd \theta_{m-1}}{2\pi \ii \theta_{m-1}}
	\end{equation}
	where the integral contours are circles centered at the origin of radii less than $1$. The function $\mathcal{D}_{\mathrm{step}}^{(\infty)}(\theta_1,\cdots,\theta_{m-1})$ is defined in terms of a Fredholm determinant in Definition~\ref{def:operators_K1Y}.
\end{theorem}
\begin{remark}
	Under the standard coupling between discrete time parallel TASEP and geometric last passage percolation,  we have the following equality in distributions:
	\begin{equation}\label{eq: taseplpp}
		\mathbb{P}_{\mathrm{step}}^{(\infty)} \left( 
		\bigcap_{\ell=1}^m \left\{ x_{k_\ell} (t_\ell) \ge a_\ell 
		\right\}\right) = \mathbb{P}_{\mathrm{GLPP}}\left( 
		\bigcap_{\ell=1}^m \left\{ G(k_\ell,a_\ell+k_\ell)\leq t_{\ell} \right\}\right),
	\end{equation}
where $G(m,n)$ is the point-to-point last passage time from $(0,0)$ to $(m,n)$ in the usual geometric last passage percolation model with parameter $q=1-p$. For the right-hand side of \eqref{eq: taseplpp} a formula of similar form as \eqref{eq:thm1} was obtained in Theorem 2 of \cite{johansson2019multi}. We expect the two formulas to be equivalent but currently we are not able to prove it for $m\geq 2$.  
\end{remark}
\smallskip
Starting from the finite-time formula, we can then take large-time limit under the $1:2:3$ KPZ scaling. 
\begin{theorem}\label{thm:main2} Consider discrete time parallel TASEP on $\mathbb{Z}$ with step initial condition $x_i(0)= -i$. We rescale the parameters such that 
	\begin{equation}\label{eq: rescale infinite}
		t_{i}=\frac{2q^{1/4}}{1-q}\tau_i T+O(1),\quad  a=\gamma_iT^{2/3}+O(1),\quad k_i= \frac{q^{1/4}}{1+\sqrt{q}}\tau_iT-\frac{1}{2}\gamma_i T^{2/3}-\frac{q^{1/4}}{2} \mathrm{x}_i T^{1/3}+O(1),
	\end{equation}
	where $q=1-p$. Then 
	\begin{equation}\label{eq: limit thm infinite}
		\begin{aligned}
		\lim_{T\to \infty}\mathbb{P}_{\mathrm{step}}^{(\infty)} \left( 
		\bigcap_{\ell=1}^m \left\{ x_{k_\ell} (t_\ell) \ge a_\ell 
		\right\}\right)&= \mathbb{F}_{\mathrm{step}}^{\mathrm{kpz}}(\mathrm{x}_1,\cdots,\mathrm{x}_m; (\tau_1,\gamma_1),\cdots,(\tau_{m},\gamma_{m}))\\
		 &= \oint\cdots\oint \left[\prod_{\ell=1}^{m-1}\frac{1}{1-\theta_\ell}\right] \mathrm{D}_{\mathrm{step}}^{\mathrm{kpz}}(\theta_1,\cdots,\theta_{m-1}) \frac{\dd \theta_1}{2\pi \ii \theta_1}\cdots \frac{\dd \theta_{m-1}}{2\pi \ii \theta_{m-1}}.
		\end{aligned}
	\end{equation}
Here $\mathrm{D}_{\mathrm{step}}^{\mathrm{kpz}}(\theta_1,\cdots,\theta_{m-1})$ is a Fredholm determinant defined in Definition~\ref{def: limiting infinite}. We remark that the right-hand side of \eqref{eq: limit thm infinite}  agrees with the right-hand side of equation (2.19) in \cite{liu2019}, with the re-scaled parameters changed from $(\tau_i,\gamma_i, \mathrm{x}_i)$ to  $(\tau_i,2x_i, h_i)$. Our parameters are chosen to be consistent with the periodic cases. 
\end{theorem}
\medskip
\subsection{Bethe equations and Bethe roots}\label{sec: bethe roots}
For our analysis on the discrete time periodic TASEP, the following polynomial and its roots play essential role: 
\begin{definition}[Bethe roots] Given $z\in \mathbb{C}$ and $0<p<1$. Define the degree $L$ polynomial $q_z(w)$ by
	\begin{equation}\label{eq: Bethepoly}
	q_z(w):= w^N(1+w)^{L-N}-z^L(1+pw)^N.
	\end{equation}
	We call this polynomial the Bethe polynomial associated to $z$ and its roots the Bethe roots. We denote the set of all roots of the Bethe polynomial $q_z(w)$ by $\mathcal{S}_z$:
	\begin{equation}
	\mathcal{S}_z :=\{w\in \mathbb{C}: q_z(w)=0\}.
	\end{equation}
\end{definition} 
The Bethe root set $\mathcal{S}_z$ is contained in the level set $\{w\in \mathbb{C}: |w|^\rho|1+w|^{1-\rho}=|z|\cdot |1+pw|^{\rho}\}$, which is sometimes called a deformed Cassini oval. For given jumping rate $p$ and density $\rho$, there exists a critical radius $|z|=\mathbbm{r}_c$ for the level sets. For such critical $z$, the level set contains a self-intersecting point at $w=w_c$. Here 
\begin{equation} \label{eq: critical point}
	w_c := -\frac{2\rho}{1+\sqrt{1-4p\cdot\rho(1-\rho)}},
\end{equation}
and
\begin{equation}\label{eq: critical radius}
	\mathbbm{r}_c:= \left(\frac{-w_c}{1+pw_c}\right)^{\rho}(1+w_c)^{1-\rho}.
\end{equation}

It is not hard to check that for $|z|<\mathbbm{r}_c$, the level set consists of two disjoint contours while for $|z|>\mathbbm{r}_c$ the two contours merge to a single contour. For $z=\mathbbm{r}_c$ there is a self-intersecting point for the contour at $w=w_c$. See Figure \ref{fig: Bethe roots} for an illustration. We remark that the Bethe polynomials (and their roots) we are considering here are one-parameter generalizations of those considered in \cite{baik2018relax,baik2019multi,baik2019general} which is related to continuous time periodic TASEP and corresponds to $p= 0$ degenerations of \eqref{eq: Bethepoly}. 

\begin{figure}[h]
	\centering
	\includegraphics[width = 0.5\textwidth]{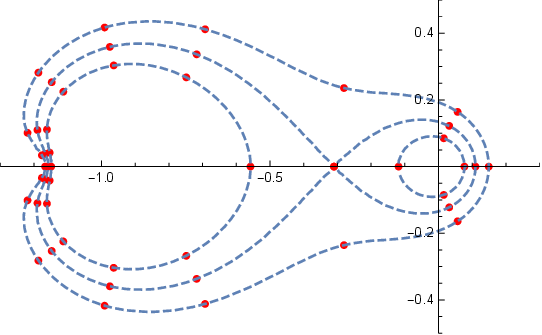}
	\caption{The solid dots are roots of $q_z(w)$ with $L=16$, $N=4$,  $p=5/6$ and $|z|= \frac{11}{10}\mathbbm{r}_c,\mathbbm{r}_c\text{ and }\frac{9}{10}\mathbbm{r}_c $ from outside to inside. The dashed lines are the corresponding level sets.}
	\label{fig: Bethe roots}
\end{figure}
\begin{definition}[Left and right Bethe roots]\label{def: left and right bethe}
	For $|z|<\mathbbm{r}_c$, we define the sets
	\begin{equation}
	\mathcal{L}_z:=\{w\in \mathcal{S}_z: \text{Re}(w)<w_c\},\quad \mathcal{R}_z:=\{w\in \mathcal{S}_z: \text{Re}(w)>w_c\},
	\end{equation}
where $w_c$ and $\mathbbm{r}_c$ are defined in \eqref{eq: critical point} and \eqref{eq: critical radius}.
	Then it is straightforward to check that $|\mathcal{L}_z|=L-N$ and $|\mathcal{R}_z|=N$. Roots in $\mathcal{L}_z$ and $\mathcal{R}_z$ are called left and right Bethe roots, respectively. We also define the left and right Bethe polynomials $q_{z,\LL}(w)$ and $q_{z,\RR}(w)$ as the monic polynomials with roots in $\mathcal{L}_z$ and $\mathcal{R}_z$:
	\begin{equation}
	q_{z,\LL}(w):= \prod_{u\in \mathcal{L}_z}(w-u),\quad q_{z,\RR}(w):= \prod_{v\in \mathcal{R}_z}(w-v).
	\end{equation}
	Then by definition we have 
	\begin{equation*}
	\mathcal{S}_z = \mathcal{L}_z\cup \mathcal{R}_z,\quad q_z(w)=q_{z,\LL}(w)q_{z,\RR}(w).
	\end{equation*}
\end{definition}

\subsection{A symmetric function related to initial conditions}
In our finite-time multi-point distribution formula \eqref{eq: multipoint general}, the quantities encoding information in the initial condition are all related to the following symmetric function:
\begin{definition}[Symmetric function]\label{def: symmetricfunction}
	Given $p\in \mathbb{C}$ and $\lambda = (\lambda_1,\cdots,\lambda_N)\in \mathbb{Z}^N$ with $\lambda_{1}\geq \cdots\geq \lambda_N$. We define 
	\begin{equation}
	\mathcal{F}_{\lambda}(w_1,\cdots,w_N;p) := \frac{\det\left[w_j^{N-i}(pw_j+1)^{i-1}(w_j+1)^{\lambda_{i}}\right]_{i,j=1}^N}{\det[w_j^{N-i}]_{i,j=1}^N}.
	\end{equation}
\end{definition}

\begin{remark}
	The symmetric function $\mathcal{F}_\lambda$ defined here is a one-parameter generalization of the Grothendieck-like symmetric function defined in equation (3.6) of \cite{baik2019general} which corresponds to the $p=0$ degeneration in our situation.
\end{remark}
The following two quantities related to $\mathcal{F}_{\lambda}$ encode the initial condition:
\begin{definition}[Global energy and characteristic function]\label{def: globalenergy}
	 For $\vec y\in \conf_{N}^{(L)}$, we set
	\begin{equation}\label{eq: ICtopartition}
	\lambda(\vec y) = (y_1+1,y_2+2,\cdots,y_N+N).
	\end{equation}
	For $|z|<\mathbbm{r}_c$, we define the global energy $\mathcal{E}_{\vec y}(z)$ by
	\begin{equation}\label{eq: global energy}
	\mathcal{E}_{\vec y}(z):= \mathcal{F}_{\lambda(\vec y)}(\mathcal{R}_z;p).
	\end{equation} 
	When  $\mathcal{E}_{\vec y}(z)\neq 0$, we define the characteristic function $\chi_{\vec y}(v,u;z)$ for a left Bethe root $u$ and a right Bethe root $v$ by
	\begin{equation}\label{eq: characteristic function}
	\chi_{\vec y}(v,u;z):= \frac{\mathcal{F}_{\lambda(\vec y)}(\mathcal{R}_z\cup\{u\}\backslash\{v\};p)}{\mathcal{F}_{\lambda(\vec y)}(\mathcal{R}_z;p)}, \quad \text{for } u\in \mathcal{L}_z \text{ and } v\in \mathcal{R}_z.
	\end{equation}
\end{definition}

\begin{remark}\label{rmk: step}
	A straightforward calculation shows that for step initial condition $\vec y=(-1,\cdots,-N)$, we have $\mathcal{F}_{\lambda(\vec y)}\equiv 1$. Hence the global energy function and characteristic function are both constant $1$ for the step initial condition. In general since the roots of the Bethe polynomial $q_z(w)$ depend analytically on $z$, the function $\mathcal{E}_{\vec y}(z)$ is analytic for $|z|<\mathbbm{r}_c$. Furthermore $\mathcal{E}_{\vec y}(z)$ can not be constant zero. In fact $\mathcal{E}_{\vec y}(0)=1$ since when $|z|\to 0$ all the right Bethe roots converge to $0$. As a consequence $\mathcal{F}_{\lambda(\vec y)}(\mathcal{R}_z;p)$ is nonzero for all but finitely many $z$ in any compact subset of $\{|z|<\mathbbm{r}_c\}$, which means $\chi_{\vec y}(v,u;z)$ is a well-defined meromorphic function in $z$ on $\{|z|< \mathbbm{r}_c\}$ for fixed $u,v$.
\end{remark}
\medskip
\subsection{Definition of \texorpdfstring{$\mathscr{C}_{\vec y}^{(L)}(\vec z)$}{Lg} and \texorpdfstring{$\mathscr{D}_{\vec y}^{(L)}(\vec z)$}{Lg}}\label{sec: def of C and D}
\begin{definition}[Definition of $\mathscr{C}_{\vec y}^{(L)}(\vec z)$]\label{def: C general}
	With the global energy function $\mathcal{E}_{\vec y}(z)$ defined in \eqref{eq: global energy}, we define 
	\begin{equation}\label{eq: C(z) general}
	\mathscr{C}_{\vec y}^{(L)}(\vec z) := \mathcal{E}_{\vec y}(z_1)\mathscr{C}_{\text{step}}^{(L)}(\vec z),
	\end{equation}
	where 
	\begin{equation}
	\begin{aligned}
	\mathscr{C}_{\text{step}}^{(L)}(\vec z) &:= \left[\prod_{\ell=1}^{m}\frac{E_{\ell}(z_{\ell})}{E_{\ell-1}(z_{\ell})}\right]\left[\prod_{\ell=1}^{m}\frac{\prod_{u\in \mathcal{L}_{z_\ell}}(-u)^N\prod_{v\in \mathcal{R}_{z_{\ell}}}(v+1)^{L-N}}{\Delta(\mathcal{R}_{z_{\ell}};\mathcal{L}_{z_{\ell}})}\right]\\
	&\cdot \left[\prod_{\ell=2}^{m}\frac{z_{\ell-1}^L}{z_{\ell-1}^L-z_{\ell}^L}\right]\left[\prod_{\ell=2}^{m}\frac{\Delta(\mathcal{R}_{z_{\ell}};\mathcal{L}_{z_{\ell-1}})}{\prod_{u\in \mathcal{L}_{z_{\ell-1}}}(-u)^N\prod_{v\in \mathcal{R}_{z_{\ell}}}(v+1)^{L-N}}\right].
	\end{aligned}
	\end{equation}
	Here 
	\begin{equation}
	E_{\ell}(z):= \prod_{u\in \mathcal{L}_z} (-u)^{-k_\ell}\prod_{v\in \mathcal{R}_z}(v+1)^{-a_\ell-k_\ell}(pv+1)^{t_\ell-k_\ell},
	\end{equation}
	for $1\leq i\leq m$ and $E_0(z):=1$.
\end{definition}

$\mathscr{D}_{\vec y}^{(L)}(\vec z)$ is a Fredholm determinant with kernel acting on certain $\ell^2$ space over discrete sets related to the Bethe roots. More precisely for $m$ distinct complex numbers $z_i$ satisfying $|z_i|<\mathbbm{r}_c$, we define the discrete sets 
\begin{equation} \label{eq:def_scrS1}
\mathscr{S}_1:= \mathcal{L}_{z_1} \cup \mathcal{R}_{z_2}\cup\mathcal{L}_{z_3}\cup\cdots\cup 
\begin{cases} 
\mathcal{L}_{z_m}, &\text{if $m$ is odd},\\
\mathcal{R}_{z_m}, &\text{if $m$ is even}, 
\end{cases}
\end{equation}

and
\begin{equation} \label{eq:def_scrS2}
\mathscr{S}_2:= \mathcal{R}_{z_1} \cup \mathcal{L}_{z_2}\cup\mathcal{R}_{z_3}\cup\cdots\cup 
\begin{cases} 
\mathcal{R}_{z_m}, &\text{if $m$ is odd},\\
\mathcal{L}_{z_m}, &\text{if $m$ is even}.
\end{cases}
\end{equation}

\begin{definition}[Definition of $\mathscr{D}_{\vec y}^{(L)}(\vec z)$]\label{def:D general}
	Let $0<|z_m|<\cdots< |z_1|<\mathbbm{r}_c$. 
	Assume $\mathcal{E}_{\vec y}(z_1)\ne 0$ so that $\chi_{\vec y}(v,u;z_1)$ is well defined.
	Define
	\begin{equation} \label{eq:def_D0}
	\mathscr{D}_{\vec y}^{(L)}(\vec z) = \det(I -  \mathscr{K}_{1}^{(L)}\mathscr{K}_{\vec y}^{(L)}) , 
	\end{equation}
	where $\mathscr{K}_{1}^{(L)} : \ell^2(\mathscr{S}_2) \to \ell^2(\mathscr{S}_1)$ and $\mathscr{K}_{\vec y}^{(L)}:\ell^2(\mathscr{S}_1) \to \ell^2(\mathscr{S}_2)$ have kernels given by
	\begin{equation}
	\mathscr{K}_{1}^{(L)}(w, w') := \left( \delta_i(j) +\delta_i( j + (-1)^i )\right) 
	\frac{ J(w) f_i(w) (H_{z_i}(w))^2 }
	{ H_{z_{i-(-1)^i}}(w) H_{z_{j-(-1)^j}}(w') (w-w')} Q_1(j) 
	\end{equation}
	and
	\begin{equation}
	\mathscr{K}_{\vec y}^{(L)}(w', w):= 
		\left( \delta_j(i) +\delta_j( i - (-1)^j )\right)  \frac{ J(w') f_j(w') (H_{z_j}(w'))^2 }
		{ H_{z_{j + (-1)^j}}(w') H_{z_{i + (-1)^i}}(w) (w'-w)} Q_2(i)\Lambda(i,w,w') ,
	\end{equation}
	for 
	\begin{equation*}
	w \in (\mathcal{L}_{z_i} \cup \mathcal{R}_{z_i}) \cap \mathscr{S}_1 \quad \text{and} \quad 
	w'\in (\mathcal{L}_{z_j} \cup \mathcal{R}_{z_j}) \cap \mathscr{S}_2
	\end{equation*}
	with $1\le i,j\le m$. Here 
	\begin{equation}
		\Lambda(i,w,w') = \begin{cases}
		 	\chi_{\vec y}(w,w';z_1)\quad &\text{for }i=1,\\
		 	1\quad & \text{for }2\leq i\leq m.
		 \end{cases}
	\end{equation}
	The functions $J(w)$ and $Q_i(j)$ are defined as follows:
	\begin{equation}
	J(w):= \frac{w(w+1)(pw+1)}{N+Lw+p(L-N)w^2},
	\end{equation}
	and 
	\begin{equation}
	Q_1(j):= 1-\left(\frac{z_{j-(-1)^j}}{z_j}\right)^L,\quad Q_2(j):= 1-\left(\frac{z_{j+(-1)^j}}{z_j}\right)^L,
	\end{equation}
	for $j=1,\cdots,m$. Here we set $z_{m+1}:=0$ for convenience.
	
	To define $H_z(w)$ and $f_i(w)$ we recall the definition of left and right Bethe polynomials and Bethe roots discussed in Section \ref{sec: bethe roots}. With the notation there we set 
	\begin{equation}
	H_z(w):=\begin{cases}
	\frac{q_{z,\RR}(w)}{w^N}\quad &\text{for }\text{Re}(w)<w_c,\\
	\frac{q_{z,\LL}(w)}{(w+1)^{L-N}}\quad &\text{for }\text{Re}(w)>w_c.
	\end{cases}
	\end{equation}
	Finally the functions $f_i(w)$ encodes the information of the parameters $a_i,k_i,t_i$'s:
	\begin{equation}\label{eq: f_i}
	f_\ell(w):= \begin{cases}
	\frac{F_{\ell}(w)}{F_{\ell-1}(w)}\quad \text{for } \text{Re}(w)<w_c,\\
	\frac{F_{\ell-1}(w)}{F_{\ell}(w)}\quad \text{for } \text{Re}(w)>w_c,
	\end{cases}
	\end{equation}
	where 
	\begin{equation}
	F_{\ell}(w):= w^{k_\ell}(1+w)^{-a_{\ell}-k_{\ell}}(1+pw)^{t_{\ell}-k_{\ell}},
	\end{equation}
	for $1\leq \ell\leq m$ and $F_0(w):=1$.
\end{definition}
\medskip
\subsection{Definition of \texorpdfstring{$\mathcal{D}_{\mathrm{step}}^{(\infty)}(\theta_1,\cdots,\theta_{m-1})$}{Lg}}
\label{sec:Fredholm_representation}

The function  $\mathcal{D}_{\mathrm{step}}^{(\infty)}(\theta_1,\cdots,\theta_{m-1})$ has a similar structure as $\mathscr{D}_{\vec y}^{(L)}(\vec z)$ defined in Definition~\ref{def:D general}. It is also a Fredholm determinant of the form $\det(I-\mathcal{K}_1^{(\infty)}\mathcal{K}_{\mathrm{step}}^{(\infty)})$.  The kernels also share certain similarities and we will illustrate the relationship between the two in Section~\ref{sec: infinite}. 
\subsubsection{Spaces of the operators}
\label{sec:spaces_operators}

Instead of discrete sets, the operators appearing in  $\mathcal{D}_{\mathrm{step}}^{(\infty)}(\theta_1,\cdots,\theta_{m-1})$ are defined on specific spaces of nested contours defined as follows: Let $\Omega_\LL$ and $\Omega_\RR$ be two simply connected regions on the complex plane such that 
\begin{enumerate}[(1)]
\item $\Omega_\LL\subset \{w\in \mathbb{C}: \mathrm{Re}(w)<-\frac{1}{1+\sqrt{1-p}}\}$ and  $\Omega_\LL$ contains the point $-1$ but not $-1/p$.
\item $\Omega_\RR\subset \{w\in \mathbb{C}: \mathrm{Re}(w)>-\frac{1}{1+\sqrt{1-p}}\}$
 and  $\Omega_\RR$ contains the point $0$.
\end{enumerate}
Let $\Sigma_{m,\LL}^{+},\cdots,\Sigma_{2,\LL}^{+}$, $\Sigma_{1,\LL}$, $\Sigma_{2,\LL}^{-},\cdots,\Sigma_{m,\LL}^-$ be $2m-1$ nested simple closed contours, from outside to inside, in $\Omega_\LL$ enclosing the point $-1$ but not $-1/p$. Similarly,
$\Sigma_{m,\RR}^{+},\cdots,\Sigma_{2,\RR}^{+}$, $\Sigma_{1,\RR}$, $\Sigma_{2,\RR}^{-},\cdots,\Sigma_{m,\RR}^-$ be $2m-1$ nested simple closed contours, from outside to inside, in $\Omega_\RR$ enclosing the point $0$.  We define
\begin{equation*}
	\Sigma_{\ell,\LL}:=\Sigma_{\ell,\LL}^+\cup \Sigma_{\ell,\LL}^-, \qquad \Sigma_{\ell,\RR}:=\Sigma_{\ell,\RR}^+\cup \Sigma_{\ell,\RR}^-,\qquad \Sigma_{\ell}= \Sigma_{\ell,\LL}\cup \Sigma_{\ell,\RR},\qquad\ell=2,\cdots,m,
\end{equation*}
and
\begin{equation*}
	\mathcal{S}_1:= \Sigma_{1,\LL} \cup \Sigma_{2,\RR} \cup \cdots \cup \begin{dcases}
		\Sigma_{m,\LL}, & \text{ if $m$ is odd},\\
		\Sigma_{m,\RR}, & \text{ if $m$ is even},
	\end{dcases}
\end{equation*}
and 
\begin{equation*}
	\mathcal{S}_2:= \Sigma_{1,\RR} \cup \Sigma_{2,\LL} \cup \cdots \cup \begin{dcases}
		\Sigma_{m,\RR}, & \text{ if $m$ is odd},\\
		\Sigma_{m,\LL}, & \text{ if $m$ is even}.
	\end{dcases}
\end{equation*}

\subsubsection{Operators $\mathcal{K}_1^{(\infty)}$ and $\mathcal{K}_{\mathrm{step}}^{(\infty)}$}
\label{sec:def_operators}

Now we introduce the operators $\mathcal{K}_1^{(\infty)}$ and $\mathcal{K}_{\mathrm{step}}^{(\infty)}$ to define $\mathcal{D}_{\mathrm{step}}^{(\infty)}(\theta_1,\cdots,\theta_{m-1})$ in Theorem~\ref{thm:main1}.

\begin{definition}
	\label{def:operators_K1Y}
	We define
	\begin{equation*}
		\mathcal{D}_{\mathrm{step}}^{(\infty)}(\theta_1,\cdots,\theta_{m-1})=\det\left( I - \mathcal{K}_1^{(\infty)} \mathcal{K}_{\mathrm{step}}^{(\infty)} \right),
	\end{equation*}
	where two operators
	\begin{equation*}
		\mathcal{K}_1^{(\infty)}: L^2(\mathcal{S}_2) \to L^2(\mathcal{S}_1),\qquad \mathcal{K}_{\mathrm{step}}^{(\infty)}: L^2(\mathcal{S}_1)\to L^2(\mathcal{S}_2)
	\end{equation*}
	are defined by the kernels
	\begin{equation}
		\label{eq:K1}
		\mathcal{K}_1^{(\infty)}(w,w'):= \left(\delta_i(j) + \delta_i( j+ (-1)^i)\right) \frac{ f_i(w) }{w-w'} Q_1^{(\infty)}(j)P_j(w')\quad \text{for }w'\in \Sigma_j\cap\mathcal{S}_2,\ w\in \Sigma_i\cap\mathcal{S}_1,
	\end{equation}
	and
	\begin{equation}
		\label{eq:KY}
		\mathcal{K}_{\mathrm{step}}^{(\infty)}(w',w):= 
			\left(\delta_j (i) + \delta_j(i - (-1)^j)\right) \frac{ f_j(w') }{w'-w} Q_2^{(\infty)}(i)P_i(w)\quad \text{for }w\in \Sigma_i\cap\mathcal{S}_1,\ w'\in \Sigma_j\cap\mathcal{S}_2,
	\end{equation}
    where $1\leq i,j\leq m$. Here the function $f_i(w)$ is the same as the one appearing in the kernels for periodic case defined in \eqref{eq: f_i} and we define 
	\begin{equation}\label{eq: Q1}
		Q_1^{(\infty)}(j):= \begin{cases}
			1-\theta_j,\quad &\text{if $j$ is odd and $j<m$},\\
			1-\theta_{j-1}^{-1},\quad &\text{if $j$ is even},\\
			1,&\text{if $j=m$ is even or $j=1$}.
		\end{cases}
	\end{equation}
and 
	\begin{equation}\label{eq: Q2}
	Q_2^{(\infty)}(j):= \begin{cases}
		1-\theta_i,\quad &\text{if $j$ is even and $j<m$},\\
		1-\theta_{j-1}^{-1},\quad &\text{if $j$ is odd and $j>1$},\\
		1,&\text{if $j=m$ is odd}.
	\end{cases}
\end{equation}
	We also define 
	\begin{equation}
		P_j(w):= \begin{cases}
			\frac{1}{1-\theta_{j-1}},\quad & w\in \Sigma_{j,\LL}^{+}\cup\Sigma_{j,\RR}^{+}, \quad j=2,\cdots,m\\
			\frac{-\theta_{j-1}}{1-\theta_{j-1}},\quad & w\in \Sigma_{j,\LL}^{-}\cup\Sigma_{j,\RR}^{-}, \quad j=2,\cdots,m\\
			1,\quad & w\in \Sigma_{1,\LL}\cup\Sigma_{1,\RR}. 
		\end{cases}
	\end{equation}
\end{definition}

\bigskip
\section{Transition probability}
\label{sec: transition}
In this section we give an explicit integral formula for the transition probability of discrete time parallel TASEP in the configuration space $\conf_{N}^{(L)}$. This is the starting point for deriving the finite-time joint distribution formulas.
\begin{proposition}\label{prop:transition probability}
	Given two particle configurations $\vec x=(x_1,\cdots,x_N),\  \vec y=(y_1,\cdots,y_N) \in \conf_{N}^{(L)}$. Let $P_t(\vec y\to \vec x)$ be the transition probability of observing configuration $\vec x$ at time $t$ under the discrete time periodic TASEP dynamics with initial configuration $\vec y$. With the convention $x_0:= x_N+L$ we have
	\begin{equation}\label{eq:transition probability}
	     P_t(\vec y\to \vec x)= \prod_{i=1}^N (1-p\1_{x_{i-1}-x_i=1})\oint_{\Gamma} \frac{\dd z}{2\pi \ii z} \det\left[\sum_{w\in \mathcal{S}_z} F_{i,j}(w;\vec x,\vec y,t)J(w)\right]_{i,j=1}^N.
	\end{equation}
	Here $\Gamma$ is any simple closed contour with $0$ inside and $\mathcal{S}_z$ consists of all the roots of the degree $L$ polynomial $q_z(w):= w^N(w+1)^{L-N}-z^L(1+pw)^N$, i.e.,
	\begin{equation}
		\mathcal{S}_z:=\{w\in \mathbb{C}: w^N(w+1)^{L-N}-z^L(1+pw)^N=0\}.
	\end{equation}
	The functions $F_{i,j}(w;t)$ and $J(w)$ are given by
	\begin{equation}
		F_{i,j}(w;\vec x,\vec y,t) = w^{j-i}(w+1)^{-x_{N-i+1}+y_{N-j+1}+i-j-1}(1+pw)^{t+i-j},\quad 1\leq i,j\leq N,
	\end{equation}
	and 
	\begin{equation}
		J(w):= \frac{w(w+1)(1+pw)}{N+Lw+p(L-N)w^2}.
	\end{equation}

\end{proposition}

\begin{remark}
	We remark that a different formula for the transition probability of discrete time parallel TASEP on a ring was obtained in \cite{povolotsky2007}. The key difference is that the formula in \cite{povolotsky2007} is expressed as an infinite sum of determinants while our formula is a single contour integral of determinants. The main reason for this is that in \cite{povolotsky2007} the authors do not distinguish particle configurations which differ by a translation of an integer multiple of the period, so in our language their transition probability really is 
	\begin{equation*}
		P_t([\vec y]\to [\vec x])=\sum_{k\in \mathbb{Z}}P_t(\vec y\to \vec x+(kL,kL,\cdots,kL)).
	\end{equation*}
	We believe our formula is simpler and more suitable for deriving finite-time joint distributions.
\end{remark}
\begin{remark}
	If we take the continuous time limit by setting $p=\epsilon$, $t= T/\epsilon$ and send $\epsilon \to 0$, the dynamics then becomes continuous time periodic TASEP considered in \cite{baik2018relax} and our formula \eqref{eq:transition probability} reduces to equation (5.4) in \cite{baik2018relax} for the transition probability of continuous time periodic TASEP up to an index reversing.
\end{remark}

We first list a few elementary properties of the transition probability formula \eqref{eq:transition probability} before discussing the proof of Proposition \ref{prop:transition probability}.
\begin{proposition}[Properties of the transition probability formula] \label{properties of the transition probability}
	The right hand side of \eqref{eq:transition probability} satisfies the following properties:
	\begin{enumerate}[(i)]
		\item The right hand side of \eqref{eq:transition probability} can also be written as:
	    \begin{equation}\label{eq:transition probability2}
		\begin{aligned}
		 \prod_{i=1}^N (1-p\1_{x_{i-1}-x_i=1})\oint_{\Gamma} \frac{\dd z}{2\pi \ii z} \det\left[\oint_{\Gamma_{\mathcal{S}_z}}\frac{\dd w}{2\pi \ii} F_{i,j}(w;\vec x,\vec y,t)\left(1+\frac{z^L}{\hat{q}_z(w)}\right) \right]_{i,j=1}^N,
		\end{aligned}
		\end{equation}
		where $\hat{q}_{z}(w)=q_z(w)(1+pw)^{-N}=w^N(1+pw)^{-N}(1+w)^{L-N}-z^L$ has the same roots as $q_z(w)$ for any $|z|>0$ and $\Gamma_{\mathcal{S}_z}$ is any simple closed contour with all the roots in $\mathcal{S}_z$ inside and $-1$ and $-1/p$ outside.
		\item 
		The outer integral with respect to $z$ in \eqref{eq:transition probability} (and also \eqref{eq:transition probability2}) does not depend on the contour $\Gamma$.
		\item Assume further that $L\geq x_1-y_N+2$. Then the right-hand side of \eqref{eq:transition probability} can be further written as 
		\begin{equation}\label{eq: transition probability large}
			\prod_{i=1}^N (1-p\1_{x_{i-1}-x_i=1}) \det\left[\oint_{\Gamma_{0,-1}}\frac{\dd w}{2\pi \ii} F_{i,j}(w;\vec x,\vec y,t) \right]_{i,j=1}^N,
		\end{equation}
		where $\Gamma_{0,-1}$ is any simple closed contour enclosing $0$ and $-1$ as the only possible poles for the integrand. Note that \eqref{eq: transition probability large} agrees with the transition probability for discrete time parallel TASEP on $\mathbb{Z}$, see for example equation (3.21) of \cite{borodin2008parallel}.
	    \item \label{prop:cyclic-shift invariance}
	    The right-hand side of \eqref{eq:transition probability} is invariant under cyclic translation and re-indexing. Namely, for any fixed $1\leq k\leq N$, set $\vec x':= (x_k,x_{k+1},\cdots,x_N,x_1-L,x_2-L,\cdots,x_{k-1}-L)$ and 
	    $\vec y':= (y_k,y_{k+1},\cdots,y_N,y_1-L,y_2-L,\cdots,y_{k-1}-L)$. Then the right-hand side of \eqref{eq:transition probability} is invariant if we replace $\vec x$ and $\vec y$ by $\vec x'$ and $\vec y'$.
	\end{enumerate} 
\end{proposition}

\begin{proof}
	\begin{enumerate}[(i)]
		\item It is easy to check that $J(w)=\frac{\hat{q}_z(w)+z^L}{\frac{\dd}{\dd w}\hat{q}_z(w)}$. Hence by the residue theorem we have 
		\begin{equation*}
			\sum_{w\in \mathcal{S}_z} F_{i,j}(w;\vec x,\vec y,t)J(w) =\oint_{\Gamma_{\R_z}} \frac{\dd w}{2\pi \ii} F_{i,j}(w;\vec x,\vec y,t)\frac{\hat{q}_z(w)+z^L}{\hat{q}_z(w)}.
		\end{equation*}
		Note that $F_{i,j}(w,\vec x,\vec y,t)\frac{\hat{q}_z(w)+z^L}{\hat{q}_z(w)}$ is analytic at $w=0$ for any $1\leq i,j\leq N$ and the only possible poles of the integrand besides $\mathcal{S}_z$ are $w=-1$ and $w=-1/p$.
		\smallskip
		\item Choose $R>0$ large enough and $\epsilon>0$ small enough so that all the roots in $\mathcal{S}_z$ are inside the region  $\{\epsilon <|w+1|<R\}\backslash \{|1+pw|\leq \epsilon\}$. Then 
		\begin{equation}\label{eq: integral rep for entry}
		\sum_{w\in \mathcal{S}_z} F_{i,j}(w;\vec x,\vec y,t)J(w) =\left(\oint_{|w+1|=R}-\oint_{|w+1|=\epsilon}-\oint_{|pw+1|=\epsilon}\right)\frac{\dd w}{2\pi \ii} F_{i,j}(w;\vec x,\vec y,t)\frac{\hat{q}_z(w)+z^L}{\hat{q}_z(w)}.
		\end{equation}
		Since $R$ and $\epsilon$ can be arbitrarily large or small, the right hand side of \eqref{eq: integral rep for entry} as a function in $z$ is analytic for any $|z|>0$. Hence the integral with respect to $z$ in \eqref{eq:transition probability} is independent of $\Gamma$ since the integrand has an analytic continuation to $\{|z|>0\}$.
		\smallskip
		\item For $L>x_1-y_N+2$, we have $-x_{N-i+1}+y_{N-j+1}+i-j+L-N-1\geq 0$ for all $1\leq i,j\leq N$. Hence for any $|z|>0$, the integrand $F_{i,j}(w;\vec x,\vec y,t)\frac{\hat{q}_z(w)+z^L}{\hat{q}_z(w)}$ in \eqref{eq: integral rep for entry} is analytic at $w=-1$. This implies that
		\begin{equation}\label{eq: integral rep for entry 2}
			\sum_{w\in \mathcal{S}_z} F_{i,j}(w;\vec x,\vec y,t)J(w) =\left(\oint_{|w+1|=R}-\oint_{|pw+1|=\epsilon}\right)\frac{\dd w}{2\pi \ii} F_{i,j}(w;\vec x,\vec y,t)\frac{\hat{q}_z(w)+z^L}{\hat{q}_z(w)}.
		\end{equation}
		Now for fixed $R$ large enough and $\epsilon$ small enough, the right hand side of \eqref{eq: integral rep for entry 2} is an analytic function in $z$ for $|z|$ sufficiently small such that all the roots of $\hat{q}_z$ are in the region $\{|w+1|\leq R\}\backslash\{|pw+1|\leq \epsilon\}$. Now by the residue theorem the outer contour integral with respect to $z$ in \eqref{eq:transition probability2} equals the integrand evaluated at $z=0$, which equals
		\begin{equation*}
			\det\left[\left(\oint_{|w+1|=R}-\oint_{|pw+1|=\epsilon}\right)\frac{\dd w}{2\pi \ii} F_{i,j}(w;\vec x,\vec y,t)\right]_{i,j=1}^N.
		\end{equation*}
		\smallskip
		
		\item We remark that this property can be easily understood if we use the probabilistic interpretation since $\vec x$ and $\vec x'$ (and also $\vec y$ and $\vec y'$) actually represent the same particle configuration on $\mathbb{Z}$ (we just use particles in different period as representatives), hence the transition probability between $\vec y$ and $\vec x$ and $\vec y'$ and $\vec x'$ should be the same. Here however we can not directly use this since we have not proven \eqref{eq:transition probability}. In fact, we will need this fact in our proof of \eqref{eq:transition probability} so we give an independent algebraic proof here.
		It suffices to assume $k=2$. Then $\vec x'=(x_2,\cdots,x_N,x_1-L)$, $\vec y'=(y_2,\cdots,y_N,y_1-L)$. Clearly we have 
			$\prod_{i=1}^{N}(1-p\1_{x_{i-1}-x_i}=1) = \prod_{i=1}^{N}(1-p\1_{x'_{i-1}-x'_i}=1)$.
		So it suffices to show 
		\begin{equation*}
			\det\left[\sum_{w\in \mathcal{S}_z} F_{i,j}(w;\vec x,\vec y,t)J(w)\right]_{i,j=1}^N = \det\left[\sum_{w\in \mathcal{S}_z} \tilde{F}_{i,j}(w;\vec x,\vec y,t)J(w)\right]_{i,j=1}^N,
		\end{equation*}
		where $\tilde{F}_{i,j}(w;\vec x,\vec y,t)= w^{j-i}(w+1)^{-x'_{N-i+1}+y'_{N-j+1}+i-j-1}(1+pw)^{t+i-j}$. By multi-linearity we have 
		\begin{align*}
			&\det\left[\sum_{w\in \mathcal{S}_z} F_{i,j}(w;\vec x,\vec y,t)J(w)\right]_{i,j=1}^{N} = \sum_{w_1,\cdots,w_N\in \mathcal{S}_z} \det\left[ F_{i,j}(w_i;\vec x,\vec y,t)J(w_i)\right]_{i,j=1}^{N}\\
			&= \sum_{w_1,\cdots,w_N\in \mathcal{S}_z} \sum_{\sigma\in S_N} \sgn(\sigma)\prod_{i=1}^N w_i^{\sigma(i)-i}(w_i+1)^{-x_{N-i+1}+y_{N-\sigma(i)+1}+i-\sigma(i)-1}(1+pw_i)^{t+i-\sigma(i)}\\
			&= \sum_{w_1,\cdots,w_N\in \mathcal{S}_z} \sum_{\sigma\in S_N} \sgn(\sigma)\prod_{i=1}^N w_i^{\sigma(i)-i}(w_i+1)^{-x'_{N-i}+y'_{N-\sigma(i)}+i-\sigma(i)-1}(1+pw_i)^{t+i-\sigma(i)}.
 		\end{align*}
 		Here again $x'_0=x'_N+L=x_1-L+L=x_1$. Now we fix a permutation $\tau:=(N\cdots 21)\in S_N$ and set $\tilde{w}_i:=w_{\tau(i)}$ for $1\leq i\leq N$ and $\tilde{\sigma}:=\tau^{-1}\sigma\tau$. Then the last line in the above equation equals
 		\begin{equation}\label{eq:expansion of the determinant}
 			\sum_{\tilde{w}_1,\cdots,\tilde{w}_N\in \mathcal{S}_z} \sum_{\tilde{\sigma}\in S_N} \sgn(\tilde{\sigma})\prod_{i=1}^N \tilde{w}_i^{\tau\tilde{\sigma}(i)-\tau(i)}(\tilde{w}_i+1)^{-x'_{N-\tau(i)}+y'_{N-\tau\tilde{\sigma}(i)}+\tau(i)-\tau\tilde{\sigma}(i)-1}(1+p\tilde{w}_i)^{t+\tau(i)-\tau\tilde{\sigma}(i)}.
 		\end{equation}
 		We claim that for any fixed $\tilde{\sigma}\in S_N$,
 		 \begin{align*}
 		&\prod_{i=1}^N \tilde{w}_i^{\tau\tilde{\sigma}(i)-\tau(i)}(\tilde{w}_i+1)^{-x'_{N-\tau(i)}+y'_{N-\tau\tilde{\sigma}(i)}+\tau(i)-\tau\tilde{\sigma}(i)-1}(1+p\tilde{w}_i)^{t+\tau(i)-\tau\tilde{\sigma}(i)}\\
 		&=\prod_{i=1}^N \tilde{w}_i^{\tilde{\sigma}(i)-i}(\tilde{w}_i+1)^{-x'_{N-i+1}+y'_{N-\tilde{\sigma}(i)+1}+i-\tilde{\sigma}(i)-1}(1+p\tilde{w}_i)^{t+i-\tilde{\sigma}(i)}.
 		\end{align*}
 		This then implies 
 		\begin{align*}
 			&\det\left[\sum_{w\in \mathcal{S}_z} F_{i,j}(w;\vec x,\vec y,t)J(w)\right]_{i,j=1}^{N} \\
 			&= \sum_{\tilde{w}_1,\cdots,\tilde{w}_N\in \mathcal{S}_z} \sum_{\tilde{\sigma}\in S_N} \sgn(\tilde{\sigma})\prod_{i=1}^N \tilde{w}_i^{\tilde{\sigma}(i)-i}(\tilde{w}_i+1)^{-x'_{N-i+1}+y'_{N-\tilde{\sigma}(i)+1}+i-\tilde{\sigma}(i)-1}(1+p\tilde{w}_i)^{t+i-\tilde{\sigma}(i)}\\
 			&= \det\left[\sum_{w\in \mathcal{S}_z} \tilde{F}_{i,j}(w;\vec x,\vec y,t)J(w)\right]_{i,j=1}^{N}.
 		\end{align*}
 		To see the claim note that for $i\neq 1$ we have $\tau(i)=i-1$ and for $\tilde\sigma(i)\neq 1$ we have $\tau\tilde{\sigma}(i)=\tilde\sigma(i)-1$. Hence 
 		\begin{align*}
 		& \tilde{w}_i^{\tau\tilde{\sigma}(i)-\tau(i)}(\tilde{w}_i+1)^{-x'_{N-\tau(i)}+y'_{N-\tau\tilde{\sigma}(i)}+\tau(i)-\tau\tilde{\sigma}(i)-1}(1+p\tilde{w}_i)^{t+\tau(i)-\tau\tilde{\sigma}(i)}\\
 		&= \tilde{w}_i^{\tilde{\sigma}(i)-i}(\tilde{w}_i+1)^{-x'_{N-i+1}+y'_{N-\tilde{\sigma}(i)+1}+i-\tilde{\sigma}(i)-1}(1+p\tilde{w}_i)^{t+i-\tilde{\sigma}(i)},
 	    \end{align*}
 	    for $i\neq 1$ and $\tilde{\sigma}(i)\neq 1$. For the other situations we split into two cases:
 	    
 	    \noindent\textbf{Case 1}: $\tilde{\sigma}(1)\neq 1$. Then we have 
 	    \begin{align*}
 	    	& \tilde{w}_1^{\tau\tilde{\sigma}(1)-\tau(1)}(\tilde{w}_1+1)^{-x'_{N-\tau(1)}+y'_{N-\tau\tilde{\sigma}(1)}+\tau(1)-\tau\tilde{\sigma}(1)-1}(1+p\tilde{w}_1)^{t+\tau(1)-\tau\tilde{\sigma}(1)}\\
 	    	&= \tilde{w}_1^{\tilde{\sigma}(1)-1}(\tilde{w}_1+1)^{-x'_{N}+y'_{N-\tilde{\sigma}(1)+1}-\tilde{\sigma}(1)}(1+p\tilde{w}_i)^{t+1-\tilde{\sigma}(1)}\cdot \tilde{w}_1^{-N}(\tilde{w}_1+1)^{N-L}(1+p\tilde{w}_1)^{N}\\
 	    	&=\tilde{w}_1^{\tilde{\sigma}(1)-1}(\tilde{w}_1+1)^{-x'_{N}+y'_{N-\tilde{\sigma}(1)+1}-\tilde{\sigma}(1)}(1+p\tilde{w}_i)^{t+1-\tilde{\sigma}(1)}\cdot z^{-L}.
 	    \end{align*}
 	    Similarly if we denote $\tilde{\sigma}^{-1}(1)=j$, then the term involving $\tilde{w}_j$ among the sum in \eqref{eq:expansion of the determinant} equals
 	    \begin{align*}
 	    & \tilde{w}_{j}^{\tau(1)-\tau(j)}(\tilde{w}_{j}+1)^{-x'_{N-\tau(j)}+y'_{N-\tau(1)}+\tau(j)-\tau(1)-1}(1+p\tilde{w}_{j})^{t+\tau(j)-\tau(1)}\\
 	    &= \tilde{w}_j^{\tilde{\sigma}(j)-j}(\tilde{w}_1+1)^{-x'_{N-j+1}+y'_{N-\tilde{\sigma}(j)+1}-\tilde{\sigma}(j)+j-1}(1+p\tilde{w}_j)^{t+j-\tilde{\sigma}(j)}\cdot \tilde{w}_j^{N}(\tilde{w}_j+1)^{L-N}(1+p\tilde{w}_j)^{-N}\\
 	    &=
 	    \tilde{w}_j^{\tilde{\sigma}(j)-j}(\tilde{w}_1+1)^{-x'_{N-j+1}+y'_{N-\tilde{\sigma}(j)+1}-\tilde{\sigma}(j)+j-1}(1+p\tilde{w}_j)^{t+j-\tilde{\sigma}(j)}\cdot z^{L}.
 	    \end{align*}
 	    Hence
 	    \begin{align*}
 	    &\prod_{i=1}^N \tilde{w}_i^{\tau\tilde{\sigma}(i)-\tau(i)}(\tilde{w}_i+1)^{-x'_{N-\tau(i)}+y'_{N-\tau\tilde{\sigma}(i)}+\tau(i)-\tau\tilde{\sigma}(i)-1}(1+p\tilde{w}_i)^{t+\tau(i)-\tau\tilde{\sigma}(i)}\\
 	    &= z^L\cdot z^{-L}\cdot\prod_{i=1}^N
 	    \tilde{w}_i^{\tilde{\sigma}(i)-i}(\tilde{w}_i+1)^{-x'_{N-i+1}+y'_{N-\tilde{\sigma}(i)+1}+i-\tilde{\sigma}(i)-1}(1+p\tilde{w}_i)^{t+i-\tilde{\sigma}(i)}\\
 	    &= \prod_{i=1}^N \tilde{w}_i^{\tilde{\sigma}(i)-i}(\tilde{w}_i+1)^{-x'_{N-i+1}+y'_{N-\tilde{\sigma}(i)+1}+i-\tilde{\sigma}(i)-1}(1+p\tilde{w}_i)^{t+i-\tilde{\sigma}(i)}.
 	    \end{align*}
 	    \noindent \textbf{Case 2}: $\tilde{\sigma}(1)=1$. Then we have 
 	    \begin{align*}
 	    	&\tilde{w}_1^{\tau\tilde{\sigma}(1)-\tau(1)}(\tilde{w}_1+1)^{-x'_{N-\tau(1)}+y'_{N-\tau\tilde{\sigma}(1)}+\tau(1)-\tau\tilde{\sigma}(1)-1}(1+p\tilde{w}_1)^{t+\tau(1)-\tau\tilde{\sigma}(1)}\\
 	    	&= \tilde{w}_1^{\tilde{\sigma}(1)-1}(\tilde{w}_1+1)^{-x'_{N}+y'_{N-\tilde{\sigma}(1)+1}-\tilde{\sigma}(1)}(1+p\tilde{w}_i)^{t+1-\tilde{\sigma}(1)},
 	    \end{align*}
 	    so the claim follows and this completes the proof of the cyclic invariance.
	\end{enumerate}
\end{proof}
\smallskip
\subsection{Proof of the transition probability formula}
Now we turn to the proof of formula \eqref{eq:transition probability}. The proof basically follows the idea of \cite{borodin2008parallel, baik2018relax} where we replace the Kolmogorov forward equation by a free evolution equation with additional boundary conditions. 
The main extra difficulty here is that due to parallel update rule, the stationary distribution for the dynamics is non-uniform. In fact one can check $\mu(\vec x)\propto \prod_{i=1}^{N}(1-p\1_{x_{i-1}-x_i=1})$ is the stationary distribution. As a result it turns out that $\frac{P_t(\vec y\to \vec x)}{\prod_{i=1}^{N}(1-p\1_{x_{i-1}-x_i=1})}$ satisfies a relatively simpler dynamics than $P_t(\vec y\to \vec x)$. More precisely we have: 
\begin{lemma}\label{free evolution equation}
	Given $\vec x,\vec y\in \mathbb{Z}^N$. Let $G(\vec x,t;\vec y,0)$ be the (unique) solution of the following free evolution equation 
	\begin{equation}\label{eq:free evolution equation}
		G(\vec x,t+1;\vec y,0)=\sum_{b_1,\cdots,b_N\in \{0,1\}}\prod_{i=1}^{N} p^{b_i}(1-p)^{1-b_i}  G(\vec x-\vec b,t;\vec y,0),
	\end{equation}
	together with the boundary conditions: 
	\begin{equation}\label{eq:boundary conditions}
	\begin{aligned}
       &(1-p)\left[G(\cdots,x_{i-1}=x_{i}+1,x_{i},x_{i+1},\cdots,t;\vec y,0)-G(\cdots,x_{i-1}=x_{i},x_{i},x_{i+1},\cdots,t;\vec y,0)\right]\\
       &= p\left[G(\cdots,x_i,x_i-1,x_{i+1},\cdots,t;\vec y,0)-G(\cdots,x_i+1,x_i-1,x_{i+1},\cdots,t;\vec y,0)\right], \text{ for }1\leq i\leq N,
	\end{aligned}
	\end{equation}
	and the initial condition:
	\begin{equation}\label{eq:initial conditions}
		G(\vec x,0;\vec y,0) = \frac{\1_{\vec x=\vec y}}{\prod_{i=1}^{N}(1-p\1_{x_{i-1}-x_1=1})}.
	\end{equation}
	Then 
	\begin{equation}
		G(\vec x,t;\vec y,0) = \frac{P_t(\vec y\to \vec x)}{\prod_{i=1}^{N}(1-p\1_{x_{i-1}-x_1=1})},\quad \text{for all }\vec x,\vec y\in \conf_{N}^{(L)}.
	\end{equation}
	Note that we have used the convention $x_0=x_N+L$ so when $i=1$, \eqref{eq:boundary conditions} should be interpreted as 
	\begin{equation}\label{eq:boundary conditions2}
		\begin{aligned}
		&(1-p)\left[G(x_{1},\cdots,x_1-L+1,t;\vec y,0)-G(x_{1},\cdots,x_1-L,t;\vec y,0)\right]\\
		&= p\left[G(x_{1}-1,\cdots,x_1-L,t;\vec y,0)-G(x_{1}-1,\cdots,x_1-L+1,t;\vec y,0)\right].
		\end{aligned}
	\end{equation}
\end{lemma}

\begin{proof}
	 Set 
	 \begin{equation}
	 	H(\vec x,t;\vec y,0):= \frac{P_t(\vec y\to \vec x)}{\prod_{i=1}^{N}(1-p\1_{x_{i-1}-x_i=1})}.
	 \end{equation}
	 It suffices to show that for $\vec x,\vec y\in \conf_{N}^{(L)}$, $H(\vec x,t;\vec y,0)$ and $G(\vec x,t;\vec y,0)$ satisfy the same evolution equation.  This then implies that $H(\vec x,t;\vec y,0)=G(\vec x,t;\vec y,0)$ for all $\vec x,\vec y\in \conf_{N}^{(L)}$ since they satisfy the same initial condition. To better describe the evolution equation for $H(\vec x,t;\vec y,0)$, it is convenient to introduce the notion of clusters of a particle configuration $\vec x$. Given $\vec x=(x_1,\cdots,x_N)\in \conf_{N}^{(L)}$, for $1\leq i\leq N$ and $1\leq k\leq N$ we say $[x_i, x_{i-1},\cdots,x_{i-k+1}]$ is a cluster of size $k$ of $\vec x$ if 
	 \begin{equation*}
	 	x_{i+1}+1<x_i=x_{i-1}-1=\cdots=x_{i-k+1}-k+1<x_{i-k}-k.
	 \end{equation*}
	 Namely particle $i$ through $i-k+1$ are right next to each other while there are at least one empty site to the left of $x_i$ and right of $x_{i-k+1}$. Here we abuse notation by allowing the index to exceed $\{1,\cdots,N\}$ and this should be understood with the convention $x_{i+kN}=x_{i}-kL$ for $1\leq i\leq N$ and $k\in \mathbb{Z}$. For convenience when $[x_i,x_{i-1},\cdots,x_1,x_0,\cdots x_{-j}]$ is a cluster for some $0\leq j<N$, we will also say $[x_i,\cdots,x_1,x_N,\cdots,x_{N-j}]$ forms a cluster so that all the indices appearing will be between $1$ and $N$.
	 
	  Let $\mathcal{N}_c(\vec x)$ be the number of clusters in configuration $\vec x$ and let $x_{c_j}$, $1\leq j\leq \mathcal{N}_c(\vec x)$ be the locations of the left-most particles in each cluster. Then it is straightforward to check that $H(\vec x,t;\vec y,0)$ satisfies 
	 \begin{equation}\label{eq: evolution equation clusters}
	 	H(\vec x,t+1;\vec y,0) = \sum_{\substack{b_{c_j}\in \{0,1\},\\ 1\leq j\leq \mathcal{N}_c(\vec x) }}\prod_{j=1}^{\mathcal{N}_c(\vec x)} p^{b_{c_j}}(1-p)^{1-b_{c_j}} H(\vec x-\sum_{j=1}^{\mathcal{N}_c(\vec x)}b_{c_j}\vec e_{c_j},t;\vec y,0),
	 \end{equation}
	 where $\vec e_{c_j}\in \mathbb{Z}^N$ has $1$ in the $c_j$-th coordinate and $0$ in the other coordinates. 
	 
	 We claim that for $\vec x,\vec y\in \conf_{N}^{(L)}$, \eqref{eq: evolution equation clusters} and \eqref{eq:free evolution equation} take the same form (with $H$ replaced by $G$) provided that $G(\vec x,t;\vec y,0)$ satisfies boundary conditions \eqref{eq:boundary conditions}. Due to the sum of products form of \eqref{eq: evolution equation clusters} and \eqref{eq:free evolution equation} it suffices to check
	 \begin{equation}\label{eq:reduction using boundary conditions1}
	 	\sum_{b_{i}\in \{0,1\}} p^{b_i}(1-p)^{1-b_i} G(\vec x-b_i\vec e_i,t;\vec y,0) = \sum_{\substack{b_j\in \{0,1\}\\j=i-m+1,\cdots,i}}\prod_{j} p^{b_j}(1-p)^{1-b_j} G(\vec x-\sum_{j}b_j\vec e_j,t;\vec y,0),
	 \end{equation}
	 for $\vec x,\vec y\in \conf_{N}^{(L)}$ and a single cluster $[x_i,x_{i-1},\cdots,x_{i-m+1}]$ of size $m$. We will show the stronger statement: \eqref{eq:reduction using boundary conditions1} actually holds for any $\vec x=(x_1,\cdots,x_N)\in \mathbb{Z}^N$ with $x_i,x_{i-1},\cdots,x_{i-m+1}$  merely satisfying $x_i=x_{i-1}-1=\cdots=x_{i-m+1}-m+1$. We do not require empty sites at the left and right ends so they may not form a cluster.

	 We prove this by induction on $m$. For $m=1$ this is trivial. Assume the claim is true for any clusters of size $\leq m$. Now let $\vec x\in \mathbb{Z}^N$ with $x_i=x_{i-1}-1=\cdots=x_{i-m}-m$ for some $1\leq i\leq N$. Here without loss of generality we can assume $i-m\geq 1$, otherwise replace $j$ by $j+N$ for indices $j\leq 0$.  Then
	 \begin{equation}\label{eq:reduction using boundary conditions2}
	 	\begin{aligned}
	 	&\sum_{\substack{b_j\in \{0,1\}\\j=i-m,\cdots,i}}\prod_{j} p^{b_j}(1-p)^{1-b_j} G(\vec x-\sum_{j}b_j\vec e_j,t;\vec y,0)\\
	 	&=\sum_{b_i\in \{0,1\}}p^{b_i}(1-p)^{1-b_i}\cdot \Big(\sum_{\substack{b_j\in \{0,1\}\\j=i-m,\cdots,i-1}}\prod_{j=i-m}^{i-1} p^{b_j}(1-p)^{1-b_j} G(\vec x-\sum_{j=i-m}^{i}b_j\vec e_j,t;\vec y,0)\Big)\\
	 	&=\sum_{b_i\in \{0,1\}}p^{b_i}(1-p)^{1-b_i}\cdot\Big(\sum_{b_{i-1}\in \{0,1\}}p^{b_{i-1}}(1-p)^{1-b_{i-1}}G(\vec x-b_{i}\vec e_{i}-b_{i-1}\vec e_{i-1},t;\vec y,0)\Big)\\
	 	&= (1-p)^2 G(\vec x,t)+p(1-p)G(\vec x-\vec e_{i},t)+p(1-p)G(\vec x-\vec e_{i-1},t)+p^{2}G(\vec x-\vec e_{i}-\vec e_{i-1},t),
	 	\end{aligned}
	 \end{equation}
	 where we used the induction hypothesis in the second equality of \eqref{eq:reduction using boundary conditions2} for the sum inside the brackets. For notational conveniece we have suppressed the dependence on $\vec y$ in the last line to save space. Now by the boundary conditions \eqref{eq:boundary conditions} (possibly \eqref{eq:boundary conditions2}) we have 
	 \begin{equation}\label{eq:reduction using boundary conditions3}
	  (1-p)[G(\vec x,t;\vec y,0)-G(\vec x-\vec e_{i-1},t;\vec y,0)] = p[G(\vec x-\vec e_i-\vec e_{i-1},t;\vec y,0)-G(\vec x-\vec e_{i},t;\vec y,0)].
	 \end{equation}
	 Inserting \eqref{eq:reduction using boundary conditions3} into \eqref{eq:reduction using boundary conditions2} we see the last line of \eqref{eq:reduction using boundary conditions2} simplifies to 
	 \begin{equation}
	 	(1-p)G(\vec x,t;\vec y,0)+pG(\vec x-\vec e_i,t;\vec y,0),
	 \end{equation}
	 which is precisely the left hand side of \eqref{eq:reduction using boundary conditions1} and this completes the proof of Lemma \ref{free evolution equation}. 
\end{proof}
\smallskip
Now we discuss the proof of Proposition \ref{prop:transition probability}.
\begin{proof}[Proof of Proposition \ref{prop:transition probability}]
By Lemma \ref{free evolution equation} it suffices to prove 
\begin{equation}\label{eq:determinant formula free evolution solution}
	G(\vec x,t;\vec y,0) = \oint_{\Gamma} \frac{\dd z}{2\pi \ii z} \det\left[\sum_{w\in \mathcal{S}_z} F_{i,j}(w;\vec x,\vec y,t)J(w)\right]_{i,j=1}^N.
\end{equation}
To see this, we check that the right-hand side of \eqref{eq:determinant formula free evolution solution} satisfies the free evolution equation \eqref{eq:free evolution equation}, the boundary conditions \eqref{eq:boundary conditions} and the initial condition \eqref{eq:initial conditions}. 
\smallskip

For the free evolution equation \eqref{eq:free evolution equation}, note first that it is straightforward to check
\begin{equation}
	\sum_{b_{N-i+1}\in \{0,1\}} p^{b_{N-i+1}}(1-p)^{1-b_{N-i+1}} F_{i,j}(w;\vec x-\vec b,\vec y;t) = F_{i,j}(w;\vec x-\vec b+b_{N-i+1}\vec e_{N-i+1},\vec y;t+1).
\end{equation}
Here $\vec e_i\in \mathbb{Z}^N$ has $1$ in the $i$-th entry and $0$ for the others. By multi-linearity of determinants we have
\begin{equation*}
	\begin{aligned}
	     &\sum_{\substack{b_k\in \{0,1\}\\k=1,\cdots,N}}\prod_{k=1}^{N}p^{b_k}(1-p)^{1-b_k}\det\left[\sum_{w\in \mathcal{S}_z} F_{i,j}(w;\vec x-\vec b,\vec y,t)J(w)\right]_{i,j=1}^N
	     \\&= \det\left[\sum_{w\in \mathcal{S}_z} \sum_{b_{N-i+1}\in \{0,1\}}p^{b_{N-i+1}}(1-p)^{1-b_{N-i+1}} F_{i,j}(w;\vec x-\vec b,\vec y,t)J(w)\right]_{i,j=1}^N\\
	     &= \det\left[\sum_{w\in \mathcal{S}_z}  F_{i,j}(w;\vec x-\vec b+b_{N-i+1}\vec e_{N-i+1},\vec y,t+1)J(w)\right]_{i,j=1}^N\\
	     &= \det\left[\sum_{w\in \mathcal{S}_z}  F_{i,j}(w;\vec x,\vec y,t+1)J(w)\right]_{i,j=1}^N.
	\end{aligned}
\end{equation*}
 Here in the last equality above we used the fact that 
$F_{i,j}(w;\vec x,\vec y,t)$ only depends on the 
$(N-i+1)$-th entry of $\vec x$. Now \eqref{eq:free evolution equation} follows from linearity of the integral.
\smallskip

Next we check the boundary conditions \eqref{eq:boundary conditions}. Given $1\leq k\leq N$, assume that $\vec x=(x_1,\cdots,x_N)\in \mathbb{Z}^N$ satisfies  $x_{k-1}=x_k+1$. Note that when $k=1$ this means $x_1=x_N+L-1$. Then the boundary conditions \eqref{eq:boundary conditions} can be expressed as
\begin{equation}\label{eq: bc1 modified}
	(1-p)[G(\vec x,t;\vec y,0)-G(\vec x-\vec e_{k-1},t;\vec y,0)] = p[G(\vec x-\vec e_{k-1}-\vec e_{k},t;\vec y,0)-G(\vec x-\vec e_{i},t;\vec y,0)],
\end{equation}
for $2\leq k\leq N$, and 
\begin{equation}\label{eq: bc2 modified}
	(1-p)[G(\vec x,t;\vec y,0)-G(\vec x-\vec e_{N},t;\vec y,0)] = p[G(\vec x-\vec e_{N}-\vec e_{1},t;\vec y,0)-G(\vec x-\vec e_{1},t;\vec y,0)],
\end{equation}
for $k=1$. We prove \eqref{eq: bc1 modified} first. Note that for $2\leq k \leq N$ and $x_{k-1}=x_k+1$ we have 
\begin{equation*}
	F_{i,j}(w;\vec x,\vec y,t)-F_{i,j}(w;\vec x-\vec e_{k-1},\vec y,t) 
	= \begin{cases}
	0,\quad &\text{if } i\neq N-k+2,\\
	-w\cdot F_{i,j}(w;\vec x,\vec y,t),\quad &\text{if } i=N-k+2.
	\end{cases}
\end{equation*}
Hence by multi-linearity we have 
\begin{equation*}
 \begin{aligned}
  &\det\left[\sum_{w\in \mathcal{S}_z} F_{i,j}(w;\vec x,\vec y,t)J(w)\right]_{i,j=1}^N-\det\left[\sum_{w\in \mathcal{S}_z} F_{i,j}(w;\vec x-\vec e_{k-1},\vec y,t)J(w)\right]_{i,j=1}^N\\
  &= \det\left[\sum_{w\in \mathcal{S}_z} F_{i,j}(w;\vec x,\vec y,t)J(w)\cdot (1-(1+w)\1_{i=N-k+2})\right]_{i,j=1}^N.
 \end{aligned}
\end{equation*}
Similarly 
\begin{equation*}
\begin{aligned}
&\det\left[\sum_{w\in \mathcal{S}_z} F_{i,j}(w;\vec x-\vec e_{k},\vec y,t)J(w)\right]_{i,j=1}^N-\det\left[\sum_{w\in \mathcal{S}_z} F_{i,j}(w;\vec x-\vec e_k-\vec e_{k-1},\vec y,t)J(w)\right]_{i,j=1}^N\\
&= \det\left[\sum_{w\in \mathcal{S}_z} F_{i,j}(w;\vec x-\vec e_k,\vec y,t)J(w)\cdot (1-(1+w)\1_{i=N-k+2})\right]_{i,j=1}^N.
\end{aligned}
\end{equation*}
Now since $ F_{i,j}(w;\vec x-\vec e_k,\vec y,t) = F_{i,j}(w;\vec x,\vec y,t)$ for $i\neq N-k+1$ and 
\begin{equation*}
	(1-p)F_{N-k+1,j}(w;\vec x,\vec y,t)+pF_{N-k+1,j} (w;\vec x-\vec e_k,\vec y,t) = w\cdot F_{N-k+2,j}(w;\vec x,\vec y,t),
\end{equation*}
we have 
\begin{equation*}
	\begin{aligned}
	&(1-p)\cdot\det\left[\sum_{w\in \mathcal{S}_z} F_{i,j}(w;\vec x,\vec y,t)J(w)\cdot (1-(1+w)\1_{i=N-k+2})\right]_{i,j=1}^N\\
	&+p\cdot\det\left[\sum_{w\in \mathcal{S}_z} F_{i,j}(w;\vec x-\vec e_k,\vec y,t)J(w)\cdot (1-(1+w)\1_{i=N-k+2})\right]_{i,j=1}^N\\
    &=  \det\left[M_{i,j}^{(k)}(\vec x,\vec y,t)\right]_{i,j=1}^N = 0.
	\end{aligned}
\end{equation*}
Here $M_{i,j}^{(k)}(\vec x,\vec y,t) =  \sum_{w\in \mathcal{S}_z} F_{i,j}(w;\vec x,\vec y,t)J(w)$ for $i\neq N-k+1, N-k+2$ and $M_{N-k+1,j}^{(k)}(\vec x,\vec y,t) = -M_{N-k+2,j}^{(k)}(\vec x,\vec y,t) =\sum_{w\in \mathcal{S}_z} F_{N-k+2,j}(w;\vec x,\vec y,t)J(w)\cdot w$. The determinant is $0$ since there are two proportional rows. Now \eqref{eq: bc1 modified} follows from linearity of the integral.

The proof of \eqref{eq: bc2 modified} is similar. The only thing changes is that when $k=1$, we have $M_{i,j}^{(1)}(\vec x,\vec y,t) = \sum_{w\in \mathcal{S}_z} F_{i,j}(w;\vec x,\vec y,t)J(w)$ for $i\neq 1,N$ while 
\begin{equation}\label{eq:bc k=1}
	M_{N,j}^{(1)}(\vec x,\vec y,t) = \sum_{w\in \mathcal{S}_z} w^{-N}(w+1)^{N-L}(1+pw)^{N} F_{1,j}(w;\vec x,\vec y,t)J(w)\cdot w=-z^{-L}M_{1,j}^{(1)}(\vec x,\vec y,t).
\end{equation}
 Hence $\det[M^{(1)}_{i,j}(\vec x,\vec y,t)]=0$ since row $1$ and $N$ are proportional. Note that in the last equality of \eqref{eq:bc k=1} we used the fact that $w\in \mathcal{S}_z$.
 \smallskip
 
 Finally we check the initial condition \eqref{eq:initial conditions}. We need to show
 \begin{equation}\label{eq:ic reduction}
 	\oint_{\Gamma} \frac{\dd z}{2\pi \ii z} \det\left[\sum_{w\in \mathcal{S}_z} F_{i,j}(w;\vec x,\vec y,0)J(w)\right]_{i,j=1}^N = \frac{\1_{\vec x=\vec y}}{\prod_{i=1}^{N}(1-p\1_{x_{i-1}-x_1=1})}.
 \end{equation}Thanks to the cyclic-shift invariance of both sides of \eqref{eq:ic reduction} (see (iv) of Proposition~\ref{properties of the transition probability}) we can assume without loss of generality that $\vec x=(x_1,\cdots,x_N)\in \conf_{N}^{(L)}$ satisfies $x_1<x_{N}+L-1$. In fact since $N<L$ there is at least one $1\leq i\leq N$ such that $x_{i}<x_{i-1}-1$ and we can replace $\vec x$ and $\vec y$ by $\vec x':=(x_{i},x_{i+1},\cdots,x_{N},x_{1}-L,\cdots, x_{i-1}-L)$ and $\vec y':=(y_{i},y_{i+1},\cdots,y_{N},y_{1}-L,\cdots, y_{i-1}-L)$. Then the two sides of \eqref{eq:ic reduction} remain the same and $x_1'< x_N'+L-1$. By \eqref{eq: integral rep for entry} we have 
 \begin{equation*}
 	\begin{aligned}
         &\sum_{w\in \mathcal{S}_z} F_{i,j}(w;\vec x,\vec y,0)J(w) =\left(\oint_{|w+1|=R}-\oint_{|w+1|=\epsilon}-\oint_{|pw+1|=\epsilon}\right)\frac{\dd w}{2\pi \ii} F_{i,j}(w;\vec x,\vec y,0)\frac{\hat{q}_z(w)+z^L}{\hat{q}_z(w)}\\
         & = \left(\oint_{|w+1|=R}-\oint_{|w+1|=\epsilon}-\oint_{|pw+1|=\epsilon}\right)\frac{\dd w}{2\pi \ii}  \frac{w^{j-i+N}(w+1)^{-x_{N-i+1}+y_{N-j+1}+i-j-1+L-N}(1+pw)^{i-j}}{w^{N}(w+1)^{L-N}-z^L(1+pw)^{N}}\\
         &:= I_1(i,j) -I_2(i,j)-I_3(i,j),
 	\end{aligned}
 \end{equation*}
 where $I_1(i,j),I_2(i,j),I_3(i,j)$ are the integrals over the three contours, respectively. Here we recall that $R$ and $\epsilon$ are chosen to be large (small) enough so that $\mathcal{S}_z$ is contained in the region $\{\epsilon <|w+1|<R\}\backslash\{|1+pw|\leq \epsilon\}$. For $I_3(i,j)$ we note that
 \begin{equation*}
 	\begin{aligned}
 	I_3(i,j) &= \oint_{|pw+1|=\epsilon} \frac{\dd w}{2\pi \ii }F_{i,j}(w;\vec x,\vec y,0) + z^L\oint_{|pw+1|=\epsilon} \frac{\dd w}{2\pi \ii }\frac{F_{i,j}(w;\vec x,\vec y,0)(1+pw)^N}{q_z(w)}\\
 	&= \oint_{|pw+1|=\epsilon} \frac{\dd w}{2\pi \ii }F_{i,j}(w;\vec x,\vec y,0),
 	\end{aligned}
 \end{equation*} 
 where the second term vanishes since $F_{i,j}(w;\vec x,\vec y,0)(1+pw)^N$ is analytic at $w=-1/p$ for all $1\leq i,j\leq N$ and $q_z(w)$ is nonzero at $w=-1/p$.  For the other parts we write
\begin{equation}
\begin{aligned}
I_1(i,j) &= \oint_{|w+1|=R} \frac{\dd w}{2\pi \ii }F_{i,j}(w;\vec x,\vec y,0) + z^L\oint_{|w+1|= R} \frac{\dd w}{2\pi \ii }\frac{F_{i,j}(w;\vec x,\vec y,0)(1+pw)^N}{q_z(w)}\\
&:= \oint_{|w+1|=R} \frac{\dd w}{2\pi \ii }F_{i,j}(w;\vec x,\vec y,0) + z^L I_{1}'(i,j),
\end{aligned}
\end{equation} 
and 
\begin{equation}
\begin{aligned}
I_2(i,j) &= z^{-L}\oint_{|w+1|=\epsilon}\frac{\dd w}{2\pi \ii}\frac{w^{j-i+N}(w+1)^{-x_{N-i+1}+y_{N-j+1}+i-j-1+L-N}(1+pw)^{i-j-N}}{z^{-L}w^{N}(w+1)^{L-N}(1+pw)^{-N}-1}\\
&:= z^{-L} I_2'(i,j).
\end{aligned}
\end{equation}
Depending on the analytical properties of the integrands in $I_1'$ and $I_2'$ we split into two cases:

\noindent\textbf{Case 1:} $x_1<y_1$. First note the basic fact that for any  $\vec x,\vec y\in \conf_{N}^{(L)}$  we have 
\begin{equation}\label{eq: relation due to gap}
	x_1-L+i\leq x_{N-i+1}\leq x_1-N+i,\quad  \text{and}\quad y_1-L+j\leq y_{N-j+1}\leq y_1-N+j,\quad 1\leq i,j\leq N.
\end{equation}
Now if $x_1<y_1$, then for any $1\leq i,j\leq N$ we have 
\begin{equation*}
	-x_{N-i+1}+y_{N-j+1}+i-j-1+L-N\geq y_1-x_1-1\geq 0.
\end{equation*}
Hence the integrand of $I_2'(i,j)$ is analytic at $w=-1$, so $I_2'(i,j)=0$. Therefore in this case we have 
\begin{equation}\label{eq: residue small circle}
	\begin{aligned}
	 &\oint_{\Gamma} \frac{\dd z}{2\pi \ii z} \det\left[\sum_{w\in \mathcal{S}_z} F_{i,j}(w;\vec x,\vec y,0)J(w)\right]_{i,j=1}^N \\
	 &= \oint_{\Gamma} \frac{\dd z}{2\pi \ii z}\det\left[\oint_{\Gamma_{0,-1}}\frac{\dd w}{2\pi \ii} F_{i,j}(w;\vec x,\vec y,0) +z^L I_1'(i,j)\right]_{i,j=1}^N\\
	 &= \det\left[\oint_{\Gamma_{0,-1}}\frac{\dd w}{2\pi \ii} F_{i,j}(w;\vec x,\vec y,0)\right]_{i,j=1}^N. 
	\end{aligned}
\end{equation}
Here in the last equality of \eqref{eq: residue small circle} we take the outer integral contour $\Gamma$ to be $|z|=r$ and let $r\to 0$. $\Gamma_{0,-1}$ is any simple closed contour with  $0$ and $-1$ inside and $-1/p$ outside.

\noindent\textbf{Case 2:} $x_1\geq y_1$. Write
\begin{equation}\label{eq:remainder term in I1}
	I_1'(i,j) = \oint_{|w+1|= R} \frac{\dd w}{2\pi \ii}\frac{w^{j-i-N}(w+1)^{-x_{N-i+1}+y_{N-j+1}+i-j-1-L+N}(1+pw)^{i-j+N}}{1-z^{L}w^{-N}(w+1)^{-L+N}(1+pw)^{N}}.
\end{equation} Again by \eqref{eq: relation due to gap} we have 
\begin{equation}\label{eq: upper bound of the exponent}
	-x_{N-i+1}+y_{N-j+1}+i-j-1-L+N\leq -x_1+L-i+y_1-N+j+i-j-1-L+N\leq -1.
\end{equation}
We claim that the first inequality in \eqref{eq: upper bound of the exponent} is strict and  hence $-x_{N-i+1}+y_{N-j+1}+i-j-1-L+N\leq -2$. This is due to our original assumption that $x_N+L-1>x_1$ and hence $x_{N-i+1}>x_1-L+i$ for all $1\leq i\leq N$. Owing to this fact, the integrand in \eqref{eq:remainder term in I1} is $O(R^{-2})$ since $w^{j-i-N}(1+pw)^{i-j+N}$ remains bounded. Hence $I_1'(i,j)\to 0$ as $R\to \infty$. But since it is independent of large enough $R$ , we have $I_1'(i,j)=0$ for all $R$ large enough.  Now a similar argument as in \eqref{eq: residue small circle} with $\Gamma$ instead be a large circle $|z|=r$ with $r\to \infty$ implies:
\begin{equation}
\begin{aligned}
&\oint_{\Gamma} \frac{\dd z}{2\pi \ii z} \det\left[\sum_{w\in \mathcal{S}_z} F_{i,j}(w;\vec x,\vec y,0)J(w)\right]_{i,j=1}^N \\
&= \oint_{\Gamma} \frac{\dd z}{2\pi \ii z}\det\left[\oint_{\Gamma_{0,-1}}\frac{\dd w}{2\pi \ii} F_{i,j}(w;\vec x,\vec y,0) +z^{-L} I_2'(i,j)\right]_{i,j=1}^N\\
&= \det\left[\oint_{\Gamma_{0,-1}}\frac{\dd w}{2\pi \ii} F_{i,j}(w;\vec x,\vec y,0)\right]_{i,j=1}^N. 
\end{aligned}
\end{equation}
In conclusion we have reduced checking \eqref{eq:ic reduction} to checking the following:
\begin{equation}\label{eq:ic further reduction}
	\det\left[\oint_{\Gamma_{0,-1}}\frac{\dd w}{2\pi \ii} F_{i,j}(w;\vec x,\vec y,0)\right]_{i,j=1}^N=\frac{\1_{\vec x=\vec y}}{\prod_{i=1}^{N}(1-p\1_{x_{i-1}-x_1=1})},
\end{equation}
for any $\vec x,\vec y\in \conf_{N}^{(L)}$ with $x_1<x_{N}+L-1$. But this is precisely equation $(3.30)$ of \cite{borodin2008parallel} which appears in checking that the determinantal formula for the transition probability of discrete parallel TASEP on $\mathbb{Z}$ satisfies the initial condition. We will not repeat the proof here but point out that due to the assumption $x_1<x_{N}+L-1$, we have $1-p\1_{x_0-x_1=1}=1$ so \eqref{eq:ic further reduction} is really identical to equation $(3.30)$ of \cite{borodin2008parallel}.
 \end{proof}
\bigskip

\section{Finite-time Multi-point joint distribution under general initial conditions}\label{sec: multitime}
\subsection{A Toeplitz-like determinant formula}
In this section we derive a formula for the finite-time multi-point joint distributions of discrete time parallel periodic TASEP under arbitrary initial condition. The proof basically follows the strategy of \cite{baik2019multi} by performing a multiple sum of transition probabilities over suitable particle configurations. The main technical part is a Cauchy-type identity for summation of left and right eigenfunctions (see Proposition \ref{prop: Cauchy identity}) which generalizes Proposition 3.4 of \cite{baik2019multi} and some new difficulties appear.

\begin{theorem}[Multi-point joint distribution for discrete time parallel TASEP in $\conf_{N}^{(L)}$] \label{thm: mutlitime distribution general}
	Given $\vec y \in \conf_{N}^{(L)}$. Let $\vec x(t)=(x_1(t),\cdots, x_N(t))\in \conf_{N}^{(L)}$ be particle configurations evolving according to the discrete time parallel TASEP in $\conf_{N}^{(L)}$ at time $t$ with initial configuration $\vec x(0)=\vec y$. Fix a positive integer $m$. Let $(k_1,t_1),\cdots,(k_m,t_m)\in \{1,\cdots,N\}\times \mathbb{N}$ be distinct with $0\leq t_1\leq \cdots\leq t_m$. Let $a_{i}\in \mathbb{Z}$ for $1\leq i\leq m$. Then 
	\begin{equation}\label{eq: multipoint}
		\mathbb{P}_{\vec y}^{(L)}\left(\bigcap_{i=1}^m \{x_{k_i}(t_i)\geq a_i\}\right) = \oint\cdots\oint \frac{\dd z_m}{2\pi \ii z_m}\cdots \frac{\dd z_1}{2\pi \ii z_1} \mathcal{C}^{(L)}(\vec z)\mathcal{D}_{\vec y}^{(L)}(\vec z),
	\end{equation}
where the contours for the integrals are nested circles $0<|z_m|<\cdots<|z_1|$. Here $\vec z=(z_1,\cdots,z_m)$.  The functions $\mathcal{C}^{(L)}(\vec z)$ and $\mathcal{D}^{(L)}_{\vec y}(\vec z)$ are defined by
\begin{equation}\label{eq: C(z)}
	\mathcal{C}^{(L)}(\vec z) := (-1)^{(N-k_m)(N-1)}z_1^{(N-k_1)L}\prod_{\ell=2}^{m}\left[z_{\ell}^{(k_{\ell-1}-k_{\ell})L}\left(\left(\frac{z_{\ell}}{z_{\ell-1}}\right)^L-1\right)^{N-1}\right],
\end{equation}
and 
\begin{equation}\label{eq: D(z)}
	\mathcal{D}_{\vec y}^{(L)}(\vec z) := \det\left[\sum_{\substack{w_{\ell}\in \mathcal{S}_{z_\ell}\\ \ell=1,\cdots,m}}\frac{w_1^{i-1-N}(1+pw_1)^{N-i+1}(1+w_1)^{y_{N-i+1}+N-i+1}\cdot w_m^{-j}}{\prod_{\ell=2}^{m}(w_{\ell}-w_{\ell-1})}\prod_{\ell=1}^{N}G_{\ell}(w_{\ell})\right]_{i,j=1}^N,
\end{equation}
where for $1\leq \ell\leq m$,
\begin{equation}
	G_{\ell}(w):= \frac{w(w+1)(1+pw)}{N+Lw+p(L-N)w^2}\cdot \frac{w^{k_{\ell}}(1+w)^{-a_{\ell}-k_{\ell}}(1+pw)^{t_{\ell}-k_{\ell}}}{w^{k_{\ell-1}}(1+w)^{-a_{\ell-1}-k_{\ell-1}}(1+pw)^{t_{\ell-1}-k_{\ell-1}}}.
\end{equation}
Here $k_0=t_0=a_0:=0$ and we suppress the dependence on $a_i$, $k_i$ and $t_i$'s in $\mathcal{C}^{(L)}(\vec z)$ and $\mathcal{D}_{\vec y}^{(L)}(\vec z)$.
\end{theorem}

\begin{proof}
	We start with the case $m=1$ for a warm-up. First by Cauchy-Binet formula we can rewrite the transition probability \eqref{eq:transition probability} as 
	\begin{equation*}
		P_t(\vec x\to\vec x') = \oint_\Gamma \frac{\dd z}{2\pi \ii z} \sum_{\vec w\in (\mathcal{S}_z)^N}\Psi_{\vec x}^{\ell}(\vec w)\Psi_{\vec x'}^{r}(\vec w) \mathcal{Q}(\vec w;t),
	\end{equation*}
	where 
	\begin{align}
		&\Psi_{\vec x}^{\ell}(\vec w) = \det\left[\left(\frac{w_i}{1+pw_i}\right)^{j-1}(1+w_i)^{x_{N-j+1}-j+1}\right]_{i,j=1}^N, \label{eq: left eigenfunction}\\
		&\Psi_{\vec x}^{r}(\vec w) = \prod_{i=1}^{N}(1-p\1_{x_{i-1}-x_i=1})\cdot\det\left[\left(\frac{w_i}{1+pw_i}\right)^{1-j}(1+w_i)^{-x_{N-j+1}+j-1}\right]_{i,j=1}^N, \label{eq: right eigenfunction}
	\end{align}
    and
	\begin{equation}
		\mathcal{Q}(\vec w;t) = \frac{1}{N!}\prod_{j=1}^N \frac{(1+pw_j)^tJ(w_j)}{w_j+1}=\frac{1}{N!}\prod_{j=1}^N \frac{w_j(1+pw_j)^{t+1}}{N+Lw_j+p(L-N)w_j^2}.
	\end{equation}
	Now to get the one-point distribution $\mathbb{P}_{\vec y}^{(L)}(x_{k}(t)\geq a)$ we perform a summation over all configurations $\vec x\in \conf_{N}^{(L)}$ with $x_k\geq a$ of the transition probability and interchange the order of integration and summation:
	\begin{equation}\label{eq: one point}
		\begin{aligned}
	     \mathbb{P}_{\vec y}^{(L)}(x_{k}(t)\geq a)& = \sum_{\vec x\in \conf_{N}^{(L)}\cap \{x_k\geq a\}} P_t(\vec y\to \vec x)\\
	     &= \oint_\Gamma \frac{\dd z}{2\pi \ii } \sum_{\vec w\in (\mathcal{S}_z)^N} \Psi_{\vec y}^{\ell}(\vec w)Q(\vec w;t)\left(\sum_{\vec x\in \conf_{N}^{(L)}\cap \{x_k\geq a\}}\Psi_{\vec x}^{r}(\vec w)\right).
		\end{aligned}
	\end{equation}
	Now by Corollary~\ref{cor: single sum arbitrary index} in Section~\ref{sec: summation identity} we have $\sum_{\vec x\in \conf_{N}^{(L)}\cap \{x_k\geq a\}}\Psi_{\vec x}^{r}(\vec w)$ equals
	\begin{equation}\label{eq: single sum for x_k}
		 (-1)^{(N-k)(N-1)}z^{(N-k)L}\prod_{i=1}^{N}\left[\left(\frac{1+pw_i}{w_i}\right)^{N-k}\cdot(1+w_i)^{-a+N-k+1}\right] \cdot \det[w_{i}^{-j}]_{i,j=1}^N.
	\end{equation}
	Inserting \eqref{eq: single sum for x_k} back to \eqref{eq: one point} and use Cauchy-Binet formula backwards we conclude that 
	\begin{equation}\label{eq: one point final}
	\begin{aligned}
			&\mathbb{P}_{\vec y}^{(L)}(x_k(t)\geq a)= \frac{(-1)^{(N-k)(N-1)}}{2\pi \ii}\\
			\cdot&\oint_\Gamma \frac{\dd z}{z^{1-(N-k)L}} \det\left[\sum_{w\in \mathcal{S}_z}\frac{w^{j-i+k-N}(1+pw)^{t-j+N-k+2}(1+w)^{y_{N-j+1}-j-a+N-k+2}}{N+Lw+p(L-N)w^2}\right]_{i,j=1}^N.
	\end{aligned}
	\end{equation}
	Here in order to interchange the order of summation and integration as in \eqref{eq: one point} we need to make sure that the summation over $\vec x\in \conf_{N}^{(L)}\cap \{x_k\geq a\}$ converges absolutely. This is guaranteed if we assume $|w+1|>1$ for all $w\in \mathcal{S}_z$ (see the discussion in Proposition \ref{prop: summation single}).  By choosing the contour $\Gamma$ to be a circle with large radius $R$ we can make sure that $|w+1|>1$ for all $w\in \mathcal{S}_z$ since for $|z|=R$ and all $w$ satisfying $w^N(1+pw)^{-N}(w+1)^{L-N}=z^L$ we have $|w+1|=O(|z|^{\frac{L}{L-N}})>1$. Finally a similar argument as in (ii) of Proposition \ref{properties of the transition probability} shows that the right-hand side of \eqref{eq: one point final} does not depend on the choice of $\Gamma$, so we can deform $\Gamma$ to be any simple closed contour containing $0$.
	
	\smallskip
	Now assume $m\geq 2$. Then
	\begin{equation}
		\mathbb{P}_{\vec y}^{(L)}\left(\cap_{i=1}^m \{x_{k_i}(t_i)\geq a_i\}\right) = \sum_{\substack{\vec x^{(\ell)}\in \conf_{N}^{(L)}\cap\{x^{(\ell)}_{k_\ell}\geq a_\ell\}\\ \ell=1,\cdots,m}} P_{t_1-t_0}(\vec y\to \vec x^{(1)})\cdots P_{t_m-t_{m-1}}(\vec x^{(m-1)}\to \vec x^{(m)}),
	\end{equation}
	where $t_0:= 0$.  As in the $m=1$ case we rewrite the transition probability using Cauchy-Binet formula and interchange the order of summation and integration (will be justified later) so that $\mathbb{P}_{\vec y}(\cap_{i=1}^m \{x_{k_i}(t_i)\geq a_i\})$ equals
	\begin{equation}
	\begin{aligned}
	\oint \frac{\dd z_1}{2\pi \ii z_1} \cdots \oint \frac{\dd z_m}{2\pi \ii z_m} \sum_{\substack{\vec w^{(\ell)}\in (\mathcal{S}_{z_{\ell}})^N\\ \ell=1,\cdots,m}}\mathcal{P}(\vec w^{(1)},\cdots,\vec w^{(m)})\prod_{\ell=1}^{m} \mathcal{Q}(\vec w^{(\ell)};t_{\ell}-t_{\ell-1}).
	\end{aligned}
	\end{equation}
	Here $\vec w^{(\ell)}=(w^{(\ell)}_1,\cdots,w^{(\ell)}_m)$ and 
	\begin{equation}\label{eq: multi-point 1}
		\mathcal{P}(\vec w^{(1)},\cdots,\vec w^{(m)}) = \Psi_{\vec y}^{(\ell)}(\vec w^{(1)})\cdot\left[\prod_{\ell=1}^{m-1}\mathcal{H}_{k_\ell,a_\ell}(\vec w^{(\ell)};\vec w^{(\ell+1)})\right]\cdot\left[\sum_{\vec x\in \conf_{N}^{(L)}\cap \{x_{k_m}\geq a_m\}}\Psi_{\vec x}^{r}(\vec w^{(m)})\right].
	\end{equation}
	Where 
	\begin{equation}\label{eq: multi-point 2}
		\mathcal{H}_{k,a}(\vec w;\vec w'):= \sum_{\vec x\in \conf_{N}^{(L)}\cap \{x_{k}\geq a\}} \Psi_{\vec x}^{r}(\vec w)\Psi_{\vec x}^{\ell}(\vec w').
	\end{equation}
	Now similar as in the $m=1$ case we evaluate the sums appearing in \eqref{eq: multi-point 1} and \eqref{eq: multi-point 2} using Corollary \ref{cor: single sum arbitrary index} and Corollary \ref{cor: sum over two arbitrary index} in Section \ref{sec: summation identity} and apply Cauchy-Binet formula backwards to conclude that 
	\begin{equation}
		\mathbb{P}_{\vec y}^{(L)}\left(\bigcap_{i=1}^m \{x_{k_i}(t_i)\geq a_i\}\right) = \oint\cdots\oint \frac{\dd z_m}{2\pi \ii z_m}\cdots \frac{\dd z_1}{2\pi \ii z_1} \mathcal{C}^{(L)}(\vec z)\mathcal{D}_{\vec y}^{(L)}(\vec z),
	\end{equation}
	for $\mathcal{C}^{(L)}(\vec z)$ and $\mathcal{D}_{\vec y}^{(L)}(\vec z)$ defined in \eqref{eq: C(z)} and \eqref{eq: D(z)}. Similar as the discussion before, in order to interchange summation and integration we need the absolute convergence of all the infinite sums which is guaranteed if we assume
	\begin{equation}
		\prod_{j=1}^{N}|w_j^{(1)}+1|>\prod_{j=1}^{N}|w_j^{(2)}+1|>\cdots >\prod_{j=1}^{N}|w_j^{(m)}+1|>1.
	\end{equation}
	By the same reasoning as in $m=1$ case this can be achieved assuming the integral contours for $z_i$'s are large nested contours $|z_\ell|=r_\ell$ with $r_{\ell}-r_{\ell+1}$ also large enough for all $1\leq \ell\leq m$. Finally we can deform the integral contours in \eqref{eq: multipoint} into  arbitrary nested contours with $0$ inside, not necessarily with large radius. This is due to the analyticity of $\mathcal{C}^{(L)}(\vec z)$ and $\mathcal{D}_{\vec y}^{(L)}(\vec z)$ in $\vec z$ for any $z_i$'s nonzero and distinct, which can be shown in a similar way as (ii) of Proposition \ref{properties of the transition probability}.
\end{proof}

\subsection{Toeplitz-like determinant to Fredholm determinant}

The finite time formula obtained in Theorem \ref{thm: mutlitime distribution general} contains a factor $\mathcal{D}_{\vec y}^{(L)}(\vec z)$ inside the integrals which is a Toeplitz-like determinant and is hard to take large-time limits. In this section we rewrite the formula 
based on a remarkable finite determinantal identity obtained in \cite{baik2019general}. It is an identity between a Toeplitz-like determinant with symbol supported on a finite set and a Fredholm determinant with kernel acting on the $\ell^2$ space supported on the same finite set.

\begin{proof}[Proof of Theorem~\ref{thm: finite-time}]
	The proof is almost verbatim to the proof of Theorem 3.1 in \cite{baik2019general} so we omit most of the details. We apply Proposition 4.1 of \cite{baik2019general} with the functions $p_i(w)$, $q_i(w)$ and $h_i(w)$ given by 
	\begin{equation}
		p_i(w) = w^{i-1}(1+pw)^{N-i}(1+w)^{y_{N-i+1}+N-i+1},\quad \text{and }\ q_i(w)= w^{N-i},
	\end{equation}
for $1\leq i\leq N$ and 
\begin{equation}
	h_i(w):=\begin{cases}
		G_i(w)w^{-N}(1+pw)\quad &i=1,\\
		G_i(w)\quad &2\leq i\leq m-1,\\
		G_i(w)w^{-N}\quad &i=m.
	\end{cases}
\end{equation}
Then the Toepliz-like determinant $\mathcal{D}_{\vec y}^{(L)}(\vec z)$ can be rewritten as the product of a function and a Fredholm determinant. Combining with the definition of the function $\mathcal{C}_{\vec y}^{(L)}(\vec z)$ and using the algebraic relation satisfied by the Bethe roots we obtain the identity 
\begin{equation}
	\mathcal{C}_{\vec y}^{(L)}(\vec z)\mathcal{D}_{\vec y}^{(L)}(\vec z)= \mathscr{C}_{\vec y}^{(L)}(\vec z)\mathscr{D}_{\vec y}^{(L)}(\vec z).
\end{equation}
\end{proof}

\section{Summation identities of eigenfunctions}\label{sec: summation identity}
In this section we state and prove the summation identities used in computing the multi-point joint distribution in Section \ref{sec: multitime}. First we recall the left and right eigenfunctions defined in \eqref{eq: left eigenfunction} and \eqref{eq: right eigenfunction}. They are certain (anti)-symmetric functions appearing naturally in the transition probability formula \eqref{eq:transition probability}. Our summation identities should be viewed as Littlewood or Cauchy type identities for these symmetric functions, but over configuration spaces $x_N+L>x_1>\cdots>x_N$ (partitions wrapped on a cylinder).
\begin{definition}[Left and right eigenfunctions]\label{def: eigenfunctions}
	Given $\vec x= (x_1,\cdots,x_N)\in \conf_N^{(\infty)}$ and $p\in \mathbb{C}$, we define the functions $\Psi_{\vec x}^{\ell}(\vec w)$ and $\Psi_{\vec x}^r(\vec w)$ for $\vec w\in \mathbb{C}^N$ as follows:
	\begin{align}
		&\Psi_{\vec x}^{\ell}(\vec w) = \det\left[\left(\frac{w_i}{1+pw_i}\right)^{j-1}(1+w_i)^{x_{N-j+1}-j+1}\right]_{i,j=1}^N,\\
		&\Psi_{\vec x}^{r}(\vec w) = \prod_{i=1}^{N}(1-p\1_{x_{i-1}-x_i=1})\cdot\det\left[\left(\frac{w_i}{1+pw_i}\right)^{1-j}(1+w_i)^{-x_{N-j+1}+j-1}\right]_{i,j=1}^N. 
	\end{align}
\end{definition} 
The rest of this section is organized as follows: we list all the summation identities we need in Section \ref{sec: summation identities} and postpone the proofs to the next few subsections. The proof of the summation identity over a single eigenfunction (Proposition \ref{prop: summation single}) consists of straightforward manipulations of determinants and will be discussed in Section \ref{sec: proof single}. The summation identity over left and right eigenfunctions (Proposition \ref{prop: Cauchy identity}) is more involved and the proof will be splitted into three steps discussed in Section \ref{sec: cauchy strategy} to Section \ref{sec: cauchy finish}. Throughout the proof several rank-one perturbation formulas for Cauchy determinants are frequently used so we collect all these elementary formulas in a separate section for convenience, see Section \ref{sec: perturbation formulas} for details. Finally in Section \ref{sec: proof cor} we discuss the proof of Corollary \ref{cor: single sum arbitrary index} and \ref{cor: sum over two arbitrary index} which are simple consequences of the previous propositions and the periodic nature of the identities. 
\smallskip

\subsection{Summation identities over eigenfunctions}\label{sec: summation identities}
We start with a summation identity for $\Psi_{\vec x}^{r}(\vec w)$:
\begin{proposition}[Summation over a single eigenfunction]\label{prop: summation single}
	Let $z\in \mathbb{C}$ be nonzero. Let $\Psi_{\vec x}^{r}(\vec w)$ be as in \eqref{eq: right eigenfunction} where $\vec w=(w_1,\cdots,w_N)\in (\mathcal{S}_z)^N$ satisfies $\prod_{j=1}^{N}|w_j+1|>1$. Then 
	\begin{equation}
	\sum_{\vec x\in \conf_{N}^{(L)}\cap \{x_N\geq 0\}} \Psi_{\vec x}^{r}(\vec w) = \prod_{i=1}^{N}(1+w_i) \cdot \det[w_{i}^{-j}]_{i,j=1}^N.
	\end{equation}
\end{proposition} 

The following corollary which allows slightly more general constraints on the summation is a simple consequence of Proposition \ref{prop: summation single} and the periodic nature of $\conf_{N}^{(L)}$.
\begin{corollary}\label{cor: single sum arbitrary index}
	Under the same assumption as in Proposition \ref{prop: summation single} we have 
	\begin{equation}
		\begin{aligned}
			\sum_{\vec x\in \conf_{N}^{(L)}\cap \{x_{N-k}\geq a\}} \Psi_{\vec x}^{r}(\vec w) = (-1)^{k(N-1)}z^{kL}\prod_{i=1}^{N}\left[\left(\frac{1+pw_i}{w_i}\right)^k\cdot(1+w_i)^{-a+k+1}\right] \cdot \det[w_{i}^{-j}]_{i,j=1}^N,
		\end{aligned}
	\end{equation}
	for all $0\leq k\leq N-1$ and $a\in \mathbb{Z}$.
\end{corollary}

The above summation identities over a single eigenfunction are sufficient for computing the one-point distribution. To get the multi-point distribution we also need the following Cauchy-type summation identities over products of left and right eigenfunctions.

\begin{proposition}[Cauchy-type summation identity over left and right eigenfunctions] \label{prop: Cauchy identity}
	Let $z,z'\in \mathbb{C}$ be nonzero such that $(z')^L\neq z^L$. Let $\vec x=(x_1,\cdots,x_N)\in \conf_{N}^{(L)}$ and $\Psi_{\vec x}^{r}(\vec w)$ and $\Psi_{\vec x}^{\ell}(\vec w')$ be as in \eqref{eq: left eigenfunction} and \eqref{eq: right eigenfunction}, where $\vec w=(w_1,\cdots,w_N)\in (\mathcal{S}_z)^N$ and $\vec w'=(w_1',\cdots,w_N')\in (\mathcal{S}_{z'})^N$ satisfy  $\prod_{j=1}^{N}|w_j+1|>\prod_{j=1}^{N}|w_j'+1|$. Then 
	\begin{equation}\label{eq: Cauchy identity}
		\sum_{\vec x\in \conf_{N}^{(L)}\cap \{x_N\geq 0\}} \Psi_{\vec x}^{r}(\vec w)\Psi_{\vec x}^{\ell}(\vec w') = \left(1-\left(\frac{z'}{z}\right)^L\right)^{N-1}\prod_{j=1}^N(w_j+1)\cdot \det\left[\frac{1}{w_{i}-w_{i'}'}\right]_{i,i'=1}^N.
	\end{equation}
\end{proposition}

\begin{remark}
	It is interesting to note that the right-hand side of \eqref{eq: Cauchy identity} does not depend on $p$ explicitly (of course the $w_i$' should satisfy certain algebraic equations which depend on $p$). Taking $p\to 0$, Proposition \ref{prop: Cauchy identity} degenerates to Proposition $3.4$ of \cite{baik2019multi}, which can be understood as the Cauchy identity (a periodic version) for the Grothendieck polynomial (and its dual), and can be derived from the finite-sum Cauchy identity for Grothendieck polynomials obtained in Theorem 5.3 of \cite{motegi2013grothendieck}. For $0<p<1$, to the best of our knowledge the corresponding Cauchy-type identity \eqref{eq: Cauchy identity} has not been discussed in the existing literature, at least for the periodic case. The key point here is instead of summing over all configuration $\vec x$ with $0\leq x_N<x_{N-1}<\cdots<x_1$ as in the usual Cauchy identity, we are only summing over those configurations satisfying the extra constraint $x_1<x_{N}+L$. For general spectral parameters $\vec w$ and $\vec w'$ this sum only gives a deformed or generalized Cauchy determinant. It further reduces to a genuine Cauchy determinant when we impose the conditions as in Proposition \ref{prop: Cauchy identity} that the spectral parameters satisfy suitable Bethe equations.
\end{remark}

Similar as in Corollary \ref{cor: single sum arbitrary index}, we can easily extend Proposition \ref{prop: Cauchy identity} using periodicity to the summations over $\vec x\in \conf_{N}^{(L)}\cap \{x_{N-k}\geq a\}$ for any $0\leq k\leq N-1$ and $a\in \mathbb{Z}$:  
\begin{corollary}\label{cor: sum over two arbitrary index}
	Under the same assumption as in Proposition \ref{prop: Cauchy identity} we have 
	\begin{equation}
		\begin{aligned}
			\sum_{\vec x\in \conf_{N}^{(L)}\cap \{x_{N-k}\geq a\}} \Psi_{\vec x}^{r}(\vec w)\Psi_{\vec x}^{\ell}(\vec w') &= \left(\frac{z}{z'}\right)^{kL}\left(1-\left(\frac{z'}{z}\right)^L\right)^{N-1}\cdot \det\left[\frac{1}{w_{i}-w_{i'}'}\right]_{i,i'=1}^N\\
			&\cdot\prod_{j=1}^N\frac{w_j^{-k}(1+pw_j)^{k}(w_j+1)^{-a+k+1}}{(w_j')^{-k}(1+pw_j')^{k}(w_j'+1)^{-a+k}},		
		\end{aligned}
	\end{equation}
	for all $0\leq k\leq N-1$ and $a\in \mathbb{Z}$.
\end{corollary}
\smallskip

\subsection{Proof of Proposition \ref{prop: summation single}}\label{sec: proof single}
	First we write 
	\begin{equation}\label{eq: split the sum}
	\begin{aligned}
		\sum_{\vec x\in \conf_{N}^{(L)}\cap \{x_N\geq 0\}} \Psi_{\vec x}^{r}(\vec w) &=\sum_{a=0}^{\infty} \sum_{\vec x\in \conf_{N}^{(L)}\cap \{x_N= a\}}\Psi_{\vec x}^{r}(\vec w)\\
		&= \sum_{a=0}^{\infty}\prod_{j=1}^{N}(1+w_j)^{-a}\sum_{\vec x'\in \conf_{N}^{(L)}\cap \{x_N'= 0\}}\Psi_{\vec x'}^{r}(\vec w),
	\end{aligned}
	\end{equation}
	where $\vec x' = \vec x-(a,\cdots,a)$. The summation over $a$ converges absolutely provided that $\prod_{j=1}^{N}|1+w_j|>1$. Thus we have reduced the computation to a summation with constraint $x_N=0$ which is a finite sum so there is no convergence issue. We will see the telescoping nature of the summation over $a$.

	 The main difficulty for computing the sum comes from the $\prod_{i=1}^{N}(1-p\mathbf{1}_{x_{i-1}-x_i=1})$ factor in $\Psi_{\vec x}^{r}(\vec w)$ which gives different weights to different terms in the sum. To handle this we first split the sum according to the position of the first empty site to the right of $x_N$ in the configuration $\vec x$. More precisely we set
	 \begin{equation*}
	 	\conf_{N,k}^{(L)}:= \{\vec x\in \conf_{N}^{(L)}: x_{N}=x_{N-1}-1=\cdots x_{N-k}-k=0,  x_{N-k-1}-k-1>0\}.
	 \end{equation*}
    Intuitively $\conf_{N,k}^{(L)}$ consists of all configurations in $\conf_{N}^{(L)}\cap\{x_N=0\}$ with the first empty site to the right of $x_N$ (which equals $0$) be at $k+1$. Then clearly we have $\conf_{N}^{(L)}\cap \{x_N=0\}=\cup_{k=0}^{N-1}\conf_{N,k}^{(L)}$. Hence
	\begin{equation*}
	\begin{aligned}
	 \sum_{\vec x\in \conf_{N}^{(L)}\cap \{x_N= 0\}}\Psi_{\vec x}^{r}(\vec w) = \sum_{k=0}^{N-1}\sum_{\vec x\in \conf_{N,k}^{(L)}}\Psi_{\vec x}^{r}(\vec w):= \sum_{k=0}^{N-1}\Sigma_k.
	\end{aligned}		
	\end{equation*} 
    We will first compute the summations over $\conf_{N,k}^{(L)}$ for each $k$ and then sum over $k$. The results for each of these steps are summarized in the following two lemmas, whose proofs will be at the end of this subsection.
    
    \begin{lemma}\label{lemma: single gap sum} We have
    	\begin{equation*}
    		\Sigma_k:=\sum_{\vec x\in \conf_{N,k}^{(L)}}\Psi_{\vec x}^{r}(\vec w) = (1-p)^k\det[R_{i,j}^{(k)}(\vec w)]_{i,j=1}^{N},
    	\end{equation*}
    	where 
    	\begin{equation*}
    		R_{i,j}^{(k)}(\vec w)=
    		\begin{cases}
    			w_i^{1-j}(1+pw_i)^{j-1},\quad & 1\leq j\leq k+1,\\
    			w_{i}^{-j}(1+pw_i)^{j-1},\quad & k+2\leq j\leq N.\\
    		\end{cases}
    	\end{equation*}
    \end{lemma}
    
    \begin{lemma}\label{lemma: sum of gaps} Furthermore,
    	\begin{equation*}
    	\sum_{k=0}^{N-1}\Sigma_k = \sum_{k=0}^{N-1} (1-p)^k\det[R_{i,j}^{(k)}(\vec w)]_{i,j=1}^{N} = \prod_{i=1}^{N}(1+w_i)\cdot \det[w_i^{-j}]_{i,j=1}^{N}-\det[w_i^{-j}]_{i,j=1}^{N}.
    	\end{equation*}
    \end{lemma}
\smallskip

    Now by Lemma \ref{lemma: sum of gaps} and \eqref{eq: split the sum} we see that  
    \begin{align*}
    	\sum_{\vec x\in \conf_{N}^{(L)}\cap \{x_N\geq 0\}} \Psi_{\vec x}^{r}(\vec w)&= \sum_{a=0}^{\infty}\prod_{i=1}^{N}(1+w_i)^{-a}\left(\prod_{i=1}^{N}(1+w_i)\cdot \det[w_i^{-j}]_{i,j=1}^{N}-\det[w_i^{-j}]_{i,j=1}^{N}\right)\\
    	&= \prod_{i=1}^{N}(1+w_i)\cdot \det[w_i^{-j}]_{i,j=1}^{N}.
    \end{align*}
    This completes the proof of Proposition \ref{prop: summation single}.
\smallskip    
  
\begin{proof}[Proof of Lemma \ref{lemma: single gap sum}]
	 Note that configurations $\vec x\in \conf_{N,k}^{(L)}$ take the form $(x_1,\cdots,x_{N-k-1}, k,\cdots,0)$ where $x_{N-k-1}>k+1$. Hence for $\vec x\in \conf_{N,k}^{(L)}$ we have
	\begin{equation*}
		\Psi_{\vec x}^{r}(\vec w) =  (1-p)^k\prod_{i=1}^{N-k-1}(1-p\1_{x_{i-1}-x_{i}=1})\det[R_{i,j}^{(k,k+1)}(\vec w)]_{i,j=1}^N,
	\end{equation*}
	where 
	\begin{equation*}
		R_{i,j}^{(k,k+1)}(\vec w):=
		\begin{cases}
			w_i^{1-j}(1+pw_i)^{j-1},\quad &1\leq j\leq k+1,\\
			w_i^{1-j}(1+pw_i)^{j-1}(1+w_i)^{-x_{N-j+1}+j-1},\quad &k+2\leq j\leq N.
		\end{cases}
	\end{equation*}
   We sum over configurations in $\conf_{N,k}^{(L)}$ in the following order:
   \begin{equation*}
   	\sum_{\vec x \in \conf_{N,k}^{(L)}}\Psi_{\vec x}^{r}(\vec w) = \sum_{x_1=N}^{L-1}\sum_{x_2=N-1}^{x_1-1}\cdots\sum_{x_{N-k-1=k+2}}^{x_{N-k-2}-1}\Psi_{\vec x}^{r}(\vec w).
   \end{equation*}
	Note that
	\begin{equation*}
		\begin{aligned}
			&\sum_{x_{N-k-1}=k+2}^{x_{N-k-2}-1} (1-p\1_{x_{N-k-2}-x_{N-k-1}=1}) R_{i,k+2}^{(k+1)}(\vec w) \\
			&= w_i^{-2-k}(1+pw_i)^{1+k}- w_i^{-2-k}(1+pw_i)^{2+k}(1+w_i)^{-x_{N-k-2}+k+2}\\
			&= w_i^{-2-k}(1+pw_i)^{1+k}- R_{i,k+3}^{(k,k+1)}(\vec w).
		\end{aligned}
	\end{equation*}
	Adding the $(k+3)$-th column to the $(k+2)$-th column we get
	\begin{equation*}
		\sum_{x_{N-k-1}= k+2}^{x_{N-k-2}-1}\Psi_{\vec x}^{r}(\vec w) = (1-p)^k\prod_{i=1}^{N-k-2}(1-p\1_{x_{i-1}-x_{i}=1})\det[R_{i,j}^{(k,k+2)}(\vec w)]_{i,j=1}^N,
	\end{equation*}
	where 
	\begin{equation*}
		R_{i,j}^{(k,k+2)}(\vec w):=
		\begin{cases}
			w_i^{1-j}(1+pw_i)^{j-1},\quad & 1\leq j\leq k+1,\\
			w_{i}^{-j}(1+pw_i)^{j-1},\quad & j=k+2,\\
			w_i^{1-j}(1+pw_i)^{j-1}(1+w_i)^{-x_{N-j+1}+j-1},\quad & k+3\leq j\leq N.
		\end{cases}
	\end{equation*}
	Now we repeat this procedure and perform the sum over $x_{N-k-2},\cdots, x_2$ to get 
	\begin{equation*}
		\sum_{x_2=N-1}^{x_1-1}\cdots\sum_{x_{N-k-1}= k+2}^{x_{N-k-2}-1}\Psi_{\vec x}^{r}(\vec w) = (1-p)^k(1-p\1_{x_0-x_1=1})\det[R_{i,j}^{(k,N-1)}(\vec w)]_{i,j=1}^{N},
	\end{equation*}
	where 
	\begin{equation*}
		R_{i,j}^{(k,N-1)}(\vec w):=
		\begin{cases}
			w_i^{1-j}(1+pw_i)^{j-1},\quad & 1\leq j\leq k+1,\\
			w_{i}^{-j}(1+pw_i)^{j-1},\quad & k+2\leq j\leq N-1,\\
			w_i^{1-j}(1+pw_i)^{j-1}(1+w_i)^{-x_{N-j+1}+j-1},\quad & j= N.
		\end{cases}
	\end{equation*}
	Finally note that
	\begin{equation*}
		\begin{aligned}
			\sum_{x_{1}=N}^{L-1} (1-p\1_{x_{0}-x_{1}=1}) R_{i,N}^{(k,N-1)}(\vec w) &= w_i^{-N}(1+pw_i)^{N-1}- w_i^{-N}(1+pw_i)^{N}(1+w_i)^{-L-N}\\
			&= w_i^{-N}(1+pw_i)^{N-1}- z^{-L}R_{i,1}^{(k,N-1)}(\vec w),
		\end{aligned}
	\end{equation*}
	where we used the fact that $w_i\in \mathcal{S}_z$. Thus we conclude that
	\begin{equation*}
		\Sigma_k:=\sum_{x_1=N}^{L-1}\cdots\sum_{x_{N-k-1}= k+2}^{x_{N-k-2}-1} \Psi_{\vec x}^{r}(\vec w) = (1-p)^k\det[R_{i,j}^{(k)}(\vec w)]_{i,j=1}^{N},
	\end{equation*}
	where 
	\begin{equation*}
		R_{i,j}^{(k)}(\vec w)=
		\begin{cases}
			w_i^{1-j}(1+pw_i)^{j-1},\quad & 1\leq j\leq k+1,\\
			w_{i}^{-j}(1+pw_i)^{j-1},\quad & k+2\leq j\leq N.\\
		\end{cases}
	\end{equation*}
\end{proof}

\begin{proof}[Proof of Lemma \ref{lemma: sum of gaps}]
	For the purpose of summing over $k$ we further rewrite $R_{i,j}^{(k)}(\vec w)$ slightly. Given $0\leq k\leq N-3$, we add the $j$-th column of $R^{(k)}(\vec w)$ to its $(j+1)$-th column, $j=N,\cdots, k+2$ so that 
	\begin{equation*}
		\det[R_{i,j}^{(k)}(\vec w)]_{i,j=1}^{N} = \det[\hat{R}_{i,j}^{(k)}(\vec w)]_{i,j=1}^{N},
	\end{equation*}
	where 
	\begin{equation*}
		\hat{R}_{i,j}^{(k)}(\vec w):=
		\begin{cases}
			w_i^{1-j}(1+pw_i)^{j-1},\quad & 1\leq j\leq k+1,\\
			w_{i}^{-j}(1+pw_i)^{j-1},\quad & j=k+2,\\
			w_{i}^{-j}(1+pw_i)^{j-1}(1+w_i),\quad & k+3\leq j\leq N.\\
		\end{cases}
	\end{equation*}
	For $k=N-2,N-1$ we just set $\hat{R}_{i,j}^{(k)}(\vec w) =R_{i,j}^{(k)}(\vec w)$. Now we perform the sum over $k$ in the order $k=N-1,N-2,\cdots,0$. Note that for each $0\leq k\leq N-1$, $\hat{R}_{i,j}^{(k)}(\vec w) = \hat{R}^{(k+1)}_{i,j}(\vec w)$ except for $j=k+1$.  Hence by multi-linearity of the determinants we have
	\begin{equation*}
		\begin{aligned}
			\Sigma_{N-1}+\Sigma_{N-2} &= (1-p)^{N-2} \left((1-p)\det[\hat{R}_{i,j}^{(N-1,N)}(\vec w)]+\det[\hat{R}_{i,j}^{(N-2)}(\vec w)]\right)\\
			&= (1-p)^{N-2}\det[T_{i,j}^{(N-2)}(\vec w)]_{i,j=1}^{N},
		\end{aligned}
	\end{equation*}
	where 
	\begin{equation*}
		T_{i,j}^{(N-2)}(\vec w):=\begin{cases}
			w_i^{1-j}(1+pw_i)^{j-1},\quad & 1\leq j\leq N-1,\\
			w_{i}^{-j}(1+pw_i)^{j-2}(1+w_i),&j= N.\\
		\end{cases}
	\end{equation*}
	Here to simplify the expression we have multiplied the $(N-1)$-th column by $-p(1-p)$ and added it to the sum of the $N$-th column of $\hat{R}^{(N-1)}(\vec w)$ and $\hat{R}^{(N-2)}(\vec w)$, using the simple identity that 
	\begin{equation*}
		\left[(1-p)w_i+1\right]\cdot w_i^{-N}(1+pw_i)^{N-1} = w_i^{-N}(1+pw_i)^{N-2}(1+w_i)+p(1-p)w_i^{2-N}(1+pw_i)^{N-2}.
	\end{equation*}
	Repeating this procedure we get 
	\begin{equation*}
		\sum_{k=0}^{N-1} \Sigma_k = \det[T^{(0)}_{i,j}(\vec w)]_{i,j=1}^N,
	\end{equation*}
	where 
	\begin{equation*}
		T_{i,j}^{(0)}(\vec w):=\begin{cases}
			w_i^{1-j}(1+pw_i)^{j-1},\quad & j=1,\\
			w_{i}^{-j}(1+pw_i)^{j-2}(1+w_i),&2\leq j\leq N.\\
		\end{cases}
	\end{equation*}
	Multiplying the $j$-th column of $T^{(0)}(\vec w)$ by $-p$ and adding to the $(j+1)$-th column for $j=N-1,\cdots,2$,  we get 
	\begin{equation*}
		\sum_{k=0}^{N-1} \Sigma_k = \det[\hat{T}^{(0)}_{i,j}(\vec w)]_{i,j=1}^N =\prod_{i=1}^{N}(1+w_i)\cdot \det[w_i^{-j}]_{i,j=1}^{N}-\det[w_i^{-j}]_{i,j=1}^{N},
	\end{equation*}
	where 
	\begin{equation*}
		\hat{T}_{i,j}^{(0)}(\vec w):=\begin{cases}
			w_i^{1-j},\quad & j=1,\\
			w_{i}^{-j}(1+w_i),&2\leq j\leq N.\\
		\end{cases}
	\end{equation*}
 This completes the proof of Lemma \ref{lemma: sum of gaps}.
\end{proof}

\subsection{Proof of Proposition \ref{prop: Cauchy identity}: Strategy}\label{sec: cauchy strategy}

 The proof of Proposition \ref{prop: Cauchy identity} is rather lengthy so we divide it into three steps and discuss them one by one in the next few subsections. The proof mainly follows the proof strategy of Proposition 3.4 of \cite{baik2019multi} but there are several new technicalities. The non-uniform term $\prod_{i=1}^{N}(1-p\1_{x_{i-1}-x_i=1})$ appearing in $\Psi_{\vec x}^r(\vec w)$ leads to extra difficulty for the sum. In Step 1 we overcome this by introducing a different way (and slightly more convenient way in our opinion) of expressing the sum in \eqref{eq: Cauchy identity} comparing to the proof in \cite{baik2019multi}, see Lemma \ref{lemma: summation over all configurations} for details. In Step 2 we establish a key summation identity (see Lemma \ref{lemma: Cauchy identity for infinite sum}) which generalizes Lemma 5.4 in \cite{baik2019multi} while the computation is more delicate. Finally in Step 3 we combine the formula obtained in Step 1 and the summation identity obtained in Step 2 to conclude the final result. 

\subsection{Proof of Proposition \ref{prop: Cauchy identity}: Step 1}
Similar as in the proof of Proposition \ref{prop: summation single} we first write
\begin{equation*}
\begin{aligned}
	\sum_{\vec x\in \conf_{N}^{(L)}\cap \{x_N\geq 0\}} \Psi_{\vec x}^{r}(\vec w)\Psi_{\vec x}^{\ell}(\vec w') &= \sum_{a= 0}^{\infty} \prod_{j=1}^{N}\left(\frac{1+w_j'}{1+w_j}\right)^a \sum_{\vec x\in \conf_{N}^{(L)}\cap \{x_N=0\}} \Psi_{\vec x}^{r}(\vec w)\Psi_{\vec x}^{\ell}(\vec w')\\
	&:= \sum_{a= 0}^{\infty} \prod_{j=1}^{N}\left(\frac{1+w_j'}{1+w_j}\right)^a\cdot \mathcal{H}_N(\vec w,\vec w'),
\end{aligned}
\end{equation*}
 so that it suffices to compute the sum over $\vec x\in \conf_{N}^{(L)}\cap \{x_N=0\}$, which is a finite sum so there is no convergence issue. The summation over $a$ converges absolutely by our assumption that $\prod_{j=1}^{N}|w_j+1|>\prod_{j=1}^{N}|w_j'+1|$. Expanding the determinants in $\Psi_{\vec x}^{\ell}(\vec w)$ and $\Psi_{\vec x}^{r}(\vec w')$  we get 
\begin{equation*}
\begin{aligned}
     \mathcal{H}_N(\vec w,\vec w')= \sum_{\sigma,\sigma'\in S_N} \sgn(\sigma\sigma')\prod_{j=1}^{N}\left(\frac{w'_{\sigma'(j)}(1+pw_{\sigma(j)})}{w_{\sigma(j)}(1+pw'_{\sigma'(j)})}\right)^{j-1} \mathcal{H}^{(N)}_{\sigma,\sigma'}(\vec w,\vec w'),
\end{aligned}
\end{equation*}
where 
\begin{equation}
	\mathcal{H}^{(N)}_{\sigma,\sigma'}(\vec w,\vec w') := \sum_{L=x_0>x_1>\cdots>x_N=0} \prod_{j= 1}^{N} (1-p\1_{x_{j-1}-x_{j}=1})\left(\frac{w'_{\sigma'(j)}+1}{w_{\sigma(j)}+1}\right)^{x_{N-j+1}-j+1}.
\end{equation}
   By Lemma \ref{lemma: summation over all configurations} below we have $\mathcal{H}^{(N)}_{\sigma,\sigma'}(\vec w,\vec w')$ equals
 \begin{equation*}
 \begin{aligned}
\sum_{k=1}^{N}\left[\prod_{i=2}^{k}\left(-p+\frac{1}{1-\prod_{j=i}^{k}\frac{w_{\sigma'(j)}'+1}{w_{\sigma(j)}+1}}\right)\cdot \prod_{i=k+1}^{N}\left(\frac{w_{\sigma'(i)}'+1}{w_{\sigma(i)}+1}\right)^{L-N}\cdot \left(-p+\frac{1}{1-\prod_{j=k+1}^{i}\frac{w_{\sigma(j)}+1}{w_{\sigma'(j)}'+1}}\right)\right].
 \end{aligned}
 \end{equation*}
 Hence 
 \begin{equation}\label{eq: reduction of the sum}
 \begin{aligned}
 &\mathcal{H}_N(\vec w,\vec w')= \sum_{k=1}^{N}\sum_{\sigma,\sigma'\in S_N} \sgn(\sigma\sigma')\prod_{j=1}^{N}\left(\frac{w'_{\sigma'(j)}(1+pw_{\sigma(j)})}{w_{\sigma(j)}(1+pw'_{\sigma'(j)})}\right)^{j-1}\\
 &\cdot\left[\prod_{i=2}^{k}\left(-p+\frac{1}{1-\prod_{j=i}^{k}\frac{w_{\sigma'(j)}'+1}{w_{\sigma(j)}+1}}\right)\cdot \prod_{i=k+1}^{N}\left(\frac{w_{\sigma'(i)}'+1}{w_{\sigma(i)}+1}\right)^{L-N}\cdot \left(-p+\frac{1}{1-\prod_{j=k+1}^{i}\frac{w_{\sigma(j)}+1}{w_{\sigma'(j)}'+1}}\right)\right].
 \end{aligned}
 \end{equation}

\begin{lemma}\label{lemma: summation over all configurations}
	Let $N\geq 1$ be an integer and  $w_1,\cdots,w_N$ and $w_1',\cdots,w_N'$ be distinct complex numbers. Then for any integer $L>N$ we have 
	\begin{equation}\label{eq: summation of all configurations}
    \begin{aligned}
    &\sum_{L=x_0>x_1>\cdots>x_N=0} \prod_{i=1}^{N}(1-p\1_{x_{i-1}-x_i=1})\prod_{i=1}^{N}\left(\frac{w_i'+1}{w_i+1}\right)^{x_{N-i+1}-i+1}\\
    &= \sum_{k=1}^{N}\left[\prod_{i=2}^{k}\left(-p+\frac{1}{1-\prod_{j=i}^{k}\frac{w_j'+1}{w_j+1}}\right)\cdot \prod_{i=k+1}^{N}\left(\frac{w_i'+1}{w_i+1}\right)^{L-N}\cdot \left(-p+\frac{1}{1-\prod_{j=k+1}^{i}\frac{w_j+1}{w_j'+1}}\right)\right].
    \end{aligned}
	\end{equation}
	Here any empty product is set to be $1$.   
\end{lemma}

\begin{proof}
    We use an induction on $N$. For $N=1$ the identity is obvious. Assume now $N\geq 2$ and the identity holds for all indices less than $N$.  We split the sum into two sums depending on whether $x_1=L-1$ or not:
    \begin{equation*}
    	\sum_{L=x_0>x_1>\cdots>x_N=0} = \sum_{L-1=x_1>\cdots>x_N=0} + \sum_{L-1>x_1>\cdots>x_N=0}:= T_1 + T_2.
    \end{equation*}
    For $T_1$ we first relabel the indices so that $(x_0',x_1',\cdots,x_{N-1}'):=(x_1,x_2,\cdots,x_N)$ and $L':=L-1$. Then by induction hypothesis we have 
    \begin{equation}\label{eq: summation of all configurations 1}
    \begin{aligned}
    T_1 &= (1-p)\left(\frac{w_N'+1}{w_N+1}\right)^{L-N}\sum_{L'=x_0'>x_1'>\cdots>x_{N-1}'=0} \prod_{i=1}^{N-1}(1-p\1_{x'_{i-1}-x'_i=1})\prod_{i=1}^{N-1}\left(\frac{w_i'+1}{w_i+1}\right)^{x_{N-i}'-i+1}\\
    &= \ \sum_{k=1}^{N-1}\left[\prod_{i=2}^{k}\left(-p+\frac{1}{1-\prod_{j=i}^{k}\frac{w_j'+1}{w_j+1}}\right)\cdot \prod_{i=k+1}^{N}\left(\frac{w_i'+1}{w_i+1}\right)^{L-N}\cdot T_1^{(k)}(\vec w,\vec w')\right].
    \end{aligned}
    \end{equation} 
    Where 
    \begin{equation*}
    	T_1^{(k)}(\vec w,\vec w'):= (1-p)\prod_{i=k+1}^{N-1}\left(-p+\frac{1}{1-\prod_{j=k+1}^{i}\frac{w_j+1}{w_j'+1}}\right),
    \end{equation*}
    for $1\leq k\leq N-1$ and we set $T_1^{(N)}(\vec w,\vec w'):=0$ for convenience. 
    
    For $T_2$ we calculate the sum directly in the order $L-1>x_1>\cdots>x_N=0$ using Lemma \ref{lemma: multi geometric sum} below which gives:
    \begin{equation}\label{eq: summation of all configurations 2}
    \begin{aligned}
     T_2 &= \sum_{x_{N-1}=1}^{L-N-1}\cdots \sum_{x_1=x_2+1}^{L-2} \prod_{i=2}^{N}(1-p\1_{x_{i-1}-x_i=1})\prod_{i=1}^{N}\left(\frac{w_i'+1}{w_i+1}\right)^{x_{N-i+1}-i+1}\\
     &= \sum_{k=1}^{N}\left[\prod_{i=2}^{k}\left(-p+\frac{1}{1-\prod_{j=i}^{k}\frac{w_j'+1}{w_j+1}}\right)\cdot\prod_{i=k+1}^{N}\left(\frac{w_i'+1}{w_i+1}\right)^{L-N}\cdot T_{2}^{(k)}(\vec w,\vec w')\right].
    \end{aligned}
    \end{equation}
    Where
    \begin{equation*}
    \begin{aligned}
    &\quad\  T_{2}^{(k)}(\vec w,\vec w') \\
    &:= \sum_{\ell=1}^{N-k}\sum_{k+1=s_1<\cdots<s_{\ell}\leq N} \prod_{i=1}^{\ell}\left(\frac{1}{\left(\prod_{j= s_{i}}^{s_{i+1}-1}\frac{w_j'+1}{w_j+1}\right)-1}\cdot \prod_{j=s_i+1}^{s_{i+1}-1}\left(-p+\frac{1}{1-\prod_{m=j}^{s_{i+1}-1}\frac{w_m'+1}{w_m+1}}\right)\right),
    \end{aligned}
    \end{equation*}
    for $1\leq k\leq N-1$ and $T_2^{(N)}(\vec w,\vec w'):= 1$.  Here we are summing over all possible partitions of $\{k+1,\cdots, N\}$ and  $s_{i+1}:= N+1$. Now comparing \eqref{eq: summation of all configurations 1} and \eqref{eq: summation of all configurations 2} with \eqref{eq: summation of all configurations} we see that it suffices to prove 
    \begin{equation}\label{eq: summation of all configurations reduction 1}
    	T_1^{(k)}(\vec w,\vec w')+T_2^{(k)}(\vec w,\vec w')=\prod_{i=k+1}^{N}\left(-p+\frac{1}{1-\prod_{j=k+1}^{i}\frac{w_j+1}{w_j'+1}}\right),
    \end{equation}
    for all $1\leq k\leq N$. Using the simple identity
    \begin{equation}
    	\frac{1}{1-\prod_{j=m}^{n}\frac{w_j+1}{w_j'+1}}+\frac{1}{1-\prod_{j=m}^{n}\frac{w_j'+1}{w_j+1}}=1,
    \end{equation} 
    \eqref{eq: summation of all configurations reduction 1} is further reduced to showing 
    \begin{equation}
    	T_k^{(2)}(\vec w,\vec w')= \prod_{i=k+1}^{N-1}\left(-p+\frac{1}{1-\prod_{j=k+1}^{i}\frac{w_j+1}{w_j'+1}}\right)\cdot \frac{1}{\prod_{j=k+1}^{N}\frac{w_j'+1}{w_j+1}-1},
    \end{equation}
    which follows from Lemma \ref{lemma: simplification of T_2(k)} below by taking $z_j= \frac{w_j'+1}{w_j+1}$ and properly shifting the indices.
\end{proof}

\begin{lemma}\label{lemma: multi geometric sum}
	For complex numbers $f_j$, set 
	\begin{equation}
	F_{m,n} = \prod_{j=m}^{n} f_j\quad \text{for }1\leq m\leq n.
	\end{equation}
	Then 
	\begin{equation}
	\begin{aligned}
	&\sum_{x_{N-1}=1}^{L-N-1}\cdots\sum_{x_{1}= x_{2}+1}^{L-2} \prod_{j= 2}^{N} (1-p\1_{x_{N-j+1}-x_{N-j+2}=1})(f_j)^{x_{N-j+1}-j+1}\\
	&= \sum_{k=1}^{N-1}\left[\prod_{i=2}^{k}\left(-p+\frac{1}{1-F_{i,k}}\right)\right]\cdot (F_{k+1,N})^{L-N}\\
	&\cdot\left[\sum_{\ell=1}^{N-k}\sum_{k<s_1<\cdots<s_{\ell}\leq N}\prod_{i=1}^{\ell} \left(\prod_{j=s_i+1}^{s_{i+1}-1}\left(-p+\frac{1}{1-F_{j,s_{i+1}-1}}\right)\cdot \frac{1}{(F_{s_i,s_{i+1}-1})-1}\right)\right].
	\end{aligned}
	\end{equation}
	Where we set $s_{\ell+1}=N+1$. 
\end{lemma}

\begin{proof}
This lemma is a slightly modified version of Lemma 5.3 of \cite{baik2019multi}. The proof is an elementary computation of geometric sums  and almost identical to the proof in \cite{baik2019multi} so we omit it.
\end{proof}

\begin{lemma}\label{lemma: simplification of T_2(k)}
	Let $n\geq 1$ be an integer and $z_1,\cdots,z_n$ be complex numbers such that $\prod_{j=\ell}^{k}z_j\neq 1$ for all $1\leq \ell\leq k\leq n$. Then 
\begin{equation}\label{eq: reduction of the sum over partitions}
\begin{aligned}
     &\sum_{\ell=1}^{n}\sum_{1= s_1<\cdots<s_{\ell}\leq n} \prod_{i=1}^{\ell}\left(\frac{1}{\left(\prod_{j= s_{i}}^{s_{i+1}-1}z_j\right)-1}\cdot \prod_{j=s_i+1}^{s_{i+1}-1}\left(-p+\frac{1}{1-\prod_{m=j}^{s_{i+1}-1}z_m}\right)\right)\\
     &= \prod_{i=1}^{n-1}\left(-p+\frac{1}{1-\prod_{j=1}^{i}z_j^{-1}}\right)\cdot \frac{1}{\left(\prod_{j=1}^{n}z_j\right)-1}.
\end{aligned}
\end{equation}
\end{lemma}

\begin{proof}
	Fixing an integer $M>n$, consider the following sum: 
	\begin{equation*}
		\sum_{M>x_1>\cdots>x_n\geq 0} \prod_{i=2}^{n}(1-p\1_{x_{i-1}-x_i}=1) \prod_{j=1}^{n} z_j^{x_{n-j+1}-j+1}.
	\end{equation*}
	We calculate the sum in two different ways: from right to left or from left to right. Namely we set 
	\begin{equation*}
		S^{r\to \ell}_{M} := \sum_{x_n=0}^{M-n}\cdots \sum_{x_1=x_2+1}^{M-1}\prod_{i=2}^{n}(1-p\1_{x_{i-1}-x_i}=1) \prod_{j=1}^{n} z_j^{x_{n-j+1}-j+1},
	\end{equation*}
	and 
	\begin{equation*}
	S^{\ell\to r}_M := \sum_{x_1=n-1}^{M-1}\cdots \sum_{x_n=0}^{x_{n-1}-1}\prod_{i=2}^{n}(1-p\1_{x_{i-1}-x_i}=1) \prod_{j=1}^{n} z_j^{x_{n-j+1}-j+1}.
	\end{equation*}
	Then clearly $S^{r\to \ell}_M=S^{\ell\to r}_M $. Now the two sums are calculated by calculating the (almost) geometric sums one by one either from left to right or vice versa. There are $2^n$ terms in total for both sums since every single term produces two terms after performing the geometric sum once. Each of the terms contains of a factor of the form $\prod_{j= n-k+1}^{n} z_j^{M-n+1}$ for some $0\leq k\leq n$ where $k=0$ corresponds to the terms containing no such factor. For each $k$ we combine all the terms with the same factor $\prod_{j= n-k+1}^{n} z_j^{M-n+1}$ and write 
	\begin{equation*}
		S^{r\to \ell}_M = \sum_{k=0}^{n} \left[C^{r\to \ell}_{k}(z_1,\cdots,z_n)\cdot\prod_{j= n-k+1}^{n} z_j^{M-n+1}\right],
	\end{equation*}
	and 
	\begin{equation*}
	S^{\ell\to r}_M = \sum_{k=0}^{n} \left[C^{\ell\to r}_{k}(z_1,\cdots,z_n) \cdot\prod_{j= n-k+1}^{n} z_j^{M-n+1}\right].
	\end{equation*}
	Where the coefficients $C^{r\to \ell}_{k}(z_1,\cdots,z_n)$'s and $C^{\ell\to r}_{k}(z_1,\cdots,z_n)$'s are some very explicit functions in $z_1,\cdots,z_n$ independent of $M$ which are analytic for all $z_i$'s satisfying the assumption that $\prod_{j=\ell}^{k} z_j\neq 1$ for all $1\leq k\leq \ell\leq n$. In particular it is straightforward to check (see Lemma \ref{lemma: multi geometric sum} for example) that 
	\begin{equation*}
		\text{LHS of }\eqref{eq: reduction of the sum over partitions} = C^{r\to \ell}_{n}(z_1,\cdots,z_n),\quad \text{RHS of }\eqref{eq: reduction of the sum over partitions} = C^{\ell\to r}_{n}(z_1,\cdots,z_n).
	\end{equation*} 
	We claim that $C^{r\to \ell}_{k}(z_1,\cdots,z_n)= C^{\ell\to r}_{k}(z_1,\cdots,z_n)$ for all $0\leq k\leq n$. This in particular implies \eqref{eq: reduction of the sum over partitions}. Due to analyticity it suffices to check this for the $z_j$'s satisfying $|z_j|<1$ for all $1\leq j\leq n$.  In this case by letting $M\to \infty$ in the equality $S^{r\to \ell}_M=S^{\ell\to r}_M$ we see 
	$C_0^{r\to \ell}= C_0^{\ell\to r}$. Similarly
	\begin{equation*}
		C_1^{r\to \ell} = \lim_{M\to \infty} z_n^{-M+n-1}\cdot (S_M^{r\to \ell}-C_0^{r\to \ell}) = \lim_{M\to \infty} z_n^{-M+n-1}\cdot (S_M^{\ell\to r}-C_0^{\ell\to r}) = C_1^{\ell\to r}.
	\end{equation*}
	Repeat this procedure we see $C_k^{r\to \ell}= C_k^{\ell\to r}$ for all $0\leq k\leq n$.
\end{proof}

\smallskip

\subsection{Proof of Proposition \ref{prop: Cauchy identity}: Step 2}

In this section we simplify the sum \eqref{eq: reduction of the sum}.
 We rewrite the sum further by first choosing two index sets $J$ and $J'$ with $|J|=|J'|=k$ and then expressing the sum in terms of summation over index sets $J,J'$:

\begin{equation}\label{eq: reduction of S_N}
\mathcal{H}_N(\vec w,\vec w') = \sum_{k=1}^N\sum_{\substack{J,J'\subset \{1,\cdots,N\}\\ |J|=|J'|=k}} (-1)^{\#(J,J^c)+\#(J',(J')^c)} \mathscr{H}_1(J,J') \mathscr{H}_2(J^c,(J')^c),
\end{equation}
where $\#(J,J^c):= |\{(m,n)\in (J,J^c): m>n\}|$ and similar for $\#(J',(J')^c)$. The functions $\mathscr{H}_1(J,J')$ and $\mathscr{H}_2(J^c,(J')^c)$ are defined as follows:
\begin{equation}
\mathscr{H}_1(J,J') := \sum_{\substack{\sigma:\{1,\cdots,k\}\to J\\ \sigma':\{1,\cdots,k\}\to J'}}\sgn(\sigma)\sgn(\sigma')\prod_{i=2}^{k} \left(\frac{w'_{\sigma'(i)}(1+pw_{\sigma(i)})}{w_{\sigma(i)}(1+pw'_{\sigma'(i)})}\right)^{i-1} \left(-p+\frac{1}{1-\prod_{j=i}^{k}\frac{w_{\sigma'(j)}'+1}{w_{\sigma(j)}+1}}\right),
\end{equation}
and 
\begin{equation}\label{eq: H_2}
\begin{aligned}
\mathscr{H}_2(J^c,(J')^c) &:= \sum_{\substack{\pi:\{k+1,\cdots,N\}\to J^c\\ \pi':\{k+1,\cdots,N\}\to (J')^c}}\sgn(\pi)\sgn(\pi')\frac{\prod_{j'\in (J')^c}(w_{j'}'+1)^{L-N}}{\prod_{j\in J^c}(w_j+1)^{L-N}}\\
&\cdot \prod_{i=k+1}^{N} \left(\frac{w'_{\pi'(i)}(1+pw_{\pi(i)})}{w_{\pi(i)}(1+pw'_{\pi'(i)})}\right)^{i-1}\cdot \left(-p+\frac{1}{1-\prod_{j=k+1}^{i}\frac{w_{\pi(j)}+1}{w_{\pi'(j)}'+1}}\right).
\end{aligned}
\end{equation}
 By Lemma \ref{lemma: Cauchy identity for infinite sum} below (which is of interest on its own) we have 
\begin{equation}\label{eq: reduction of H_1}
	\mathscr{H}_1(J,J') = \left[\prod_{j\in J}(w_j+1)-\prod_{j'\in J'}(w_{j'}'+1)\right]\cdot \det\left[\frac{1}{w_j-w_{j'}'}\right]_{j\in J, j'\in J' }.
\end{equation}
To simplify $\mathscr{H}_2(J^c,(J')^c)$ we use the assumption that $\vec w\in (\mathcal{S}_z)^N$ and $\vec w'\in (\mathcal{S}_{z'})^N$. Namely for all $1\leq i,i'\leq N$ we have 
\begin{equation}\label{eq: Bethe equations}
	w_i^N(1+pw_i)^{-N}(1+w_i)^{L-N}=z^{L},\quad (w_{i'}')^N(1+pw_{i'}')^{-N}(1+w_{i'}')^{L-N}=(z')^{L}.
\end{equation} 
Inserting \eqref{eq: Bethe equations} into \eqref{eq: H_2} we get
\begin{equation*}
\begin{aligned}
H_2(J^c,(J')^c) &= \sum_{\substack{\pi:\{k+1,\cdots,N\}\to J^c\\ \pi':\{k+1,\cdots,N\}\to (J')^c}}\sgn(\pi)\sgn(\pi')\frac{\prod_{j'\in (J')^c}(z')^{L}}{\prod_{j\in J^c}z^{L }}\\
&\cdot \prod_{i=k+1}^{N} \left(\frac{w_{\pi(i)}(1+pw_{\pi'(i)}')}{w_{\pi'(i)}'(1+pw_{\pi(i)})}\right)^{N-i+1}\cdot \left(-p+\frac{1}{1-\prod_{j=k+1}^{i}\frac{w_{\pi(j)}+1}{w_{\pi'(j)}'+1}}\right).
\end{aligned}
\end{equation*}
Now in order to apply Lemma \ref{lemma: Cauchy identity for infinite sum} we reflect the permutations by defining $\hat{\pi}:\{1,\cdots,N-k\}\to J^c$ as $\hat{\pi}(j):=\pi(N-j+1)$ for $1\leq j\leq N$ and similarly for $\hat{\pi}':\{1,\cdots,N-k\} \to (J')^c$.  Then
\begin{equation}\label{eq: reduction of H_2}
\begin{aligned}
\mathscr{H}_2(J^c,(J')^c) &= \sum_{\substack{\hat{\pi}:\{1,\cdots,N-k\}\to J^c\\ \hat{\pi}':\{1,\cdots,N-k\}\to (J')^c}}\sgn(\hat{\pi})\sgn(\hat{\pi}')\frac{\prod_{j'\in (J')^c}(z')^{L}}{\prod_{j\in J^c}z^{L }}\\
&\cdot \prod_{i=1}^{N-k} \left(\frac{w_{\hat{\pi}(i)}(1+pw_{\hat{\pi}'(i)}')}{w_{\hat{\pi}'(i)}'(1+pw_{\hat{\pi}(i)})}\right)^{i}\cdot \left(-p+\frac{1}{1-\prod_{j=i}^{N-k}\frac{w_{\hat{\pi}(j)}+1}{w_{\hat{\pi}'(j)}'+1}}\right).
\end{aligned}
\end{equation}
Now apply Lemma \ref{lemma: Cauchy identity for infinite sum} again with the role of $\vec w$ and $\vec w'$ exchanged we have 
\begin{equation}
\begin{aligned}\label{eq: further reduction of H_2}
\mathscr{H}_2(J^c,(J')^c) &= \frac{\prod_{j'\in (J')^c}(z')^{L}}{\prod_{j\in J^c}z^{L }}\cdot \frac{\prod_{j\in J^c} w_j/(1+pw_j)}{\prod_{j'\in (J')^c}w_{j'}'/(1+pw_{j'}')}\cdot \left(-p+\frac{1}{1-\frac{\prod_{j\in J^c}(w_j+1)}{\prod_{j'\in (J')^c}(w_{j'}'+1)}}\right)\\
&\cdot \left[\prod_{j'\in (J')^c}(w_{j'}'+1)-\prod_{j\in J^c}(w_{j}+1)\right]\cdot \det\left[\frac{1}{-w_j+w_{j'}'}\right]_{j\in J^c, j'\in (J')^c }.
\end{aligned}
\end{equation}
Here the extra factor in \eqref{eq: further reduction of H_2} comes from the fact that in \eqref{eq: reduction of H_2} the product starts from $i=1$ instead of $i=2$ and the exponent is $i$ instead of $i-1$ comparing to \eqref{eq: Cauchy identity for infinite sum}. This completes Step 2 of the proof.

\begin{lemma}\label{lemma: Cauchy identity for infinite sum}
	Let $n\in \mathbb{N}$. Given any complex numbers $w_i$ and $w_{i'}'$, $i=1,\cdots,n$ such that $w_i\neq w_{i'}$,  for all $1\leq i,i'\leq n$. Then for any $0< p<1$ we have 
	\begin{equation}\label{eq: Cauchy identity for infinite sum}
	\begin{aligned}
	&\sum_{\sigma,\sigma'\in S_n} \sgn(\sigma\sigma')\prod_{j=1}^{n}\left(\frac{w'_{\sigma'(j)}(1+pw_{\sigma(j)})}{w_{\sigma(j)}(1+pw'_{\sigma'(j)})}\right)^{j-1}
	\prod_{j=2}^{n}\left[-p+\frac{1}{1-\prod_{i=j}^{n}\frac{w'_{\sigma'(i)}+1}{w_{\sigma(i)}+1}}\right]\\
	&= \left[\prod_{j=1}^{n}(w_j+1)-\prod_{j=1}^{n}(w_j'+1)\right] \det\left[\frac{1}{w_i-w_{i'}'}\right]_{i,i'=1}^n.
	\end{aligned}
	\end{equation}
\end{lemma}
\begin{remark}
	Equation \eqref{eq: Cauchy identity for infinite sum} should  be understood also as a Cauchy summation identity (simpler version than Proposition \ref{prop: Cauchy identity}, for summation over all partitions $\lambda=(x_1-n+1, x_2-n+2,\cdots,x_n)$ with at most $n-1$ rows) of the symmetric funtions $\Psi_{\vec x}^{r}(\vec w)/\Delta(\vec w)$ and $\Psi_{\vec x}^{\ell}(\vec w')/\Delta(\vec w')$, where $\Delta(\vec w)$ is the usual Vandermonde determinant. In fact formally we have 
	\begin{equation}
		\sum_{0=x_n<x_{n-1}<\cdots<x_1} \Psi_{\vec x}^{r}(\vec w)\Psi_{\vec x}^{\ell}(\vec w') = \sum_{x_{1}=x_2+1}^{\infty}\cdots \sum_{x_{n-1}=1}^{\infty} \Psi_{\vec x}^{r}(\vec w)\Psi_{\vec x}^{\ell}(\vec w') = \text{LHS of }\eqref{eq: Cauchy identity for infinite sum},
	\end{equation}
	assuming all the infinite geometric series converge absolutely. 
\end{remark}

\begin{proof}[Proof of Lemma \ref{lemma: Cauchy identity for infinite sum}]
	The proof is based on induction on $n$. The main tools are several rank-one perturbation formulas for the Cauchy determinants. Since we will use them several times we will collect them in a separate section, see Section \ref{sec: perturbation formulas}. For $n=1$, \eqref{eq: Cauchy identity for infinite sum} is trivial. Let $n\geq 2$ and assume \eqref{eq: Cauchy identity for infinite sum} is true for all indices $\leq n-1$. Given $\sigma,\sigma'\in S_n$ we first choose two indices $\ell=\sigma(1)$ and $\ell'=\sigma'(1)$ and shift the restriction of $\sigma$ and $\sigma'$ on $\{2,\cdots,n\}$ by $1$ but still denote them by $\sigma$ and $\sigma'$. Then 
	\begin{equation*}
		\begin{aligned}
		&\text{LHS of }\eqref{eq: Cauchy identity for infinite sum} = \sum_{\ell,\ell'=1}^{n} (-1)^{\ell+\ell'}\frac{w_{\ell}(1+pw_{\ell'}')}{w_{\ell'}'(1+pw_{\ell})}\prod_{k=1}^{n} \frac{w_k'(1+pw_{k})}{w_{k}(1+pw_k')}\cdot \left(-p+\frac{1}{1-\frac{w_{\ell}+1}{w_{\ell'}'+1}\prod_{k=1}^{n}\frac{w_{k}'+1}{w_{k}+1}}\right)\\
		&\cdot\left[\sum_{\substack{\sigma:\{1,\cdots,n-1\}\to \{1,\cdots,n\}\backslash\{\ell\}\\ \sigma':\{1,\cdots,n-1\}\to \{1,\cdots,n\}\backslash\{\ell'\}}} \sgn(\sigma)\sgn(\sigma')\prod_{i=2}^{n-1}\left(\frac{w'_{\sigma'(i)}(1+pw_{\sigma(i)})}{w_{\sigma(i)}(1+pw'_{\sigma'(i)})}\right)^{i-1}\cdot\prod_{j=2}^{n-1}\left(-p+\frac{1}{1-\prod_{i=j}^{n-1}\frac{w'_{\sigma'(i)}+1}{w_{\sigma(i)}+1}}\right)\right].
		\end{aligned}
	\end{equation*}
	By the induction hypothesis the term inside the big bracket above equals 
	\begin{equation*}
		\left[\prod_{1\leq j\leq n, j\neq\ell}(w_j+1)-\prod_{1\leq j\leq n, j\neq \ell'}(w_j'+1)\right] \det\left[\frac{1}{w_i-w_{i'}'}\right]_{\substack{1\leq i,i'\leq n\\ i\neq \ell, i'\neq \ell'}}.
	\end{equation*}
	Hence 
	\begin{equation}\label{eq: cauchy identity infinite induction hypothesis}
		\begin{aligned}
		&\text{LHS of }\eqref{eq: Cauchy identity for infinite sum}\\
		& = (1-p)\cdot\prod_{k=1}^{n}\frac{w_k'(1+w_k)(1+pw_k)}{w_k(1+pw_k')}\sum_{\ell,\ell'=1}^{n}(-1)^{\ell+\ell'} \frac{w_\ell(1+pw_{\ell'}')}{ w_{\ell'}'(1+w_\ell)(1+pw_\ell)}\det\left[\frac{1}{w_i-w_{i'}'}\right]_{\substack{1\leq i,i'\leq n\\ i\neq \ell, i'\neq \ell'}}\\
		&+ p\cdot \prod_{k=1}^{n}\frac{w_k'(1+w_k')(1+pw_k)}{w_k(1+pw_k')}\sum_{\ell,\ell'=1}^{n} (-1)^{\ell+\ell'}\frac{w_\ell(1+pw_{\ell'}')}{ w_{\ell'}'(1+w_{\ell'}')(1+pw_\ell)}\det\left[\frac{1}{w_i-w_{i'}'}\right]_{\substack{1\leq i,i'\leq n\\ i\neq \ell, i'\neq \ell'}}\\
		&:= (1-p)\cdot (-1)^n A(-1)\cdot \frac{B(0)A(-1/p)}{A(0)B(-1/p)}\cdot D_1+p\cdot (-1)^n B(-1)\cdot \frac{B(0)A(-1/p)}{A(0)B(-1/p)}\cdot D_2.
		\end{aligned}
	\end{equation}
	Here we set $A(z):=\prod_{i=1}^{n}(z-w_i)$ and $B(z):= \prod_{i=1}^{n}(z-w_i')$ and $D_1$, $D_2$ represent the first and second sum over $\ell, \ell'$ respectively.  Now by \eqref{eq: rank-one cauchy 2} and Lemma \ref{lemma: rank-one cauchy residue} we have 
	\begin{equation}\label{eq: rank1cauchy1}
		D_1 = \frac{1}{1-p}\cdot\left[\frac{A(0)B(-1/p)}{B(0)A(-1/p)}-\frac{B(-1/p)}{A(-1/p)}+p\cdot\frac{B(-1)}{A(-1)}-\frac{A(0)B(-1)}{B(0)A(-1)}+1-p\right]\cdot \det\left[\frac{1}{w_i-w_{i'}'}\right]_{i,i'=1}^n,
	\end{equation}
	and 
	\begin{equation}\label{eq: rank1cauchy2}
        D_2 = \frac{1}{p}\cdot\left[\frac{A(0)}{B(0)}-\frac{B(-1/p)A(0)}{A(-1/p)B(0)}+(p-1)\cdot\frac{A(-1)}{B(-1)}+\frac{A(-1)B(-1/p)}{B(-1)A(-1/p)}-p\right]\cdot \det\left[\frac{1}{w_i-w_{i'}'}\right]_{i,i'=1}^n.
\end{equation}
    Inserting \eqref{eq: rank1cauchy1} and \eqref{eq: rank1cauchy2} into \eqref{eq: cauchy identity infinite induction hypothesis}, after necessary cancellation  we obtain
    \begin{equation}
    \begin{aligned}
          \text{LHS of }\eqref{eq: Cauchy identity for infinite sum}&= \left[(-1)^n A(-1)-(-1)^{n} B(-1)\right]\cdot \det\left[\frac{1}{w_i-w_{i'}'}\right]_{i,i'=1}^n\\
          &= \left[\prod_{j=1}^{n}(1+w_j)-\prod_{j=1}^{n}(1+w_j')\right]\det\left[\frac{1}{w_i-w_{i'}'}\right]_{i,i'=1}^n.
    \end{aligned}
    \end{equation}
    This completes the proof of Lemma \ref{lemma: Cauchy identity for infinite sum}.
\end{proof}

Finally inserting \eqref{eq: reduction of H_1} and \eqref{eq: further reduction of H_2} into \eqref{eq: reduction of S_N} and apply Lemma \ref{lemma: sum of product of determinants} below we obtain
\begin{equation}\label{eq: final simplification of the sum}
\begin{aligned}
\hat{S}_N &= \sum_{\substack{J,J'\subset \{1,\cdots,N\}\\ |J|=|J'|}} (-1)^{\#(J,J^c)+\#(J',(J')^c)} H_1(J,J')H_2(J^c,(J')^c)\\
&= p\cdot\prod_{j=1}^{N}(w_j+1)\left(\det\left[\hat{C}(i,i')\right]_{i,i'=1}^N-\det\left[\hat{C}(i,i')-\frac{1}{w_i+1}\right]_{ i,i'=1}^N\right)\\
&-(1-p)\cdot\prod_{j=1}^{N}(w_j'+1)\left(\det\left[\hat{C}(i,i')\right]_{i,i'=1}^N-\det\left[\hat{C}(i,i')+\frac{1}{w_{i'}'+1}\right]_{ i,i'=1}^N\right),
\end{aligned}
\end{equation}
where 
\begin{equation}
\hat{C}(i,i') = \frac{1}{w_i-w_{i'}'}-\left(\frac{z'}{z}\right)^L\frac{w_i(1+pw_{i'}')}{w_{i'}'(1+pw_i)}\frac{1}{w_i-w_{i'}'}.
\end{equation}
Note that in the first equality of \eqref{eq: final simplification of the sum} we have added an extra term corresponding to $|J|=|J'|=0$ comparing to \eqref{eq: reduction of S_N} which is harmless since the summand is $0$ in this case.
\begin{lemma}[Lemma 5.9 of \cite{baik2019multi}]\label{lemma: sum of product of determinants} For two $n\times n$ matrices $A$ and $B$,
	\begin{equation}
	\sum_{\substack{J,J'\subset\{1,\cdots,n\}\\|J|=|J'|}} (-1)^{\#(J,J^c)+\#(J',(J')^c)}\det[A(i,i')]_{i\in J,i'\in J'}\det[B(i,i')]_{i\in J^c,i'\in (J')^c} = \det[A+B]_{1\leq i,i'\leq n}.
	\end{equation}
\end{lemma}
\smallskip

\subsection{Proof of Proposition \ref{prop: Cauchy identity}: Step 3} \label{sec: cauchy finish}
In this section we further simplify equation \eqref{eq: final simplification of the sum} to conclude the proof of Proposition \ref{prop: Cauchy identity}. We compute $\det[\hat{C}(i,i')]$ first. For notational convenience we set $(z'/z)^L:=\mu$. Note that 
\begin{equation*}
\begin{aligned}
	\hat{C}(i,i') &= \frac{1}{w_i-w_{i'}'}-\mu\cdot \frac{w_i(1+pw_{i'}')}{w_{i'}'(1+pw_i)}\frac{1}{w_i-w_{i'}'}\\
	              &= \left(1-\mu\right)\frac{1}{w_i-w_{i'}'}-\mu\cdot\frac{1}{w_{i'}'(1+pw_i)}.
\end{aligned} 
\end{equation*}

Hence by Lemma \ref{Lemma: rank-one cauchy} and Lemma \ref{lemma: rank-one cauchy residue} we have 
\begin{equation}\label{eq: deformed cauchy determinant}
\begin{aligned}
\det[\hat{C}(i,i')]_{i,i'=1}^N &= (1-\mu)^N \det\left[\frac{1}{w_i-w_{i'}'}-\frac{\mu}{1-\mu}\frac{1}{w_{i'}'(1+pw_i)}\right]_{i,i'=1}^N\\
&= (1-\mu)^N\det[C(i,i')]_{i,i'=1}^N\cdot\left(1-\frac{\mu}{1-\mu}\left(\frac{B(-1/p)A(0)}{A(-1/p)B(0)}-1\right)\right)\\
&= (1-\mu)^{N-1}\det[C(i,i')]_{i,i'=1}^N\cdot\left(1-\mu\cdot \frac{B(-1/p)A(0)}{A(-1/p)B(0)}\right).
\end{aligned}
\end{equation}
Here $C(i,i') = \frac{1}{w_i-w_{i'}'}$ and $A(0)$ is the evaluation at $z=0$ of the polynomial $A(z):=\prod_{j=1}^{N}(z-w_j)$. The other terms involving $A(\cdot)$ and $B(\cdot)$ are defined in a similar way with $B(z):= \prod_{j=1}^{N}(z-w_{j}')$. Now by Lemma \ref{lemma: rank one} we have 
\begin{equation}\label{eq: perturbation of hat C 1}
	\det\left[\hat{C}(i,i')\right]-\det\left[\hat{C}(i,i')-\frac{1}{w_i+1}\right] = \sum_{k,\ell=1}^{N}(-1)^{\ell+k} \det[\hat{C}^{\ell,k}]\cdot \frac{1}{w_{\ell}+1},
\end{equation}
where $\hat{C}^{\ell,k}$ is the matrix obtained by removing row $\ell$ and column $k$ from $\hat{C}$. Similarly 
\begin{equation}\label{eq: perturbation of hat C 2}
\det\left[\hat{C}(i,i')\right]-\det\left[\hat{C}(i,i')+\frac{1}{w_{i'}'+1}\right] = -\sum_{k,\ell=1}^{N}(-1)^{\ell+k} \det[\hat{C}^{\ell,k}]\cdot \frac{1}{w_{k}'+1}.
\end{equation}
Now since $\hat{C}^{\ell,k}$ has the same entries as $\hat{C}$ only omitting row $\ell$ and column $k$, by \eqref{eq: deformed cauchy determinant} we have 
\begin{equation}\label{eq: determinant hat C_lk}
	\det[\hat{C}^{\ell,k}] = (1-\mu)^{N-2}\det[C^{\ell,k}]\cdot \left(1-\mu\cdot \frac{B(-1/p)A(0)}{A(-1/p)B(0)}\cdot \frac{w_{k}'(1+pw_{\ell})}{w_\ell(1+pw_{k}')}\right).
\end{equation}
Where $C^{\ell,k}$ is obtained from removing row $\ell$ and column $k$ from the $N\times N$ Cauchy matrix $C$ with $C(i,i')=\frac{1}{w_i-w_{i'}'}$ and $A(z):=\prod_{j=1}^{N}(z-w_j)$ and $B(z):=\prod_{j=1}^{N}(z-w_j')$. Inserting \eqref{eq: determinant hat C_lk} into \eqref{eq: perturbation of hat C 1} we see 
\begin{equation}\label{eq: difference of hat C 1}
\begin{aligned}
	&\det\left[\hat{C}(i,i')\right]-\det\left[\hat{C}(i,i')-\frac{1}{w_i+1}\right]= (1-\mu)^{N-2}\sum_{k,\ell=1}^{N}(-1)^{\ell+k} \det[C^{\ell,k}]\cdot \frac{1}{w_{\ell}+1}\\
	&-\mu(1-\mu)^{N-2}\cdot \frac{B(-1/p)A(0)}{A(-1/p)B(0)} \sum_{k,\ell=1}^{N}(-1)^{\ell+k} \det[C^{\ell,k}]\cdot \frac{w_{k}'(1+pw_{\ell})}{w_\ell(1+w_\ell)(1+pw_{k}')}.
\end{aligned}
\end{equation}
Similarly
\begin{equation}\label{eq: difference of hat C 2}
\begin{aligned}
&\det\left[\hat{C}(i,i')\right]-\det\left[\hat{C}(i,i')+\frac{1}{w_{i'}'+1}\right]= -(1-\mu)^{N-2}\sum_{k,\ell=1}^{N}(-1)^{\ell+k} \det[C^{\ell,k}]\cdot \frac{1}{w_{k}'+1}\\
&+\mu(1-\mu)^{N-2}\cdot \frac{B(-1/p)A(0)}{A(-1/p)B(0)} \sum_{k,\ell=1}^{N}(-1)^{\ell+k} \det[C^{\ell,k}]\cdot \frac{w_{k}'(1+pw_{\ell})}{w_\ell(1+w_{k'}')(1+pw_{k}')}.
\end{aligned}
\end{equation}
Now by \eqref{eq: minor of cauchy} and  Lemma \ref{lemma: rank-one cauchy residue} we have 
\begin{equation}\label{eq: several rank-one cauchy}
\begin{aligned}
	&\sum_{k,\ell=1}^{N}(-1)^{\ell+k} \det[C^{\ell,k}]\cdot \frac{1}{w_{\ell}+1} = \det[C]\cdot \left(1-\frac{B(-1)}{A(-1)}\right),\\
	&\sum_{k,\ell=1}^{N}(-1)^{\ell+k} \det[C^{\ell,k}]\cdot \frac{1}{w_{k}'+1} = \det[C]\cdot \left(\frac{A(-1)}{B(-1)}-1\right),\\
	&\sum_{k,\ell=1}^{N}(-1)^{\ell+k} \det[C^{\ell,k}]\cdot \frac{w_{k}'(1+pw_{\ell})}{w_\ell(1+w_\ell)(1+pw_{k}')} \\
	&= -\frac{1}{p}\det[C]\cdot\left[\left(1-\frac{A(-1/p)}{B(-1/p)}\right)\frac{B(0)}{A(0)}+ \left(p-1+\frac{A(-1/p)}{B(-1/p)}\right)\frac{B(-1)}{A(-1)}-p\right],\\
	&\sum_{k,\ell=1}^{N}(-1)^{\ell+k} \det[C^{\ell,k}]\cdot \frac{w_{k}'(1+pw_{\ell})}{w_\ell(1+w_\ell)(1+pw_{k}')} \\
	&= -\frac{1}{1-p}\det[C]\cdot\left[\left(\frac{B(0)}{A(0)}-1\right)\frac{A(-1/p)}{B(-1/p)}+ \left(p-\frac{B(0)}{A(0)}\right)\frac{A(-1)}{B(-1)}+1-p\right].
\end{aligned}
\end{equation}
Inserting \eqref{eq: several rank-one cauchy} into \eqref{eq: difference of hat C 1} and \eqref{eq: difference of hat C 2} and combine with \eqref{eq: final simplification of the sum}, after some tedious simplification we conclude that 
\begin{equation}
	\hat{S}_N = (1-\mu)^N \left(\prod_{j=1}^{N} (w_j+1)-\prod_{j=1}^{N}(w_j'+1)\right)\cdot \det\left[\frac{1}{w_i-w_{i'}'}\right]_{i,i'=1}^N.
\end{equation}
This completes the proof of Proposition \ref{prop: Cauchy identity}.
\medskip

\subsection{Perturbation formulas for Cauchy determinants}\label{sec: perturbation formulas}
In this section we collect all the elementary linear algebra facts needed in the proof of Proposition \ref{prop: Cauchy identity}. Some of them have already been discussed in \cite{baik2019multi}. First we state a general linear algebra lemma on rank-one perturbations: 

\begin{lemma} \label{lemma: rank one}
	Let $D=[D_{ij}]_{i,j=1}^n$ be an $n\times n$ matrix.  Then for any function $f,g: \mathbb{C}\to \mathbb{C}$ and complex numbers $x_1,\cdots,x_n$, $y_1,\cdots,y_n$ we have 
\begin{equation} \label{eq: rank-one cauchy 1}
\begin{aligned}
\det\left[D_{ij}+f(x_i)g(y_j)\right]_{i,j=1}^n =  \det[D]+\sum_{k,\ell=1}^{n}(-1)^{\ell+k}\det[D^{\ell,k}]g(y_k)f(x_\ell),
\end{aligned}
\end{equation}
where $D^{\ell,k}$ is obtained by removing row $\ell$ and column $k$ from $D$.
\end{lemma}
\begin{proof}
	For $D$ invertible, by the rank-one property and the Cramer's rule we have 
	\begin{equation*}
	\begin{aligned}
	 \det\left[D_{ij}+f(x_i)g(y_j)\right]_{i,j=1}^n &= \det[D]\left(1+\sum_{k,\ell=1}^{n}g(y_k)(D^{-1})_{k,\ell}f(x_\ell)\right)\\
	 &=  \det[D]+\sum_{k,\ell=1}^{n}(-1)^{\ell+k}\det[D^{\ell,k}]g(y_k)f(x_\ell).
	\end{aligned}
	\end{equation*}
	For general matrix $D$ we pick $\{\epsilon_k\}_{k=1}^\infty$ such that $\epsilon_k\to 0$ as $k\to \infty$ and $D+\epsilon_k I_n$ are invertible for all $\epsilon_k$. Now apply the above argument for $D+\epsilon_k I_n$ and let $k\to \infty$.
\end{proof}
Next we specialize to the case of $C$ being a Cauchy matrix when the minors can be explicitly calculated:
\begin{lemma}\label{Lemma: rank-one cauchy}
	Assume further that the matrix $C$ is a Cauchy matrix with $(i,j)$-th entry $\frac{1}{x_i-y_j}$ for distinct complex numbers $x_1,\cdots, x_n$ and $y_1,\cdots,y_n$. Then we further have 
	\begin{equation}\label{eq: rank-one cauchy 2}
	\det\left[C_{ij}+f(x_i)g(y_j)\right]_{i,j=1}^n = \det[C]\cdot\left(1-\sum_{k,\ell=1}^{n}\frac{f(x_\ell)B(x_\ell)A(y_k)g(y_k)}{(x_\ell-y_k)A'(x_\ell)B'(y_k)}\right).
	\end{equation}
	Here $A(z):= \prod_{i=1}^{n}(z-x_i)$ and $B(z):= \prod_{i=1}^{n}(z-y_i)$ are monic polynomials with roots at $x_i$'s and $y_i$'s.
\end{lemma}
\begin{proof}
	For Cauchy matrix $C$ we have 
	\begin{equation*}
		\det[C]= \frac{\prod_{1\leq i<j\leq n}(x_i-x_j)(y_i-y_j)}{\prod_{1\leq i,j\leq n}(x_i-y_j)}.
	\end{equation*}
	Note that $C^{\ell,k}$ is also a Cauchy matrix so we have 
	\begin{equation*}
	\det[C^{\ell,k}]= \frac{\prod_{\substack{1\leq i<j\leq n\\ i\neq \ell, j\neq k}}(x_i-x_j)(y_i-y_j)}{\prod_{\substack{1\leq i,j\leq n\\ i\neq \ell, j\neq k}}(x_i-y_j)}.
	\end{equation*}
	Hence 
	\begin{equation}\label{eq: minor of cauchy}
		\frac{\det[C^{\ell,k}]}{\det[C]} = (-1)^{\ell+k+1} \frac{B(x_\ell)A(y_k)}{(x_\ell-y_k)A'(x_\ell)B'(y_k)}.
	\end{equation}
	Now \eqref{eq: rank-one cauchy 1} and \eqref{eq: minor of cauchy} imply \eqref{eq: rank-one cauchy 2}.
\end{proof}
For special choices of $f$ and $g$, \eqref{eq: rank-one cauchy 2} can be further simplified using the residue theorem. We list here all the situations encountered in the proof of Proposition \ref{prop: Cauchy identity}.
\begin{lemma}\label{lemma: rank-one cauchy residue}
	 Given distinct complex numbers $x_1,\cdots,x_n$ and $y_1,\cdots,y_n$. Let $C$ be the Cauchy matrix with $(i,j)$-th entry $\frac{1}{x_i-y_j}$ and $A(z)= \prod_{i=1}^{n}(z-x_i)$ and $B(z)= \prod_{i=1}^{n}(z-y_i)$. Then 
	\begin{enumerate}
		\item For $f(x) = \frac{x}{(1+x)(1+px)}$ and $g(y)=\frac{1+py}{y}$ we have $\sum_{k,\ell=1}^{n}\frac{f(x_\ell)B(x_\ell)A(y_k)g(y_k)}{(x_\ell-y_k)A'(x_\ell)B'(y_k)}$ equals
		\begin{equation}
			 -\frac{1}{1-p}\cdot\left[\left(\frac{A(0)}{B(0)}-1\right)\frac{B(-1/p)}{A(-1/p)}+ \left(p-\frac{A(0)}{B(0)}\right)\frac{B(-1)}{A(-1)}+1-p\right].
		\end{equation}
		\item For $f(x) = \frac{x}{1+px}$ and $g(y)=\frac{1+py}{y(1+y)}$ we have $\sum_{k,\ell=1}^{n}\frac{f(x_\ell)B(x_\ell)A(y_k)g(y_k)}{(x_\ell-y_k)A'(x_\ell)B'(y_k)}$ equals
		\begin{equation}
         -\frac{1}{p}\cdot\left[\left(1-\frac{B(-1/p)}{A(-1/p)}\right)\frac{A(0)}{B(0)}+ \left(p-1+\frac{B(-1/p)}{A(-1/p)}\right)\frac{A(-1)}{B(-1)}-p\right].
		\end{equation}
		\item For $f(x) = \frac{1}{1+x}$ and $g(y)= 1$ we have
		\begin{equation}
			\sum_{k,\ell=1}^{n}\frac{f(x_\ell)B(x_\ell)A(y_k)g(y_k)}{(x_\ell-y_k)A'(x_\ell)B'(y_k)} = \frac{B(-1)}{A(-1)}-1.
		\end{equation} 
	\end{enumerate}
\end{lemma}

\begin{proof}
     We will only prove part $(1)$ since the arguments for the other parts are similar. For $f(z) = \frac{z}{(1+z)(1+pz)}$ and $g(\xi)=\frac{1+p\xi}{\xi}$ we consider the following double contour integral: 
     \begin{equation*}
     \oint_{|z|=R}\oint_{|\xi|=r}\frac{\dd\xi}{2\pi \ii}\frac{\dd z}{2\pi \ii} f(z)g(\xi)\frac{B(z)A(\xi)}{(z-\xi)A(z)B(\xi)}.
     \end{equation*}
     Where $R>r$ are both large enough so that all the possible poles of the  integrand are inside the integral contours. Now since for fixed $r$ the integrand is of order $O(R^{-2})$, the double integral goes to $0$ as $R\to \infty$. Thus for all $R$ large enough the double contour integral equals $0$. On the other hand by the residue theorem we have 
     \begin{equation}\label{eq: residues for double contour integral}
     \begin{aligned}
     	0&=\oint_{|z|=R}\oint_{|\xi|=r}\frac{\dd\xi}{2\pi \ii}\frac{\dd z}{2\pi \ii} f(z)g(\xi)\frac{B(z)A(\xi)}{(z-\xi)A(z)B(\xi)}\\
     	& = \oint_{|z|=R} \frac{\dd z}{2\pi \ii} \frac{f(z)}{z}\frac{B(z)A(0)}{A(z)B(0)}+ \sum_{k=1}^{n} \oint_{|z|=R}\frac{\dd z}{2\pi \ii} f(z)g(y_k) \frac{B(z)A(y_k)}{(z-y_k)A(z)B'(y_k)}\\
     	&= \frac{1}{1-p}\sum_{k=1}^{n}\frac{A(y_k)}{B'(y_k)}\left(\frac{B(-1)}{A(-1)}\frac{1+py_k}{y_k(1+y_k)}-\frac{B(-1/p)}{A(-1/p)}\frac{1}{y_k}\right)\\
     	&+ \sum_{k,\ell=1}^{n} f(x_{\ell})g(y_k)\frac{B(x_\ell)A(y_k)}{(x_\ell-y_k)A'(x_\ell)B'(y_k)}.
     \end{aligned}
     \end{equation}
     Where in the first equality the first contour integral is $O(R^{-2})$ hence $0$ for $R$ large enough.  The single sum over $1\leq k\leq n$ can be obtained as the residue terms for the following single contour integral:
     \begin{equation*}
     	\oint_{|\xi|=r} \frac{\dd\xi}{2\pi \ii}\frac{1+p\xi}{\xi(1+\xi)}\frac{A(\xi)}{B(\xi)} = \frac{A(0)}{B(0)}-(1-p)\frac{A(-1)}{B(-1)}+\sum_{k=1}^{n}\frac{A(y_k)}{B'(y_k)}\frac{1+py_k}{y_k(1+y_k)}.
     \end{equation*}
     Here $r$ is large enough so that $|\xi|=r$ contains all possible poles of the integrand inside. On the other hand by considering the residue at $\infty$ we have $\oint_{|\xi|=r} \frac{d\xi}{2\pi i}\frac{1+p\xi}{\xi(1+\xi)}\frac{A(\xi)}{B(\xi)} = p$. Hence 
     \begin{equation}\label{eq: residues for single contour integral 1}
     \sum_{k=1}^{n}\frac{A(y_k)}{B'(y_k)}\frac{1+py_k}{y_k(1+y_k)} = p-\frac{A(0)}{B(0)}+(1-p)\frac{A(-1)}{B(-1)}.
     \end{equation}
     Similar residue analysis on the contour integral $\oint_{|\xi|=r} \frac{\dd\xi}{2\pi \ii}\frac{1}{\xi}\frac{A(\xi)}{B(\xi)}$ gives
     \begin{equation}\label{eq: residues for single contour integral 2}
     	\sum_{k=1}^{n}\frac{A(y_k)}{B'(y_k)}\frac{1}{y_k} = 1-\frac{A(0)}{B(0)}.
     \end{equation}
     Inserting \eqref{eq: residues for single contour integral 1} and \eqref{eq: residues for single contour integral 2} into \eqref{eq: residues for double contour integral} we obtain 
     \begin{equation}
     \begin{aligned}
     	&\sum_{k,\ell=1}^{n} f(x_{\ell})g(y_k)\frac{B(x_\ell)A(y_k)}{(x_\ell-y_k)A'(x_\ell)B'(y_k)} \\
     	&= -\frac{1}{1-p}\cdot \left[\left(p-\frac{A(0)}{B(0)}\right)\frac{B(-1)}{A(-1)}+1-p -\left(1-\frac{A(0)}{B(0)}\right)\frac{B(-1/p)}{A(-1/p)}\right].
     \end{aligned}
     \end{equation}
     This completes the proof of part (1).
\end{proof}
\smallskip
\subsection{Proof of Corollary \ref{cor: single sum arbitrary index} and \ref{cor: sum over two arbitrary index}}\label{sec: proof cor}
Finally we discuss how Corollary \ref{cor: single sum arbitrary index} and \ref{cor: sum over two arbitrary index} follows from the Propositions using periodicity. For given $\vec x=(x_1,\cdots,x_N)\in \conf_{N}^{(L)}$ and $0\leq k\leq N-1$, set $\vec x':=(x_1',\cdots,x_N')=(x_{N-k+1}+L,\cdots, x_N+L,x_1,\cdots,x_{N-k})$. Then the condition $\vec x\in \conf_{N}^{(L)}\cap \{x_{N-k}\geq a\}$ is the same as $\vec x'\in \conf_{N}^{(L)}\cap \{x_N'\geq a\}$. Now consider 
\begin{equation*}
	\Psi_{\vec x}^r(\vec w) = \det\left[\left(\frac{w_i}{1+pw_i}\right)^{1-j}(1+w_i)^{-x_{N-j+1}+j-1}\right]_{i,j=1}^N.
\end{equation*}
We move the first $k$ columns of the matrix to the end. The resulting determinant equals $(-1)^{k(N-1)}$ times the determinant of the matrix whose $(i,j)$-th entry has the form $\left(\frac{w_i}{1+pw_i}\right)^{1-j-k}(1+w_i)^{-x_{N-j+1}'+j+k-1}$ for $1\leq j\leq N-k$ and the form $\left(\frac{w_i}{1+pw_i}\right)^{1-j-k+N}(1+w_i)^{-x_{N-j+1}'+j+k-1+L-N}$ for $N-k+1\leq j\leq N$. But since $\left(\frac{w_i}{1+pw_i}\right)^N(1+w_i)^{L-N}=z^L$ for all $1\leq i\leq N$, we have 
$$\left(\frac{w_i}{1+pw_i}\right)^{1-j-k+N}(1+w_i)^{-x_{N-j+1}'+j+k-1+L-N} = z^L\cdot \left(\frac{w_i}{1+pw_i}\right)^{1-j-k}(1+w_i)^{-x_{N-j+1}'+j+k-1}.$$
After factoring out the common factor $\left(\frac{w_i}{1+pw_i}\right)^{-k}(1+w_i)^{k}$ from each row and $z^L$ from the last $k$ columns we conclude that 
\begin{equation}
	\Psi_{\vec x}^r(\vec w) = (-1)^{k(N-1)}z^{kL}\prod_{i=1}^{N}\left[\left(\frac{1+pw_i}{w_i}\right)^k\cdot(1+w_i)^{k}\right]\Psi_{\vec x'}^r(\vec w).
\end{equation}
Hence 
\begin{equation}
	\begin{aligned}
		\sum_{\vec x\in \conf_{N}^{(L)}\cap \{x_{N-k}\geq a\}}\Psi_{\vec x}^r(\vec w)& = (-1)^{k(N-1)}z^{kL}\prod_{i=1}^{N}\left[\left(\frac{1+pw_i}{w_i}\right)^k\cdot(1+w_i)^{k}\right]\cdot\sum_{\vec x'\in \Omega_{L,N}\cap \{x_{N}'\geq a\}}\Psi_{\vec x'}^r(\vec w)\\
		&=(-1)^{k(N-1)}z^{kL}\prod_{i=1}^{N}\left[\left(\frac{1+pw_i}{w_i}\right)^k\cdot(1+w_i)^{-a+k+1}\right]\cdot \det[w_{i}^{-j}]_{i,j=1}^N.
	\end{aligned}
\end{equation}
Corollary \ref{cor: sum over two arbitrary index} follows from Proposition \ref{prop: Cauchy identity} in a similar way.
\medskip

\section{Large-time asymptotics under relaxation time scale}
\label{sec: relaxation time limit}
In this section we discuss the large time limit of the multi-point distribution of dpTASEP($L,N,\vec y$) under the relaxation time scale $t=O(L^{3/2})$. In Theorem~\ref{thm:main} we state the limit theorem for general initial condition satisfying certain assumptions. Below are the precise assumptions on the initial conditions we need:
\subsection{Assumptions on the initial condition}
We now state the assumptions on the sequence of the initial conditions $\vec y(L)$ under which we prove the limit theorem. 
The conditions are in terms of the global energy function and the characteristic function defined in Definition~\ref{def: globalenergy}.

Recall that the finite-time formula \eqref{eq: multipoint general} involves a $m$-fold contour integral with respect to $\vec z=(z_1,\cdots,z_m)$ for $0<|z_m|<\cdots<|z_1|<\mathbbm{r}_c$ where the critical radius $\mathbbm{r}_c$ is defined in \eqref{eq: critical radius}. It turns out that we need to rescale the parameters $z_i$ to be close to the critical value $\mathbbm{r}_c$ in order to make the large time limits converge.

\begin{notation}
	For given complex parameter $z$ with $0<|z|<\mathbbm{r}_c$, we introduce the rescaled parameter $\z$ defined by
	\begin{equation}\label{eq: rescaled z}
		z^L = (-1)^N \mathbbm{r}_c^L \z.
	\end{equation}
    The constraint $0<|z|<\mathbbm{r}_c$ then translates to $0<|\z|<1$. We introduce a similar rescaling for the parameters $\vec z=(z_1,\cdots,z_m)$. Then for $0<|z_m|<\cdots<|z_1|<\mathbbm{r}_c$, the rescaled parameters satisfy $0<|\z_m|<\cdots<|\z_1|<1$.  Throughout the rest of the paper we will always use $z_i$ to represent the unscaled parameters and $\z_i$ to represent the rescaled ones satisfying \eqref{eq: rescaled z}.
\end{notation}
\begin{assum} \label{def:asympstab} 
	We assume that the sequence of the initial profiles $\vec y=\vec y(L)$ satisfies the following three conditions as $L\to \infty$.

	\begin{enumerate}[(A)]
		\item (Convergence of global energy) 
		There exist a constant $r\in (0,1)$ and a non-zero function $E_{\mathrm{ic}}(\mathrm{z})$ such that for every $0<\epsilon<1/2$, 
		\begin{equation*} 
			\mathcal{E}_{\vec y} (z) = E_{\mathrm{ic}}(\mathrm{z}) \left( 1+ O(L^{\epsilon -1/2})\right),
		\end{equation*}
		uniformly for $ |\mathrm{z}|<  {r}$  as $L\to \infty$.  
		
		\item (Convergence of characteristic function) 
		There exist constants $0<r_1<r_2<1$ and a function $\cchi_{\mathrm{ic}}(\eta,\xi;\mathrm{z})$ 
		such that for every $0<\epsilon<1/8$,  
		\begin{equation*} 
			\chi_{\vec y}(v,u; z) = \cchi_{\mathrm{ic}}(\eta,\xi;\mathrm{z}) + O(L^{4\epsilon - 1/2}),
		\end{equation*}
		uniformly for $ r_1<|\mathrm{z}|<  r_2$, $u\in \mathcal{L}_z^{(\epsilon)}$ and $v \in \mathcal{R}_z^{(\epsilon)}$ as $L\to \infty$
		where
		\begin{equation*}
			\xi = \mathcal{M}_{L,\mathrm{left}} (u)\in \mathrm{L}_{\mathrm{z}} 
			\quad \text{and}\quad 
			\eta = \mathcal{M}_{L,\mathrm{right}}(v) \in \mathrm{R}_{\mathrm{z}}
		\end{equation*}
		are the images under the maps defined in Lemma \ref{lemma: convergenceofroots}.

		\item (Tail estimates of characteristic function) 
		Let $r_1$ and $r_2$ be same as in (B). 
		There are constants $\epsilon'', C'>0$ such that 
		\begin{equation}\label{eq:easier_tail_estimates}
			|\chi_{\vec y}(v,u; z)| \le C' L^{\epsilon''},
		\end{equation}
		for all $(v,u)\in\mathcal{R}_z\times\mathcal{L}_z$ for all $r_1<|\mathrm{z}|<r_2$. 
	\end{enumerate}
\end{assum}

\medskip
\subsection{Step and flat initial conditions} It turns out that Assumption~\ref{def:asympstab} is not easy to check in general. Nevertheless we are able to verify them for at least the classical step and flat initial conditions. The following proposition combined with Theorem~\ref{thm:main} gives the corresponding limit theorems for dpTASEP starting with step and flat initial conditions.
\begin{proposition}\label{thm:special_IC}
	\begin{enumerate}[(i)]
		\item For the step initial condition $\vec y_{\mathrm{step}}=(-1,-2,\cdots-N)$, Assumption~\ref{def:asympstab} holds with 
		\begin{equation}
			E_{\mathrm{step}}(\mathrm{z}) = 1
			\quad \text{and}\quad
			\cchi_{\mathrm{step}} (\eta,\xi;\mathrm{z}) = 1.
		\end{equation}
		\item For the flat initial condition $\vec y_{\mathrm{flat}}=(-d,\cdots,-Nd)$, where we assume $d=L/N\in \mathbb{N}$, 
		Assumption~\ref{def:asympstab} holds with 
		\begin{equation}\label{eq: characteristicflat}
			E_{\mathrm{flat}}(\mathrm{z}) = (1-\mathrm{z})^{-1/4}  e^{-B(\mathrm{z})}
			\quad \text{and}\quad
			\cchi_{\mathrm{flat}} (\eta,\xi;\mathrm{z}) =  e^{-\mathrm{h}(\xi,\mathrm{z})-\mathrm{h}(\eta,\mathrm{z})}\eta(\eta-\xi)\mathbf{1}_{\xi=-\eta}
		\end{equation}
		for $0<|\mathrm{z}|<1$, where $B(\mathrm{z})$ is defined in~\eqref{eq:def_Bz}, $\mathrm{h}(\zeta,\mathrm{z})$ is defined in~\eqref{eq:def_h} and~\eqref{eq:2019_10_18_02}.
	\end{enumerate}
\end{proposition}
The step case is trivial since by the discussion in Remark~\ref{rmk: step} we have $\mathcal{E}_{\mathrm{step}}(z)=\chi_{\mathrm{step}}(v,u;z)\equiv 1$. The calculation for the flat case is a bit more involved and we postpone the proof to Section~\ref{sec: verifyIC}.

\subsection{Formula for the limiting distribution}\label{sec: limitingdistribution}
The following formula for the relaxation-time limiting distribution was first obtained in \cite{baik2019multi} for the step initial condition and \cite{baik2019general} for more general initial conditions. 
The formula involves $\mathrm{C_{\mathrm{ic}}^{\mathrm{per}}}(\vec \z)$ which are limits of $\mathscr{C}_{\vec y}^{(L)}(\vec z)$
and operators $\mathrm{K}_{1}^{\mathrm{per}}$ and $\mathrm{K}_{\mathrm{ic}}^{\mathrm{per}}$ which are limits of $\mathscr{K}_{1}^{(L)}$ and $\mathscr{K}_{\vec y}^{(L)}$. 
The operators $\mathrm{K}_{1}^{\mathrm{per}}$ and $\mathrm{K}_{\mathrm{ic}}^{\mathrm{per}}$ are defined on the sets 
\begin{equation} \label{eq:limSo}
	\mathrm{S}_1:= \mathrm{L}_{\mathrm{z}_1}\cup \mathrm{R}_{\mathrm{z}_2} \cup \mathrm{L}_{\mathrm{z}_3} \cup \cdots 
	\cup
	\begin{dcases}
		\mathrm{R}_{\mathrm{z}_m},& \text{if $m$ is even},\\
		\mathrm{L}_{\mathrm{z}_m}, & \text{if $m$ is odd},
	\end{dcases} 
\end{equation}
and
\begin{equation} \label{eq:limSt}
	\mathrm{S}_2:=\mathrm{R}_{\mathrm{z}_1}\cup \mathrm{L}_{\mathrm{z}_2} \cup \mathrm{R}_{\mathrm{z}_3} \cup \cdots 
	\cup
	\begin{dcases}
		\mathrm{L}_{\mathrm{z}_m},& \text{if $m$ is even},\\
		\mathrm{R}_{\mathrm{z}_m}, & \text{if $m$ is odd},
	\end{dcases}
\end{equation}
where $\mathrm{L}_{\mathrm{z}}$ and $\mathrm{R}_{\mathrm{z}}$  are the sets defined in Definition~\ref{def: limitingroots}.
We express the limiting distribution function $\Fic^{\mathrm{per}}$ in terms of the above terms. 

\begin{definition}\label{def: limitingroots}
	Given $0<|\z|<1$, we define the discrete sets $\mathrm{S}_{\z}:= \mathrm{L}_{\z}\cup\mathrm{R}_{\z}$ where
	\begin{equation}
		\mathrm{L}_{\z}:= \{\xi\in \mathbb{C}: e^{-\xi^2/2}=\z\}\cap\{\text{Re}(\xi)<0\},\quad \mathrm{R}_{\z}:= \{\eta\in \mathbb{C}: e^{-\eta^2/2}=\z\}\cap\{\text{Re}(\eta)>0\}.
	\end{equation}
\end{definition}

\begin{figure}[h]
	\centering
	\includegraphics[width=0.35\textwidth]{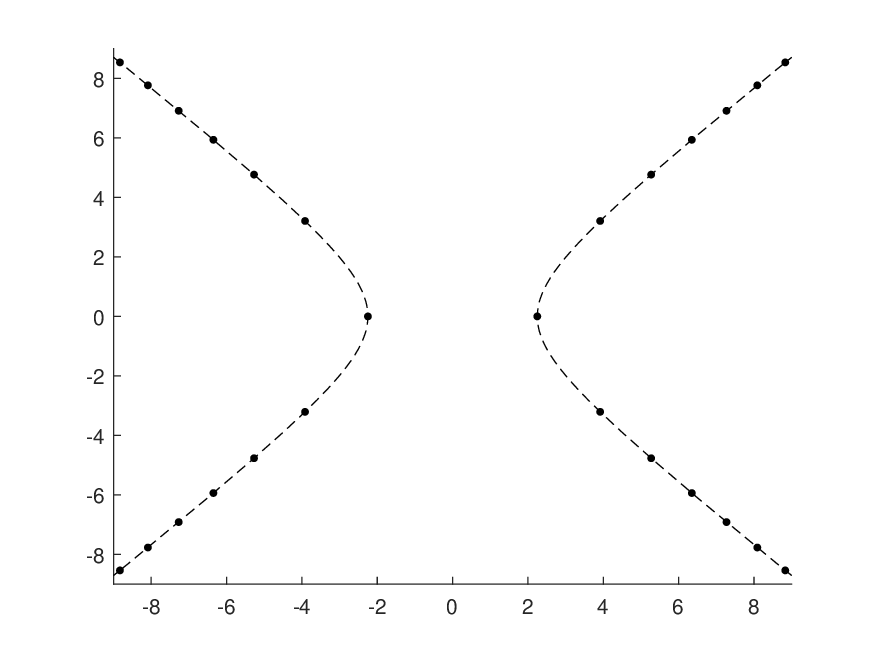}
	\caption{The roots for the equation $e^{-\zeta^2/2}=\z$ with $\z=0.08$, the dashed lines are the corresponding level curve $|e^{-\zeta^2/2}|=|\z|$ for the same $\z$.}
	\label{fig: limitingroots}
\end{figure}

\begin{definition}[Limiting function] \label{defn:Fic}
	Let $\mathbf{x} = (\mathrm{x}_1,\cdots,\mathrm{x_m})$, $\boldsymbol{\gamma} = (\gamma_1,\cdots,\gamma_m)$, and $\boldsymbol{\tau} = (\tau_1,\cdots,\tau_m)$ be points in $\mathrm{R}^m$ such that $\mathrm{p}_j = (\gamma_j,\tau_j) \in [0,1]\times \mathbb{R}_{>0}$. 
	Assume that 
	\begin{equation*}
		0< \tau_1\le \cdots\le \tau_m
	\end{equation*}
	and that $\mathrm{x}_i<\mathrm{x}_{i+1}$ when $\tau_i=\tau_{i+1}$ for $i=1,\cdots,m-1$.
	Define 
	\begin{equation}
		\label{eq:def_Fic}
		\Fic^{\mathrm{per}}(\mathrm{x}_1,\cdots, \mathrm{x}_m; \mathrm{p}_1,\cdots,\mathrm{p}_m) := \oint\cdots\oint \mathrm{C}_{\mathrm{ic}}^{\mathrm{per}}(\vec \z) \mathrm{D}_{\mathrm{ic}}^{\mathrm{per}} (\vec \z) \ddbar{\mathrm{z}_m}{\mathrm{z}_{m}} \cdots \ddbar{\mathrm{z}_1}{\mathrm{z}_{1}}, 
	\end{equation}
	where $\vec \z = (\mathrm{z}_1,\cdots,\mathrm{z}_m)$ and the contours are nested circles satisfying $0<|\mathrm{z}_m|<\cdots<|\mathrm{z}_1|<1$ and also, $r_1<|\mathrm{z}_1|<r_2$ with $r_1,r_2$ being the constants in Assumption~\ref{def:asympstab} (B).
	The first function in the integrand is given by 
	\begin{equation}\label{eq:def_rmC}
		\mathrm{C}_{\mathrm{ic}}^{\mathrm{per}} (\vec \z) := E_{\mathrm{ic}} (\mathrm{z}_1) \mathrm{C}_\mathrm{step}^{\mathrm{per}} (\vec \z).
	\end{equation}
	The second function is given by
	\begin{equation} \label{eq:def_rmD}
		\mathrm{D}_{\mathrm{ic}}^{\mathrm{per}} (\vec \z) := \det (I-  \mathrm{K}_{1}^{\mathrm{per}} \mathrm{K}_{\mathrm{ic}}^{\mathrm{per}}), 
	\end{equation}
	where $\mathrm{K}_{1}^{\mathrm{per}}: \ell^2(\mathrm{S}_2) \to \ell^2(\mathrm{S}_1)$ and 
	$\mathrm{K}_{\mathrm{ic}}^{\mathrm{per}}: \ell^2(\mathrm{S}_1) \to \ell^2(\mathrm{S}_2)$ are given by $\mathrm{K}_{1}^{\mathrm{per}}:= \mathrm{K}_{\mathrm{step},1}^{\mathrm{per}}$ and 
	\begin{equation*}
		\mathrm{K}_{\mathrm{ic}}^{\mathrm{per}} (\zeta,\zeta'):= 
		\begin{dcases}
			\cchi_{\mathrm{ic}} (\zeta,\zeta';\mathrm{z}_1) \mathrm{K}_{\mathrm{step},2}^{\mathrm{per}}(\zeta,\zeta'), \quad  & \text{if $\zeta\in\mathrm{R}_{\mathrm{z}_1} \text{ and } \zeta'\in\mathrm{L}_{\mathrm{z}_1}$,}\\
			\mathrm{K}_{\mathrm{step},2}^{\mathrm{per}} (\zeta,\zeta'), &\text{otherwise.}
		\end{dcases}
	\end{equation*}
\end{definition}
The function  $\mathrm{C}_{\mathrm{step}}^{\mathrm{per}}(\vec \z)$ and kernels $\mathrm{K}_{\mathrm{step},1}^{\mathrm{per}}$ and $\mathrm{K}_{\mathrm{step},2}^{\mathrm{per}}$ are first obtained in \cite{baik2019multi} and the definitions will be recalled in the next two sections for completeness.

\subsection{The factor $\mathrm{C}_{\mathrm{step}}^{\mathrm{per}} (\vec \z)$} \label{sec:Cforsteplimit}


Let $\polylog_s(\mathrm{z})$ be the polylogarithm function which is defined by
\begin{equation}\label{eq: polylogf}
	\polylog_s(\mathrm{z}) := \sum_{k=1}^\infty \frac{\mathrm{z}^k}{ k^s} = \frac{\z}{\Gamma(s)}\int_{0}^{\infty} \frac{t^{s-1}}{e^{t}-\z}dt,
\end{equation}
for $|\mathrm{z}|<1$ and $s\in\mathbb{C}$. 
Set
\begin{equation}\label{eq: polylog}
	A_1(\mathrm{z}):= -\frac{1}{\sqrt{2\pi}} \polylog_{3/2}(\mathrm{z}) \quad \text{and} \quad  A_2(z):=-\frac{1}{\sqrt{2\pi}} \polylog_{5/2}(\mathrm{z}).
\end{equation}
Let $\log \mathrm{z}$ denote the principal branch of the logarithm function with branch cut $\mathbb{R}_{\le 0}$. 
Set
\begin{equation} \label{eq:deofBz}
	B(\mathrm{z},\mathrm{z}') := \frac{\mathrm{z} \mathrm{z}'}{2} 
	\int\int \frac{\eta\xi \log(-\xi+\eta)}{(e^{-\xi^2/2} - \mathrm{z}) (e^{-\eta^2/2} - \mathrm{z}')}
	\ddbar{\xi}{} \ddbar{\eta}{}
	= \frac{1}{4\pi} \sum_{k,k'\ge 1} \frac{\mathrm{z}^k (\mathrm{z}')^{k'}}{(k+k')\sqrt{kk'}}
\end{equation}
for $0<|\mathrm{z}|, |\mathrm{z}'|<1$
where the integral contours are the vertical lines $\Re(\xi) = \mathrm{a}$ and $\Re (\eta) = \mathrm{b}$ with  constants $\mathrm{a}$ and $\mathrm{b}$ satisfying
$-\sqrt{-\log|\mathrm{z}|} <\mathrm{a} <0 <\mathrm{b} <\sqrt{-\log|\mathrm{z}'|}$. We also set $B(\mathrm{z}):= B(\mathrm{z},\mathrm{z})$. 
One can check that 
\begin{equation}\label{eq:def_Bz}
	B(\mathrm{z})= B(\mathrm{z},\mathrm{z}) =\frac{1}{4\pi} \int_0^{\mathrm{z}} \frac{(\polylog_{1/2}(\mathrm{y}))^2}{\mathrm{y}} \dd \mathrm{y}.
\end{equation}

\begin{definition}
	\label{def:limSstep}
	For $\vec \z = (\mathrm{z}_1,\cdots,\mathrm{z}_m)$ satisfying $0<|\mathrm{z}_j|<1$ and $\mathrm{z}_j \ne \mathrm{z}_{j+1}$ for all $j$, we define
	\begin{equation*}	
		\mathrm{C}_{\mathrm{step}}^{\mathrm{per}}(\vec \z)
		:= \left[ \prod_{\ell=1}^{m} \frac{\mathrm{z}_\ell}{\mathrm{z}_\ell -\mathrm{z}_{\ell+1}} \right]
		\left[ \prod_{\ell=1}^{m} \frac{e^{\mathrm{x}_\ell A_1(\mathrm{z}_\ell) +\tau_\ell A_2(\mathrm{z}_\ell)}} {e^{\mathrm{x}_\ell A_1(\mathrm{z}_{\ell+1}) +\tau_\ell A_2(\mathrm{z}_{\ell+1})}}
		e^{2B(\mathrm{z}_\ell) - 2 B(\mathrm{z}_{\ell+1}, \mathrm{z}_\ell)}
		\right],
	\end{equation*}
	where we set $\mathrm{z}_{m+1}=0$. 
\end{definition}

Note that $\mathrm{C}_{\mathrm{step}}^{\mathrm{per}}(\vec \z)$, and hence $\mathrm{C}_{\mathrm{ic}}^{\mathrm{per}}(\vec \z)$, depend on $\mathrm{x}_i$ and $\tau_i$, but not the spatial parameters $\gamma_i$.

\subsection{The operators $\mathrm{K}_{\mathrm{step},1}^{\mathrm{per}}$ and $\mathrm{K}_{\mathrm{step},2}^{\mathrm{per}}$} \label{sec:Dforsteplimit}

Set 
\begin{equation} \label{eq:def_h}
	\mathrm{h} (\zeta,\mathrm{z}) := \frac{\mathrm{z}}{2\pi \mathrm{i}} 
	\int_{\mathrm{i} \mathbb{R}} \frac{\mathrm{w} \log (\mathrm{w}-\zeta)}{e^{-\mathrm{w}^2/2}-\mathrm{z}} \dd \mathrm{w}
	\qquad \text{for $\Re(\zeta)<0$ and $|\mathrm{z}|<1$} 
\end{equation}
and define
\begin{equation}
\label{eq:2019_10_18_02}
\mathrm{h} (\zeta,\mathrm{z}) := \mathrm{h} (-\zeta,\mathrm{z}) \qquad \text{for $\Re(\zeta)>0$ and $|\mathrm{z}|<1$.} 
\end{equation}
For each $i$, define 
\begin{equation}\label{eq: limitingf}
	\mathrm{f}_i(\zeta) := \begin{dcases}
		e^{-\frac{1}{3}(\tau_i-\tau_{i-1})\zeta^3 +\frac{1}{2} (\gamma_i-\gamma_{i-1}) \zeta^2 +(\mathrm{x}_i -\mathrm{x}_{i-1})\zeta}& \text{for $\Re(\zeta)<0$,}\\
		e^{\frac{1}{3}(\tau_i-\tau_{i-1})\zeta^3 -\frac{1}{2} (\gamma_i-\gamma_{i-1}) \zeta^2 -(\mathrm{x}_i -\mathrm{x}_{i-1})\zeta}& \text{for $\Re(\zeta)>0$,}
	\end{dcases}
\end{equation}
where we set $\tau_0=\gamma_0=\mathrm{x}_0=0$. 
We also define 
\begin{equation*}
	\mathrm{Q}_1(j) := 1 -\frac{\mathrm{z}_{j-(-1)^j}}{\mathrm{z}_j} \quad \text{and} \quad
	\mathrm{Q}_2(j) := 1 -\frac{\mathrm{z}_{j+(-1)^j}}{\mathrm{z}_j},
\end{equation*}
where we set $\mathrm{z_0} =\mathrm{z}_{m+1}=0$.

\begin{definition} \label{def:mathrmK}
	Let $\mathrm{S}_1$ and $\mathrm{S}_2$ be the discrete sets defined in \eqref{eq:limSo} and \eqref{eq:limSt}. 
	Let
	\begin{equation*}
		\mathrm{K}_{\mathrm{step},1}^{\mathrm{per}}  : \ell^2(\mathrm{S}_2) \to \ell^2(\mathrm{S}_1)
		\quad \text{and} \quad  
		\mathrm{K}_{\mathrm{step},2}^{\mathrm{per}}  : \ell^2(\mathrm{S}_1) \to \ell^2(\mathrm{S}_2)
	\end{equation*}
	denote the operators defined by their kernels 
	\begin{equation*}
		\mathrm{K}_{\mathrm{step},1}^{\mathrm{per}}  (\zeta,\zeta') 
		:= (\delta_i(j) + \delta_i(j + (-1)^i)) 	\frac{\mathrm{f}_i(\zeta) e^{2\mathrm{h}(\zeta,\mathrm{z}_i)
				-\mathrm{h}(\zeta,\mathrm{z}_{i-(-1)^i}) - \mathrm{h}(\zeta',z_{j-(-1)^j})}}{\zeta(\zeta-\zeta')} \mathrm{Q}_1(j)
	\end{equation*}
	and
	\begin{equation*}
		\mathrm{K}_{\mathrm{step},2}^{\mathrm{per}}  (\zeta',\zeta) 
		:= (\delta_j(i) + \delta_j(i - (-1)^j)) 
		\frac{\mathrm{f}_j(\zeta') e^{2\mathrm{h}(\zeta',\mathrm{z}_j) -\mathrm{h}(\zeta',\mathrm{z}_{j+(-1)^j}) - \mathrm{h}(\zeta,z_{i+(-1)^i}) }}{\zeta'(\zeta'-\zeta)}
		\mathrm{Q}_2(i)
	\end{equation*}
	for
	\begin{equation*}
		\zeta \in (\mathrm{L}_{\mathrm{z}_i}\cup \mathrm{R}_{\mathrm{z}_i}) \cap \mathrm{S}_1 \quad \text{and} \quad
		\zeta' \in (\mathrm{L}_{\mathrm{z}_j}\cup\mathrm{R}_{\mathrm{z}_j}) \cap \mathrm{S}_2
	\end{equation*}
	with $1\le i,j\le m$. Here $\mathrm{L}_{\z_i}$ and $\mathrm{R}_{\z_i}$ are again defined in Definition~\ref{def: limitingroots}.
\end{definition}

\bigskip

\section{Proof of Theorem \ref{thm:main}}\label{sec: proof_limittheorem}
In this section we discuss the proof of Theorem~\ref{thm:main}. The ideas are similar to the one in \cite{baik2019multi,baik2019general} so we omit some technical details. Clearly the theorem follows immediately from the following two lemmas, dealing with the asymptotics of $\mathscr{C}_{\vec y}^{(L)}(z)$ and $\mathscr{D}_{\vec y}^{(L)}(z)$ appearing in the finite-time formula \eqref{eq: multipoint general}, respectively.

\begin{lemma}[Asymptotics of $\mathscr{C}_{\vec y}^{(L)}(\vec z)$]\label{lemma: C(z)}
	Under the same assumption as in Theorem~\ref{thm:main}, we have for fixed $0<\epsilon<1/2$
	\begin{equation}
		\mathscr{C}_{\vec y}^{(L)}(\vec z) = \mathrm{C}_{\mathrm{ic}}^{\mathrm{per}}(\vec \z)\left(1+O(L^{\epsilon-1/2})\right),\quad \mathrm{as}\  L\to \infty.
	\end{equation}
The functions $\mathscr{C}_{\vec y}^{(L)}(z)$ and $\mathrm{C}_{\mathrm{ic}}^{\mathrm{per}}(\z)$ are defined in \eqref{eq: C(z) general} and \eqref{eq:def_rmC} respectively, with $z_i^L=(-1)^N\mathbbm{r}_c^L\z_i$, for $1\leq i\leq m$.
\end{lemma}

\begin{lemma}[Asymptotics of $\mathscr{D}_{\vec y}^{(L)}(\vec z)$]\label{lemma: D(z)}
	Under the same assumption as in Theorem~\ref{thm:main}, we have the convergence 
	\begin{equation}
		\lim_{L\to \infty}\mathscr{D}_{\vec y}^{(L)}(\vec z) = \mathrm{D}_{\mathrm{ic}}^{\mathrm{per}}(\vec \z),
	\end{equation}
where the Fredholm determinants $\mathscr{D}_{\vec y}^{(L)}(\vec z)$ and $\mathrm{D}_{\mathrm{ic}}^{\mathrm{per}}(\vec \z)$ are defined in Definition~\ref{def:D general} and \eqref{eq:def_rmD}, respectively,  and the convergence is locally uniform in $\vec \z$. Here again $z_i$ and $\z_i$ are related by the equation $z_i^L=(-1)^N\mathbbm{r}_c^L\z_i$.
\end{lemma}

The rest of the section is devoted to proving Lemma~\ref{lemma: C(z)} and \ref{lemma: D(z)}. We start with a discussion on the asymptotic behaviors of the roots of the Bethe polynomial $q_z(w)= w^N(1+w)^{L-N}-z^L(1+pw)^N$ under the critical re-scaling in Section~\ref{sec: rootasymptotics}. Then in Section~\ref{sec: productasymptotics} we list a few lemmas discussing the asymptotics of several products involving these roots under the critical re-scaling. With these preparations we prove Lemma \ref{lemma: C(z)} and Lemma \ref{lemma: D(z)} in Section~\ref{sec: proofC} and Section~\ref{sec: proofD}, respectively. Finally in Section~\ref{sec: verifyIC} we verify the Assumption \ref{def:asympstab} for the classical step and flat initial conditions.

\subsection{Asympotics of the Bethe roots}\label{sec: rootasymptotics}
We assume that the particle density $\rho:=N/L$ stays within a compact subset of $(0,1)$ for all $L$. From the discussion in Section \ref{sec: bethe roots} we know the level set $\{w\in \mathbb{C}: |w^N(1+w)^{L-N}|=|z^L(1+pw)^N|\}$ consists of two disjoint closed contour for $|z|<\mathbbm{r}_c$  so we can define:
\begin{definition}
	Given $|z|<\mathbbm{r}_c$, we define two closed contours $\Lambda_L$ and $\Lambda_R$ by 
	\begin{equation}
		\begin{aligned}
			\Lambda_L&:= \{w\in \mathbb{C}: |w^N(1+w)^{L-N}|=|z^L(1+pw)^N|\}\cap \{\mathrm{Re}(w)<w_c\},\\
			\Lambda_R&:= \{w\in \mathbb{C}: |w^N(1+w)^{L-N}|=|z^L(1+pw)^N|\}\cap \{\mathrm{Re}(w)>w_c\}.
		\end{aligned}
	\end{equation}
\end{definition}
A formal Taylor expansion at $w=w_c$ indicates that as $L\to \infty$, the Bethe equation $w^N(1+w)^{L-N}=z^L(1+pw)^N$ converges to the equation 
\begin{equation}\label{eq: LimitBethe}
	e^{-\zeta^2/2}=\z,
\end{equation}
  where $z^L = (-1)^{N}\mathbbm{r}_c^L \z$ and 
\begin{equation}\label{eq: rescaling w}
	w = w_c +\frac{1+\nu-2\rho}{1+\nu}\sqrt{\frac{\rho}{(1-\rho)\nu}}\zeta L^{-1/2}:= w_c +c_0\zeta L^{-1/2},
\end{equation}
where $\nu:= \sqrt{1-4p\cdot \rho(1-\rho)}$ and $c_0:=\frac{1+\nu-2\rho}{1+\nu}\sqrt{\frac{\rho}{(1-\rho)\nu}} $. The solution of equation \eqref{eq: LimitBethe}  is a discrete set given by $\{\pm\sqrt{-2\log \z+4k\pi i}: k\in \mathbb{Z}\}$ for an arbitrary choice of branch of the logarithm and square root, see Figure \ref{fig: limitingroots}. Lemma \ref{lemma: convergenceofroots} below precisely quantifies the convergence of the Bethe roots near $w=w_c$ to the corresponding roots of the limiting equation.

\begin{lemma}\label{lemma: convergenceofroots}
	For any $0<\epsilon<1/8$ and $|z|<\mathbbm{r}_c$ fixed, we define
	\begin{equation}
		\mathcal{L}_{z}^{(\epsilon)}:= \mathcal{L}_{z} \cap \mathbb{D}(w_c, c_0 L^{-1/2+\epsilon}),
	\end{equation}
	where $\mathbb{D}(a, r)$ is a disc centered at $a$ with radius $r$ and $c_0$ is defined in \eqref{eq: rescaling w}. Then for the re-scaled parameter $\z= (-1)^N z^L\mathbbm{r}_c^{-L}$ we have an injective map $\mathcal{M}_{L,\text{left}}: \mathcal{L}_{z}^{(\epsilon)}\to \mathrm{L}_{\z}$ satisfying
	\begin{equation}
		\left|\mathcal{M}_{L,\text{left}}(u)- L^{1/2}c_0^{-1}(u-w_c)\right|\leq L^{-1/2+3\epsilon}\log L,
	\end{equation}
	for all $ u\in \mathcal{L}_{z}^{(\epsilon)}$ and $L$ large enough. Furthermore, the map satisfies
	\begin{equation}
		\mathrm{L}_{\z}\cap \mathbb{D}(0,L^{\epsilon}-1)\subset \mathcal{M}_{L,\text{left}}(\mathcal{L}_{z}^{(\epsilon)})\subset \mathrm{L}_{\z}\cap \mathbb{D}(0,L^{\epsilon}+1).
	\end{equation} 
	Similar results hold if we replace $\mathcal{L}_{z}$ and $\mathrm{L}_{\z}$ by $\mathcal{R}_{z}$ and $\mathrm{R}_{\z}$.
\end{lemma}
\begin{proof}
	This lemma is a minor generalization of Lemma 8.1 of \cite{baik2018relax} by allowing one extra parameter $p$ in the Bethe equation. The proof is almost identical to the one in \cite{baik2018relax} so we omit the details.
\end{proof}
\subsection{Asymptotics of various products over Bethe roots}\label{sec: productasymptotics}
In this section we collect all the results involving limits of  products of the Bethe roots appearing in the finite-time formula that are independent of the parameters $\vec y$, $a_i$, $t_i$ and $k_i$. The starting point is the following simple integral formula for the sums of functions evaluated at the left or right Bethe roots. 
\begin{lemma}\label{lemma: integral formula roots}
	Let $\phi(w)$ be a function analytic in the interior and a neighborhood of $\Lambda_R$. Then
	\begin{equation}\label{eq: residue_right}
		\sum_{v\in \mathcal{R}_z} \phi(v) = N\phi(0)+\oint_{\Sigma_{\RR}}\frac{dw}{2\pi i}\frac{z^L(1+pw)^N}{q_z(w)}\frac{\phi(w)}{J(w)},
	\end{equation}
	where we recall that $J(w)=\frac{w(w+1)(1+pw)}{N+Lw+p(L-N)w^2}$. Similarly if $\phi(w)$ be a function analytic in the interior and a neighborhood of $\Lambda_L$, then
	\begin{equation}\label{eq: residue_left}
		\sum_{u\in \mathcal{L}_z} \phi(u) = (L-N)\phi(-1)+\oint_{\Sigma_{\LL}}\frac{\dd w}{2\pi \ii}\frac{z^L(1+pw)^N}{q_z(w)}\frac{\phi(w)}{J(w)}.
	\end{equation}
	Here $\Sigma_{\LL}$ and $\Sigma_{\RR}$ are simple closed contours lie in the half-plane $\{w:\mathrm{Re}(w)<w_c\}$ (respectively $\{w:\mathrm{Re}(w)>w_c\}$) with $\Lambda_{\mathrm{L}}$ (respectively $\Lambda_{\mathrm{R}}$) inside. Taking the function $\phi$ to be the constant function $1$ implies in particular that 
	\begin{equation}
		\oint_{\Sigma_{\LL}}\frac{dw}{2\pi i}\frac{z^L(1+pw)^N}{q_z(w)}\frac{1}{J(w)}=\oint_{\Sigma_{\RR}}\frac{\dd w}{2\pi \ii}\frac{z^L(1+pw)^N}{q_z(w)}\frac{1}{J(w)}=0.
	\end{equation}
\end{lemma}
\begin{proof}
	A direct differentiation shows
	\begin{equation*}
		\begin{aligned}
			q_z'(w) = q_z(w)\cdot \left(\frac{N}{w}+\frac{L-N}{w+1}\right)+z^L(1+pw)^N\cdot \frac{1}{J(w)}.
		\end{aligned}
	\end{equation*}
	Hence by the residue theorem we know 
	\begin{equation*}
		\begin{aligned}
			\sum_{v\in \mathcal{R}_z} \phi(u) = \oint_{\Sigma_R}\frac{\dd w}{2\pi \ii }\frac{q_z'(w)}{q_z(w)}\phi(w)
			= N\phi(0)+\oint_{\Sigma_R}\frac{\dd w}{2\pi \ii }\frac{z^L(1+pw)^N}{q_z(w)}\frac{1}{J(w)}\phi(w).
		\end{aligned}
	\end{equation*}
	The proof for \eqref{eq: residue_left} is similar.
\end{proof}
As taking logarithm transforms products into sums, the following lemma is a direct consequence of Lemma \ref{lemma: integral formula roots} and the method of steepest descent.
\begin{lemma}\label{lemma: asymptoticsofproducts}
	Given $|z|<\mathbbm{r}_c$ and $\z=(-1)^N z^L \mathbbm{r}_c^{-L}$. Suppose $\rho=N/L$ stays in a compact subset of $(0,1)$, then for every $0<\epsilon<1/2$ the following holds for all large enough $L$. 
	\begin{enumerate}[(i)]
		\item For $w=w_c + c_0\zeta L^{-1/2}$ with $|\zeta|\leq  L^{\epsilon/4}$, where  $c_0=\frac{1+\nu-2\rho}{1+\nu}\sqrt{\frac{\rho}{(1-\rho)\nu}}$ and $\nu=\sqrt{1-4p\cdot\rho(1-\rho)}$,
		\begin{equation}
			\frac{z^L(1+pw)^N}{q_z(w)} = \frac{\z}{e^{-\zeta^2/2}-\z}\cdot \left(1+O(L^{-1/2+\epsilon})\right).
		\end{equation}
		\item On the other hand if $|w-w_c|\geq C\cdot L^{\epsilon-1/2}$ for some $C>0$, we have for some $c>0$ and $\alpha>0$ ,
		\begin{equation}
			\left|\frac{z^L(1+pw)^N}{q_z(w)}\right|\leq e^{-cL^{\alpha}}.
		\end{equation}
		\item For $w$ of $O(1)$ distance away from $0,-1$ and $-1/p$ we have $\left|\frac{1}{J(w)}\right| \leq C\cdot L$ for some constant $C>0$. Furthermore if $w=w_c + c_0\zeta L^{-1/2}$ with $|\zeta|\leq C\cdot L^{\epsilon}$ we have 
		\begin{equation}\label{eq: J(w) taylor}
			\frac{1}{J(w)} = -c_0^{-1}\zeta L^{1/2}\cdot \left(1+O(L^{\epsilon-1/2})\right).
		\end{equation}
		\item 
		For $w= w_c +c_0\zeta L^{-1/2}$ with $|\zeta|\leq L^{\epsilon/4}$, we have
		\begin{equation}
			\begin{split}
				\prod_{u\in \mathcal{L}_z} \sqrt{w-u} &= (\sqrt{w+1})^{L-N} e^{\frac12 \mathrm{h} (\zeta,\mathrm{z}) } \left( 1+ O(L^{\epsilon-1/2}\log L)\right) \quad \text{if $\Re(\zeta) \ge 0$},\\
				\prod_{v\in \mathcal{R}_z} \sqrt{v-w} &= (\sqrt{-w})^N e^{\frac12 \mathrm{h} (\zeta,\mathrm{z}) } \left( 1+ O(L^{\epsilon-1/2}\log L)\right) \qquad \text{if $\Re(\zeta) \le 0$},
			\end{split} 
		\end{equation}
		where $\mathrm{h} (\zeta,\mathrm{z})$ is the function defined in \eqref{eq:def_h}. When $\Re(\zeta)=0$, $\mathrm{h}(\zeta, \mathrm{z})$ is the limit of $\mathrm{h}(\eta, \mathrm{z})$ as $\eta\to \zeta$ from $\Re(\eta)>0$ for the first case and from $\Re(\eta)<0$ for the second case. 
		\item For $w$ of a finite distance $O(1)$ away from the trajectory $\Lambda_{\mathrm{L}}\cup\Lambda_{\mathrm{R}}$, 
		\begin{equation} \begin{aligned}
				\prod_{u\in \mathcal{L}_z} \sqrt{w-u} &= (\sqrt{w+1})^{L-N} \left( 1+ O(L^{\epsilon-1/2})\right) \quad \text{if\  $\Re(w)>w_c$}, \\
				\prod_{v\in \mathcal{R}_z} \sqrt{v-w} &= (\sqrt{-w})^N \left( 1+ O(L^{\epsilon-1/2})\right) \qquad \text{if\  $\Re(w)<w_c$.}
			\end{aligned} 
		\end{equation}
		
		\item 
		There is a constant $C>0$ such that for every $w$ satisfying $|w-w_c|\geq L^{\epsilon-1/2}$, 
		\begin{equation*}
			e^{-CL^{-\epsilon}}\le \left|\frac{q_{z,\mathrm{L}}(w)}{(w +1)^{L-N}} \right| \le e^{CL^{-\epsilon}} 
			\quad \text{if $\Re(w)>w_c$},
		\end{equation*}
		and
		\begin{equation*}
			e^{-CL^{-\epsilon}}\le \left|\frac{q_{z,\mathrm{R}}(w)}{w^N} \right| \le e^{CL^{-\epsilon}}
			\qquad \text{if $\Re(w)< w_c$.}
		\end{equation*}

		\item 
		We have 
		\begin{equation}
			\frac{\prod_{v\in\mathcal{R}_z} \prod_{u\in\mathcal{L}_{z'}} \sqrt{v-u}}
			{\prod_{u\in\mathcal{L}_{z'}} \left(\sqrt{-u}\right)^N \prod_{v\in\mathcal{R}_z} \left( \sqrt{v+1} \right)^{L-N}}
			= 
			e^{-B(\mathrm{\z,\z'})} \left(1+O(L^{\epsilon-1/2})\right),
		\end{equation}
		where $B(\mathrm{\z,\z'})$ is the function defined in \eqref{eq:deofBz}. 
	\end{enumerate}
\end{lemma}
\begin{proof}
	These estimates are again one-parameter generalizations of Lemma 8.2 and Lemma 8.4 of \cite{baik2018relax}. The main difference is due to the extra parameter $p$ in the Bethe polynomial $q_z(w)=w^N(1+w)^{L-N}-z^L(1+pw)^N$, the proper critical point for steepest descent analysis now is at $w_c=-\frac{2\rho}{1+\sqrt{1-4p\cdot\rho(1-\rho)}}$, which comes from the larger root of the quadratic equation $p(L-N)w^2+Lw+N=0$, as opposed to $w=-\rho$ for the $p=0$ degeneration discussed in \cite{baik2018relax}. A standard steepest descent analysis using integral representations obtained in Lemma \ref{lemma: integral formula roots} with critical point $w_c$ yields all the estimates. We omit the details.  
\end{proof}
\medskip
\subsection{Asymptotics of \texorpdfstring{$\mathscr{C}_{\vec y}^{(L)}(\vec z)$}{Lg}}\label{sec: proofC} Now we are ready to prove Lemma~\ref{lemma: C(z)}. Recall that (see Definition~\ref{def: C general}), $\mathscr{C}_{\vec y}^{(L)}(\vec z)= \mathcal{E}_{\vec y}(z_1)\mathscr{C}_{\mathrm{step}}^{(L)}(\vec z)$ where  $\mathcal{E}_{\vec y}(z_1)= E_{\mathrm{ic}}(\z_1)\left(1+O(L^{\epsilon-1/2})\right)$ as $L\to \infty$ due to Assumption~\ref{def:asympstab}. Hence it suffices to prove 
\begin{equation}\label{eq: asymptoticsofCstep}
	\mathscr{C}_{\mathrm{step}}^{(L)}(\vec z) = \mathrm{C}_{\mathrm{step}}^{\mathrm{per}}(\vec \z)\left(1+O(L^{\epsilon-1/2})\right),\quad  \mathrm{as}\  L\to \infty,
\end{equation}
 where $\mathrm{C}_{\mathrm{step}}^{\mathrm{per}}(\vec \z)$ is defined in Definition~\ref{def:limSstep} and $\mathscr{C}_{\mathrm{step}}^{(L)}(\vec z)$ is defined as follows (see also Definition~\ref{def: C general}): 
\begin{equation*}
	\begin{aligned}
		\mathscr{C}_{\text{step}}^{(L)}(\vec z) &:= \left[\prod_{\ell=1}^{m}\frac{E_{\ell}(z_{\ell})}{E_{\ell-1}(z_{\ell})}\right]\left[\prod_{\ell=1}^{m}\frac{\prod_{u\in \mathcal{L}_{z_\ell}}(-u)^N\prod_{v\in \mathcal{R}_{z_{\ell}}}(v+1)^{L-N}}{\Delta(\mathcal{R}_{z_{\ell}};\mathcal{L}_{z_{\ell}})}\right]\\
		&\cdot \left[\prod_{\ell=2}^{m}\frac{z_{\ell-1}^L}{z_{\ell-1}^L-z_{\ell}^L}\right]\left[\prod_{\ell=2}^{m}\frac{\Delta(\mathcal{R}_{z_{\ell}};\mathcal{L}_{z_{\ell-1}})}{\prod_{u\in \mathcal{L}_{z_{\ell-1}}}(-u)^N\prod_{v\in \mathcal{R}_{z_{\ell}}}(v+1)^{L-N}}\right].
	\end{aligned}
\end{equation*}
Here $E_{\ell}(z):= \prod_{u\in \mathcal{L}_z} (-u)^{-k_\ell}\prod_{v\in \mathcal{R}_z}(v+1)^{-a_\ell-k_\ell}(pv+1)^{t_\ell-k_\ell}$ with $E_0(z):=1$. Under the re-scaling $z_\ell^L= (-1)^N \mathbbm{r}_c^L \z_\ell$ we clearly have $\prod_{\ell=2}^{m}\frac{z_{\ell-1}^L}{z_{\ell-1}^L-z_{\ell}^L}= \prod_{\ell=2}^{N}\frac{\z_{\ell-1}}{\z_{\ell-1}-\z_{\ell}}$. On the other hand by Lemma~\ref{lemma: asymptoticsofproducts} (iv) 
\begin{equation*}
\begin{aligned}
	&\frac{\prod_{u\in\mathcal{L}_{z_\ell}}(-u)^N\prod_{v\in \mathcal{R}_{z_{\ell}}}(v+1)^{L-N}}{\Delta(\mathcal{R}_{z_{\ell}};\mathcal{L}_{z_{\ell}})} = e^{B(\z_\ell)}\left(1+O(L^{\epsilon-1/2})\right),\\
    &\frac{\Delta(\mathcal{R}_{z_{\ell}};\mathcal{L}_{z_{\ell-1}})}{\prod_{u\in \mathcal{L}_{z_{\ell-1}}}(-u)^N\prod_{v\in \mathcal{R}_{z_{\ell}}}(v+1)^{L-N}}=e^{-B(\z_\ell,\z_{\ell-1})}\left(1+O(L^{\epsilon-1/2})\right).
\end{aligned}
\end{equation*}
Hence \eqref{eq: asymptoticsofCstep} follows immediately once we establish the following lemma on the asymptotics of $E_{\ell}(z)$.
\begin{lemma}\label{lemma: asymptoticsE}
	Let $E(z)= E(z;a,k,t):= \prod_{u\in \mathcal{L}_z} (-u)^{-k}\prod_{v\in \mathcal{R}_z}(v+1)^{-a-k}(pv+1)^{t-k}$ where $a,k\in \mathbb{Z}$ and $t\in \mathbb{N}$ are given parameters. Then for $z^L= (-1)^N\mathbbm{r}_c^L\z$ with $0<|\z|<1$ and the parameters satisfying 
	\begin{equation}\label{eq: rescaling parameters}
		t = c_1\tau L^{3/2}+O(1),\quad a= c_2 t+\gamma L+O(1),\quad k=c_3 t+c_4\gamma L+c_5\mathrm{x} L^{1/2}+O(1),
	\end{equation}
we have for $L$ large enough and fixed $0<\epsilon<1/2$
\begin{equation}\label{eq: asymptoticsofE}
	E(z) = e^{\mathrm{x} A_1(\z)+\tau A_2(\z)}\cdot \left(1+O(L^{\epsilon-1/2})\right).
\end{equation}
where $A_1(\z)$ and $A_2(\z)$ are scaled polylogarithm functions as in \eqref{eq: polylog}. The constants $c_i$ are the same as in \eqref{eq: scaling constants} and the re-scaled parameters are chosen such that $\tau>0$, $\gamma\in[0,1]$ and $\mathrm{x}\in \mathbb{R}$. 
\end{lemma}
\begin{proof}
	Note that $\log E(z) = \sum_{u\in \mathcal{L}_z}(-k)\log(-u)+\sum_{v\in \mathcal{R}_z}\left[(-a-k)\log(v+1)+(t-k)\log(1+pv)\right]$. Apply Lemma~\ref{lemma: integral formula roots} to the two sums over left and right Bethe roots and deform both contours $\Sigma_{\LL}$ and $\Sigma_{\RR}$ to the vertical line with real part $w_c= -\frac{2\rho}{1+\sqrt{1-4p\rho(1-\rho)}}$ we have 
	\begin{equation}\label{eq: integralrepE}
		\log E(z) = \int_{w_c-\ii\infty}^{w_c+\ii\infty} \frac{\dd w}{2\pi \ii} \frac{z^L(1+pw)^N}{q_z(w)}\frac{1}{J(w)}\left(G(w)-G(w_c)\right),
	\end{equation}
where $G(w)= (-k)\log(-w)+(a+k)\log(w+1)+(-t+k)\log(1+pw)$. Note that in \eqref{eq: integralrepE} we have used the fact discussed in Lemma~\ref{lemma: integral formula roots} that
\begin{equation*}
	\oint_{\Sigma_{\LL}} \frac{\dd w}{2\pi \ii} \frac{z^L(1+pw)^N}{q_z(w)}\frac{1}{J(w)}G(w_c)=\oint_{\Sigma_{\RR}} \frac{\dd w}{2\pi \ii} \frac{z^L(1+pw)^N}{q_z(w)}\frac{1}{J(w)}G(w_c)=0.
\end{equation*}
Now a  Taylor expansion at $w=w_c$ shows
\begin{equation*}
	G(w)-G(w_c)=G'(w_c)(w-w_c)+\frac{G''(w_c)}{2}(w-w_c)^2+\frac{G'''(w_c)}{6}(w-w_c)^3+ O\left(G^{(4)}(w_c)\cdot (w-w_c)^4\right).
\end{equation*} 
Set $w=w_c+c_0\zeta L^{-1/2}$ where $c_0$ is the same as in \eqref{eq: rescaling w} and assume the parameters are re-scaled as in \eqref{eq: rescaling parameters}. After a tedious but straightforward calculation we see that for $|\zeta|\leq L^{\epsilon/4}$ with $0<\epsilon<1/2$,
\begin{equation}
	G(w)-G(w_c) = -\mathrm{x}\zeta-\frac{\gamma}{2}\zeta^2+\frac{\tau}{3}\zeta^3+O(L^{\epsilon-1/2}).
\end{equation}
Splitting the integral representation for $\log E(\z)$ into two parts with $|\zeta|\leq L^{\epsilon/4}$ and $|\zeta|> L^{\epsilon/4}$ and using  the estimates for $\frac{z^L(1+pw)^N}{q_z(w)}$ and $\frac{1}{J(w)}$ obtained in Lemma~\ref{lemma: asymptoticsofproducts}  we see 
\begin{equation}
\begin{aligned}
	\log E(z) = \int_{-\ii\infty}^{\ii\infty}\frac{\dd\zeta}{2\pi \ii}\frac{\z}{e^{-\zeta^2/2}-\z}\cdot \left(\mathrm{x}\zeta^2+\frac{\gamma}{2}\zeta^3-\frac{\tau}{3}\zeta^4\right)\cdot \left(1+O(L^{\epsilon-1/2})\right) +O(e^{-cL^{\alpha}}),
\end{aligned}
\end{equation}
for some constants $c,\alpha>0$. Now \eqref{eq: asymptoticsofE} follows from integral representations of polylogarithm  \eqref{eq: polylogf}.
\end{proof}
\medskip

\subsection{Asymptotics of \texorpdfstring{$\mathscr{D}_{\vec y}^{(L)}(\vec z)$}{Lg}}\label{sec: proofD} 
Next we discuss the asymptotics of the Fredholm determinant part $\mathscr{D}_{\vec y}(z)$. Note first that by a standard series expansion of Fredholm determinants we have 
\begin{equation}
	\mathscr{D}_{\vec y}^{(L)}(\z) = \sum_{\vec n\in (\mathbb{Z}_{\geq 0})^m} \frac{1}{(n_1!\cdots n_m!)^2} \mathscr{D}_{\vec y,\vec n}^{(L)}(\vec z),
\end{equation}
where $\vec n=(n_1,\cdots,n_m)$ and 
\begin{equation}
	\mathscr{D}_{\vec y,\vec n}^{(L)}(\vec z)= (-1)^{|\vec n|} \sum_{\substack{U^{(\ell)\in (\mathcal{L}_{z_\ell})^{n_\ell}}\\V^{(\ell)\in (\mathcal{R}_{z_\ell})^{n_\ell}}\\ \ell=1,\cdots,m}} \det[\mathscr{K}_{1}^{(L)}(w_i,w_j')]_{i,j=1}^{|\vec n|}\det[\mathscr{K}_{\vec y}^{(L)}(w_i',w_j)]_{i,j=1}^{|\vec n|},
\end{equation}
where $U^{(\ell)}= (u^{(\ell)}_{1},\cdots u^{(\ell)}_{n_\ell})$ and $V^{(\ell)}= (v^{(\ell)}_{1},\cdots v^{(\ell)}_{n_\ell})$ and 
\begin{equation}
	w_i = \begin{cases}
		u^{(\ell)}_k\quad \mathrm{if}\ k=n_1+\cdots+n_{\ell-1}+k\quad \text{for integer $k\leq n_{\ell}$ with $\ell$ odd },\\
		v^{(\ell)}_k\quad \mathrm{if}\ k=n_1+\cdots+n_{\ell-1}+k\quad \text{for integer $k\leq n_{\ell}$ with $\ell$ even },
	\end{cases}
\end{equation}
and 
\begin{equation}
	w_i' = \begin{cases}
		v^{(\ell)}_k\quad \mathrm{if}\ k=n_1+\cdots+n_{\ell-1}+k\quad \text{for integer $k\leq n_{\ell}$ with $\ell$ odd },\\
		u^{(\ell)}_k\quad \mathrm{if}\ k=n_1+\cdots+n_{\ell-1}+k\quad \text{for integer $k\leq n_{\ell}$ with $\ell$ even }.
	\end{cases}
\end{equation}
A similar series expansion holds for the limiting Fredholm determinant $\mathrm{D}_{\mathrm{ic}}^{\mathrm{per}}(\vec \z)$ with $\mathcal{L}_{z_{\ell}}$ and $\mathcal{R}_{z_{\ell}}$ replaced by the limiting roots $\mathrm{L}_{\z_{\ell}}$ and $\mathrm{R}_{\z_{\ell}}$ and the kernels replaced by the limiting kernels $\mathrm{K}_{1}^{\mathrm{per}}$ and $\mathrm{K}_{\mathrm{ic}}^{\mathrm{per}}$ defined in Definition~\ref{def:mathrmK}.
\begin{equation}
	\mathrm{D}_{\mathrm{ic}}^{\mathrm{per}}(\vec \z) = \sum_{\vec n\in (\mathbb{Z}_{\geq 0})^m} \frac{1}{(n_1!\cdots n_m!)^2} \mathrm{D}_{\mathrm{ic},\vec n}^{\mathrm{per}}(\vec \z).
\end{equation}
We will prove the convergence of each of these $\mathscr{D}_{\vec y,\vec n}^{(L)}(\vec z)$ as well as some exponential bounds. 
\begin{lemma}\label{lemma: seriesconvergence}
	Under the same assumption in Theorem \ref{thm:main}, for every fixed $\vec n\in (\mathbb{Z}_{\geq 0})^m$, we have 
	\begin{enumerate}[(i)]
	\item $\mathscr{D}_{\vec y,\vec n}^{(L)}(\vec z)\to \mathrm{D}_{\mathrm{ic},\vec n}^{\mathrm{per}}(\z)$ as $L\to \infty$.
	\item There exists constant $C>0$ such that $|\mathscr{D}_{\vec y,\vec n}^{(L)}(\vec z)|\leq C^{|\vec n|}$ for all $L$ large enough.
	\end{enumerate}
\end{lemma}
It is clear that Lemma~\ref{lemma: D(z)} follows immediately from Lemma~\ref{lemma: seriesconvergence} by dominated convergence theorem. To prove Lemma~\ref{lemma: seriesconvergence} we will prove the convergence of the kernels after proper conjugation for points inside the critical region as well as exponential decay estimates for the kernel at points outside the critical region. The conjugation is as follows: we replace $\mathscr{K}_{1}^{(L)}$ and $\mathscr{K}_{\vec y}^{(L)}$ defined in Definition~\ref{def:D general} by $\tilde{\mathscr{K}}_{1}^{(L)}$ and $\tilde{\mathscr{K}}_{\vec y}^{(L)}$ where 
\begin{equation}\label{eq: conjugated kernel}
	\begin{aligned}
	&\tilde{\mathscr{K}}_{1}^{(L)}(w,w'):= -\left( \delta_i(j) +\delta_i( j + (-1)^i )\right) 
	\frac{ J(w) \sqrt{\tilde{f}_i(w)}\sqrt{\tilde{f}_j(w')} (H_{z_i}(w))^2 }
	{ H_{z_{i-(-1)^i}}(w) H_{z_{j-(-1)^j}}(w') (w-w')} Q_1(j), \\
	&\tilde{\mathscr{K}}_{\vec y}^{(L)}(w,w'):=\left( \delta_j(i) +\delta_j( i - (-1)^j )\right)  \frac{ J(w') \sqrt{\tilde{f}_j(w')}\sqrt{\tilde{f}_i(w)} (H_{z_j}(w'))^2 }
	{ H_{z_{j + (-1)^j}}(w') H_{z_{i + (-1)^i}}(w) (w'-w)} Q_2(i)\Lambda(i,w,w'),
	\end{aligned}
\end{equation}
for $w \in (\mathcal{L}_{z_i} \cup \mathcal{R}_{z_i}) \cap \mathscr{S}_1$ and $w' \in (\mathcal{L}_{z_j} \cup \mathcal{R}_{z_j}) \cap \mathscr{S}_2$. Here 
\begin{equation}
	\tilde{f}_{i}(w):= \begin{cases}
		&\frac{F_i(w)F_{i-1}(w_c)}{F_{i-1}(w)F_{i}(w_c)}\quad \text{for } \mathrm{Re}(w)<w_c,\\
		&\frac{F_i(w)F_{i-1}(w_c)}{F_{i-1}(w)F_{i}(w_c)}\quad \text{for } \mathrm{Re}(w)<w_c.
	\end{cases}
\end{equation}
We define the square root to be $\sqrt{w}= r^{1/2}e^{i\theta/2}$ for $w=re^{i\theta}$ with $-\pi<\theta\leq \pi$. Note that the product of determinants will always be continuous even though the square root function is not since every $(\tilde{f}_i)^{1/2}$ is multiplied twice. We change the limiting kernels  in a similar way by replacing $\mathrm{f_i}$ or $\mathrm{f}_j$ in the kernels with $\sqrt{\mathrm{f_i}}\sqrt{\mathrm{f_j}}$ and denote them as $\tilde{\mathrm{K}}_{1}^{\mathrm{per}}$ and  $\tilde{\mathrm{K}}_{\mathrm{ic}}^{\mathrm{per}}$. Then we have the following asymptotics for the conjugated kernels, which easily implies Lemma~\ref{lemma: seriesconvergence}.
\begin{lemma} \label{lem:K1Kw2estmfo}
	Fix $0<\epsilon<1/(1+2m)$. 
	Let
	\begin{equation}
	\Omega=\Omega_L:= \bigg\{ w\in \mathbb{C}: |w-w_c| \le c_0^{-1}L^{-1/2+\epsilon/4}\bigg \}
	\end{equation}
	be a disk centered at $w_c$.
	Under the same assumption in Theorem~\ref{thm:main} we have 
	\begin{enumerate}[(i)]
		\item As $L\to \infty$, uniformly for $w\in \mathscr{S}_1\cap \Omega$ and $w'\in \mathscr{S}_2\cap \Omega$
		\begin{equation}
		\label{eq:aux_2017_07_01_01}
		|\tilde{\mathscr{K}}_{1}^{(L)}(w,w')| = |\tilde{\mathrm{K}}_{1}^{\mathrm{per}}(\zeta, \zeta')| +O(L^{\epsilon-1/2}\log L), 
		\quad 
		|\tilde{\mathscr{K}}_{\vec y}^{(L)}(w,w')| = |\tilde{\mathrm{K}}_{\mathrm{ic}}^{\mathrm{per}}(\zeta, \zeta')| +O(L^{\epsilon-1/2}\log L),
		\end{equation}
		where $\zeta\in\mathrm{S}_1, \zeta'\in\mathrm{S}_2$ are the image of $w, w'$ under either the map $\mathcal{M}_{L,\text{left}}$ or $\mathcal{M}_{L,\text{right}}$ in Lemma~\ref{lemma: convergenceofroots}. 
		\item 
		As $L\to \infty$, for $w_i\in \mathscr{S}_1\cap \Omega$ and $w_i'\in \mathscr{S}_2\cap \Omega$, 
		\begin{equation*} 
		\det\left[\tilde{\mathscr{K}}_{1}^{(L)}(w_i, w'_j)\right]_{i,j=1}^{|\vec n|} \to \det\left[\tilde{\mathrm{K}}_{1}^{\mathrm{per}}(\zeta_i, \zeta'_j)\right]_{i,j=1}^{|\vec n|}, 
		\quad
		\det\left[\tilde{\mathscr{K}}_{\vec y}^{(L)}(w_i, w'_j)\right]_{i,j=1}^{|\vec n|} \to \det\left[\tilde{\mathrm{K}}_{\mathrm{ic}}^{\mathrm{per}}(\zeta_i, \zeta'_j)\right]_{i,j=1}^{|\vec n|}, 
		\end{equation*}
	    for each $\vec n\in (\mathbb{Z}_{\geq 0})^m$, where $\zeta\in\mathrm{S}_1, \zeta'\in\mathrm{S}_2$ are the image of $w, w'$ under either the map $\mathcal{M}_{L,\text{left}}$ or $\mathcal{M}_{L,\text{right}}$ in Lemma~\ref{lemma: convergenceofroots}.
		\item 
		There are positive constants $c$ and $\alpha$ such that 
		\begin{equation}
		|\tilde{\mathscr{K}}_{1}^{(L)}(w,w')|= O(e^{-cL^{\alpha }} ), 
		\qquad
		|\tilde{\mathscr{K}}_{\vec y}^{(L)}(w',w)|= O(e^{-cL^{\alpha }} )
		\end{equation}
		as $L\to\infty$, uniformly for $w\in \mathscr{S}_1\cap \Omega^c$ and $w'\in \mathscr{S}_2$, and also for $w'\in \mathscr{S}_2\cap \Omega^c$ and $w\in \mathscr{S}_1$. 
	\end{enumerate}
\end{lemma}
\begin{proof}
  Due to the structure of the kernel \eqref{eq: conjugated kernel}, the lemma is proved once we establish the corresponding asymptotics and tail estimates for the functions $J(w)$, $H_{z_i}(w)$ and $\tilde{f}_{i}(w)$. For $J(w)$ and $H_z(w)$ these have already been discussed in Lemma~\ref{lemma: asymptoticsofproducts} (iii),(iv) and (vi). The needed estimates for $\tilde{f}_i(w)$ is summarized in the following Lemma~\ref{lemma: asymptoticskey}, the proof of which is similar to that of Lemma~\ref{lemma: asymptoticsE} so we omit the details.
\end{proof}

\begin{lemma}\label{lemma: asymptoticskey} Under the same assumption as in Theorem~\ref{thm:main} for the parameters $a_i,k_i,t_i$, for $w=w_c+c_0\zeta L^{-1/2}$ with $w\in \mathcal{L}_z\cup \mathcal{R}_z$ we have for fixed $0<\epsilon<1/2$
	\begin{equation}
		\tilde{f}_j(w) = \begin{cases}
			\mathrm{f}_j(\zeta)\left(1+O(L^{\epsilon-1/2})\right)\quad &\text{if }\  |\zeta|\leq L^{\epsilon/4},\\
			O(e^{-cL^{3\epsilon/4}})\quad &\text{if }\  |\zeta|\geq L^{\epsilon/4},
		\end{cases}
	\end{equation}
	for $1\leq j\leq m$, where $\mathrm{f}_j(\zeta)$ is defined in \eqref{eq: limitingf}.
\end{lemma}
\smallskip
\subsection{Proof of Theorem~\ref{thm:special_IC}}\label{sec: verifyIC} In this section we verify that the flat initial condition satisfies Assumption~\ref{def:asympstab} with the limiting functions $E_{\mathrm{flat}}$ and $\cchi_{\mathrm{flat}}$ given by \eqref{eq: characteristicflat}. We start with a product formula for the pre-limit functions $\mathcal{E}_{\mathrm{flat}}(z)$ and $\chi_{\mathrm{flat}}(v,u;z)$.  
\begin{lemma}\label{lemma: pre-limit energy and char}
	Recall the global energy function $\mathcal{E}_{\vec y}(z)$ and characteristic function $\chi_{\vec y}(v,u;z)$ defined in Definition~\ref{def: globalenergy}. For the flat initial condition $\vec y=(-d,-2d,\cdots,-Nd)$ with $d=L/N\in \mathbb{N}$ we have 
	\begin{enumerate}[(i)]
	\item  With the standard square root function $\sqrt{w}$ with branch cut $\mathbb{R}_{\leq 0}$,
	\begin{equation}
		\mathcal{E}_{\mathrm{flat}}(z) = \frac{\prod_{v\in \mathcal{R}_z}\left(\sqrt{v+1}\right)^{2-d}}{\prod_{v\in \mathcal{R}_z}\sqrt{p(d-1)v^2+dv+1}}\left[\frac{\prod_{v\in \mathcal{R}_z}\prod_{u\in \mathcal{L}_z}\sqrt{v-u}}{\prod_{u\in \mathcal{L}_z}\left(\sqrt{-u}\right)^N\prod_{v\in \mathcal{R}_z}\left(\sqrt{v+1}\right)^{L-N}}\right].
	\end{equation}
   \item For $v\in \mathcal{R}_z$ and $u\in \mathcal{L}_z$,
   	\begin{equation}
   		\chi_{\mathrm{flat}}(v,u;z) =
   		\begin{cases}
   		 \frac{1}{J(v)}\frac{(v+1)^{L-N}u^N}{q_{z,\mathrm{L}}(v)q_{z,\mathrm{R}}(u)}\frac{1+pv}{1+pu}(u-v)\quad &\text{if }\ \frac{u(u+1)^{d-1}}{1+pu}=\frac{v(v+1)^{d-1}}{1+pv},\\
   		 0 &\text{otherwise},
   		 \end{cases}
   	\end{equation}
   where $J(w)=\frac{w(w+1)(1+pw)}{p(L-N)w^2+Lw+N}$, $q_{z,\mathrm{L}}(w):=\prod_{u\in \mathcal{L}_z}(w-u)$ and $q_{z,\mathrm{R}}(w):=\prod_{v\in \mathcal{R}_z}(w-v)$.
	\end{enumerate}
\end{lemma}
\begin{proof}
	The proof is similar to the proof of Lemma 10.2 and Lemma 10.3 in \cite{baik2019general}. The key observation is the existence of a $d-1$ to $1$ map $\mathcal{M}$ from $\mathcal{L}_z$ to $\mathcal{R}_z$ satisfying that if $v=\mathcal{M}(u)$ for $u\in \mathcal{L}_z$, then $\frac{u(u+1)^{d-1}}{1+pu}=\frac{v(v+1)^{d-1}}{1+pv}$. Using this relation we can express the global energy and characteristic functions in terms of products over the Bethe roots. We omit the details.
\end{proof}
Combining Lemma~\ref{lemma: pre-limit energy and char} with the asymptotics obtained in Lemma~\ref{lemma: asymptoticsofproducts} we can now prove Theorem~\ref{thm:special_IC}.
\begin{proof}[Proof of Theorem~\ref{thm:special_IC}]
	For fixed $0<\epsilon<1/2$ by Lemma~\ref{lemma: asymptoticsofproducts} (viii) we have 
	\begin{equation} \label{eq: asymptoticsprod}
		\frac{\prod_{v\in\mathcal{R}_z} \prod_{u\in\mathcal{L}_{z'}} \sqrt{v-u}}
		{\prod_{u\in\mathcal{L}_{z'}} \left(\sqrt{-u}\right)^N \prod_{v\in\mathcal{R}_z} \left( \sqrt{v+1} \right)^{L-N}}
		= 
		e^{-B(\mathrm{\z,\z'})} \left(1+O(L^{\epsilon-1/2})\right).
	\end{equation}
Similarly by Lemma~\ref{lemma: asymptoticsofproducts} (v) we have 
\begin{equation} \label{eq: asymptoticsO1}
	\prod_{v\in \mathcal{R}_z}\left(\sqrt{v+1}\right)^{2-d} = 1+O(L^{\epsilon-1/2}).
\end{equation}
On the other hand we have $p(d-1)v^2+dv+1=p(d-1)(v-w_c)(w-w_c^{*})$ where
\begin{equation*}
	 w_c=\frac{-2\rho}{1+\sqrt{1-4p\rho(1-\rho)}}\quad \text{and}\quad  w_c^{*}=\frac{-2\rho}{1-\sqrt{1-4p\rho(1-\rho)}}.
\end{equation*}
 Hence by Lemma~\ref{lemma: asymptoticsofproducts} (iv) and (v) we have 
 \begin{equation}\label{eq: asymptoticsquadratic}
 	\begin{aligned}
 		\prod_{v\in \mathcal{R}_z}\sqrt{p(d-1)v^2+dv+1}&= \prod_{v\in \mathcal{R}_z} \sqrt{v-w_c}\cdot \prod_{v\in \mathcal{R}_z}\sqrt{p(d-1)(w-w_c^{*})}\\
 		&=\left(\sqrt{-w_c}\right)^{N}e^{\frac{1}{2}h(0,\z)}\cdot \left(\sqrt{p(d-1)}\cdot \sqrt{-w_c^{*}}\right)^{N}\cdot \left(1+O(L^{\epsilon-1/2})\right)\\
 		&= (1-\z)^{1/4}\left(1+O(L^{\epsilon-1/2}\right),
 	\end{aligned}
 \end{equation}
where we used the fact that $w_c\cdot w_{c}^{*}=\frac{1}{p(d-1)}$ and also $e^{h(0,\z)}=(1-\z)^{1/2}$ which follows from \eqref{eq:def_h}. Combining \eqref{eq: asymptoticsprod}, \eqref{eq: asymptoticsO1} and \eqref{eq: asymptoticsquadratic} we conclude that $\mathcal{E}_{\mathrm{flat}}(z)= E_{\mathrm{flat}}(\z)(1+O(L^{\epsilon-1/2}))$ as $L\to \infty$. 

The argument for the characteristic function part is quite similar. To verify part (B) of Assumption~\ref{def:asympstab} we note that  given $0<\epsilon<1/8$, for $u\in \mathcal{L}_z^{(\epsilon)}$ and $v\in \mathcal{R}_z^{(\epsilon)}$ as defined in Lemma~\ref{lemma: convergenceofroots} we have 
\begin{equation}\label{eq: products}
 \frac{q_{z,\mathrm{L}}(v)}{(1+v)^{L-N}} = e^{h(\eta,\z)}\cdot \left(1+O(L^{4\epsilon-1/2})\right),\quad \frac{q_{z,\mathrm{R}}(u)}{u^{N}} = e^{h(\xi,\z)}\cdot \left(1+O(L^{4\epsilon-1/2})\right),
\end{equation}
where $\xi=\mathcal{M}_{L,\mathrm{left}}(u)$ and $\eta=\mathcal{M}_{L,\mathrm{right}}(v)$ with the injective maps $\mathcal{M}_{L,\mathrm{left}}$ and $\mathcal{M}_{L,\mathrm{right}}$ defined in Lemma~\ref{lemma: convergenceofroots} satisfying 
$\left|\xi-c_0^{-1}L^{1/2}(u-w_c)\right|\leq L^{-1/2+3\epsilon}\log L$, $\quad \left|\eta-c_0^{-1}L^{1/2}(v-w_c)\right|\leq L^{-1/2+3\epsilon}\log L$.
This then implies  that 
\begin{equation}\label{eq: difference}
	u-v = c_0L^{-1/2}(\xi-\eta)\cdot \left(1+O(L^{3\epsilon-1/2}\log L)\right).
\end{equation} 
A straightforward Taylor expansion combined with the injectivity of $\mathcal{M}_{L,\mathrm{left}}$ and $\mathcal{M}_{L,\mathrm{right}}$ shows 
\begin{equation}\label{eq: indicator}
	\mathbf{1}_{\frac{u(u+1)^{d-1}}{1+pu}=\frac{v(v+1)^{d-1}}{1+pv}} = \mathbf{1}_{\xi^2=\eta^2}=\mathbf{1}_{\xi=-\eta}.
\end{equation}
Finally by Lemma~\ref{lemma: asymptoticsofproducts} (iii) we have
\begin{equation}\label{eq: J(v)}
	\frac{1}{J(v)}= -c_0^{-1}\eta L^{1/2}\cdot (1+O(L^{\epsilon-1/2})).
\end{equation} 
Combining \eqref{eq: products}, \eqref{eq: difference}, \eqref{eq: indicator} and \eqref{eq: J(v)} we conclude that 
\begin{equation}
	\chi_{\mathrm{flat}}(v,u;z) = \cchi_{\mathrm{flat}}(\eta,\xi;\z)+ O(L^{4\epsilon-1/2}),\quad \text{as}\ L\to \infty.
\end{equation}
Finally part (C) in Assumption~\ref{def:asympstab} is clearly true since by Lemma~\ref{lemma: asymptoticsofproducts} every factor in for $\chi_{\mathrm{flat}}(v,u;z)$ is $O(1)$ except $\frac{1}{J(v)}$ which is $O(L)$. Thus $|\chi_{\mathrm{flat}}(v,u;z)|\leq C\cdot L$ and part (C) of Assumption~\ref{def:asympstab} is satisfied.
\end{proof}
\bigskip
\section{Multi-time distribution of discrete time parallel TASEP on \texorpdfstring{$\mathbb{Z}$}{Lg}}
\label{sec: infinite}
So far we have been focusing on the so-called relaxation time scale when the period $L$ is comparable to $t^{2/3}$ and the height fluctuation is critically affected by the boundary conditions. In this section we discuss another important regime, the sub-relaxation time scale $L\gg t^{2/3}$ when the particles are not affected by the boundary and are expected to behave the same as when the underlying space has infinite-volume. 
\subsection{Discrete time parallel TASEP with large period}
Instead of taking the re-scaled time parameter $\tau\to 0$ in the relaxation-time limiting distribution formula, we will start with the pre-limit joint distribution formula \eqref{eq: multipoint general}. The following elementary proposition confirms the  intuition that for fixed spatial and time parameters $a_i, k_i$ and $t_i$, when $L$ is large enough, the joint distribution $\mathbb{P}_{\vec y}^{(L)}(\cap_{i=1}^m\{x_{k_i}(t_i)\geq a_i\})$ agrees with the one when the underlying space is the whole integer lattice $\mathbb{Z}$, hence should be independent of $L$.

\begin{proposition} Given $\vec y=(y_1,\cdots,y_N)\in \conf_{N}$. Let $m$ be a positive integer and $k_i$ be integers in $\{1,\cdots,N\}$, for $1\leq i\leq m$. Then for any integers $a_i$, $1\leq i\leq m$, as long as $L$ is large enough such that 
	\begin{equation}\label{eq: constraint large}
		L>y_1-y_N\quad \text{and }\ L>\max\{a_1+k_1,\cdots,a_m+k_m\}-y_N,
	\end{equation}
we have 
	\begin{equation}\label{eq: large period}
				\mathbb{P}_{\vec y}^{(L)} \left( 
			\bigcap_{\ell=1}^m \left\{ x_{k_\ell} (t_\ell) \ge a_\ell 
			\right\}
			\right)=\mathbb{P}_{\vec y}^{(\infty)} \left( 
		\bigcap_{\ell=1}^m \left\{ x_{k_\ell} (t_\ell) \ge a_\ell 
		\right\}
		\right), 
	\end{equation}
for any positive integers $t_1,\cdots,t_m$. Here $\mathbb{P}_{\vec y}^{(L)}$ is the probability with respect to discrete time parallel TASEP on a periodic domain with period $L$ and $\mathbb{P}_{\vec y}^{(\infty)}$ is the probability with respect to an independent discrete time parallel TASEP on the infinite lattice $\mathbb{Z}$, both starting with the initial condition satisfying $x_i(0)=y_i$ for $1\leq i\leq N$.
\end{proposition}
\begin{proof}
	The proof is almost identical to the similar statement for continuous time TASEP as in Lemma 8.1 of \cite{baik2019multi} where the particle ordering is reversed. We omit the details. 
\end{proof}

In principle for every $L$ satisfying \eqref{eq: constraint large}, the left-hand side of \eqref{eq: large period} which has an expression as contour integrals of Fredholm determinant given by \eqref{eq: multipoint general} gives a formula for the multi-point joint distribution of discrete time parallel TASEP on $\mathbb{Z}$ for fixed parameters $a_i,k_i$ and $t_i$. The main obstacle here is the dependence on the extra parameter $L$ for the left-hand side of \eqref{eq: large period}. The key idea to get rid of the dependence on $L$ was already illustrated in Proposition~\ref{properties of the transition probability} (iii) and can be summarized as follows:
\begin{enumerate}
	\item By rewriting the sum over the Bethe roots as contour integrals using the residue theorem we get an analytic continuation of the integrand (for the outer integral with respect to  $z$) to $\{|z|>0\}$.
	\item Under proper assumption on the parameter $L$ (in fact exactly \eqref{eq: constraint large}), one can further show the integrand is analytic at $z=0$, thus the outer integral with respect to $z$ equals the evaluation of the integrand at $z=0$, which turns out to be independent of $L$.
\end{enumerate} 
Now as in the above discussion, we would like to send $z_1, \cdots, z_m$ all to $0$ in \eqref{eq: multipoint general} while still keep the nested relation $0<|z_m|< \cdots< |z_1|<\mathbbm{r}_c$. For this we change the integral variables  as follows: for given $L$ satisfying \eqref{eq: constraint large}, we set 
\begin{equation}
	\theta_{i}:= \frac{z_{i+1}^L}{z_{i}^L},\quad \text{for } 0\leq i\leq m-1,
\end{equation}
where $z_0:=1$ and $0<|z_m|< \cdots< |z_1|<\mathbbm{r}_c$. Then the corresponding relations between $\theta_i$'s are precisely $0<|\theta_0|<\mathbbm{r}_c^{L}:= \mathbf{r}_c$ and $0<|\theta_i|<1$ for $1\leq i\leq m-1$. After this change of variable, the finite-time multi-point joint distribution formula \eqref{eq: multipoint general} can be rewritten as 
\begin{equation}
	\mathbb{P}_{\mathrm{step}}^{(L)} \left( 
	\bigcap_{\ell=1}^m \left\{ x_{k_\ell} (t_\ell) \ge a_\ell 
	\right\}
	\right) = \oint\cdots\oint \hat{\mathscr{C}}_{\mathrm{step}}^{(L)}(\theta_0,\cdots,\theta_{m-1})\hat{\mathscr{D}}_{\mathrm{step}}^{(L)}(\theta_0,\cdots,\theta_{m-1}) \frac{d\theta_0}{2\pi \ii\theta_0}\cdots \frac{d\theta_{m-1}}{2\pi \ii\theta_{m-1}},
\end{equation}
where the integral contours are circle centered at origin with radius smaller than $\mathbf{r}_c$ for $\theta_0$ and circles centered at origin with radii less than $1$ for $\theta_i$, $1\leq i\leq m-1$. Here the integrand is defined such that
\begin{equation}
	\hat{\mathscr{C}}_{\mathrm{step}}^{(L)}(\theta_0,\cdots,\theta_{m-1}) := 	\mathscr{C}_{\mathrm{step}}^{(L)}(z_1,\cdots,z_m),
	\quad \hat{\mathscr{D}}_{\mathrm{step}}^{(L)}(\theta_0,\cdots,\theta_{m-1}) := 	\mathscr{D}_{\mathrm{step}}^{(L)}(z_1,\cdots,z_m),
\end{equation}
for any $z_1,\cdots,z_m$ satisfying the relations $z_i^L=\prod_{j=0}^{i-1}\theta_j$ for $1\leq i\leq m$. Note that by definition (see Section~\ref{sec: def of C and D}), $\mathscr{C}_{\mathrm{step}}^{(L)}(z_1,\cdots,z_m)$ and $\mathscr{D}_{\mathrm{step}}^{(L)}(z_1,\cdots,z_m)$ only depend on $z_1^L,\cdots,z_m^L$ so $\hat{\mathscr{C}}^{(L)}$ and $\hat{\mathscr{D}}^{(L)}$ are well-defined functions analytic in $\mathbb{D}_0(\mathbf{r}_c)\times (\mathbb{D}_0(1))^{m-1}$, where $\mathbb{D}_0(r)$ is the disk centered at $0$ with radius $r$ with the origin removed.
\smallskip
\subsection{Proof of Theorem~\ref{thm:main1}} As discussed above, Theorem~\ref{thm:main1} follows immediately once we establish the analyticity of $\hat{\mathscr{C}}_{\vec y}^{(L)}(\theta_0,\cdots,\theta_{m-1})$ and $\hat{\mathscr{D}}_{\vec y}^{(L)}(\theta_0,\cdots,\theta_{m-1})$ at $\theta_0=0$ so that we can get rid of the integral with respect to $\theta_0$ whose value equals the integrand evaluated at $\theta_0=0$. The precise statement needed are summarized in the following two lemmas which will be proved in the next two subsections. 
\begin{lemma}\label{lemma: Chat0}
	The function $\hat{\mathscr{C}}_{\mathrm{step}}^{(L)}(\theta_0,\cdots,\theta_{m-1})$ is in fact analytic in $\mathbb{D}(\mathbf{r}_c)\times (\mathbb{D}(1))^{m-1}$, where $\mathbb{D}(r)$ is a disk centered at $0$ with radius $r$. Moreover we have 
	\begin{equation}
		\hat{\mathscr{C}}_{\mathrm{step}}^{(L)}(0,\theta_1,\cdots,\theta_{m-1}) = \prod_{\ell=1}^{m-1}\frac{1}{1-\theta_{\ell}},
	\end{equation}
\end{lemma}

\begin{lemma}\label{lemma: Dhat0}
	For large enough integer $L$ satisfying \eqref{eq: constraint large}, the function $\hat{\mathscr{D}}_{\vec y}^{(L)}(\theta_0,\cdots,\theta_{m-1})$ is analytic in $\mathbb{D}(\mathbf{r}_c)\times (\mathbb{D}_0(1))^{m-1}$. Moreover we have for any $(\theta_1,\cdots,\theta_{m-1})\in (\mathbb{D}_0(1))^{m-1}$
	\begin{equation}
		\hat{\mathscr{D}}_{\mathrm{step}}^{(L)}(0,\theta_1,\cdots,\theta_{m-1}) = \mathcal{D}_{\mathrm{step}}^{(\infty)}(\theta_1,\cdots,\theta_{m-1}),
	\end{equation}
    where $\mathcal{D}_{\mathrm{step}}^{(\infty)}(\theta_1,\cdots,\theta_{m-1})$ is defined in Section~\ref{sec:Fredholm_representation}.
\end{lemma}

Combing Lemma~\ref{lemma: Chat0} and Lemma~\ref{lemma: Dhat0} we conclude that 
\begin{equation*}
\begin{aligned}
		\mathbb{P}_{\mathrm{step}}^{(L)} \left( 
	\bigcap_{\ell=1}^m \left\{ x_{k_\ell} (t_\ell) \ge a_\ell 
	\right\}
	\right) &= \oint\cdots\oint \hat{\mathscr{C}}_{\mathrm{step}}(0,\theta_1,\cdots,\theta_{m-1})\hat{\mathscr{D}}_{\mathrm{step}}(0,\theta_1,\cdots,\theta_{m-1}) \frac{d\theta_1}{2\pi \ii\theta_1}\cdots \frac{d\theta_{m-1}}{2\pi \ii\theta_{m-1}}\\
	&=\oint\cdots\oint \left[\prod_{\ell=1}^{m-1}\frac{1}{1-\theta_{\ell}}\right]\mathcal{D}_{\mathrm{step}}^{(\infty)}(\theta_1,\cdots,\theta_{m-1}) \frac{d\theta_1}{2\pi \ii\theta_1}\cdots \frac{d\theta_{m-1}}{2\pi \ii\theta_{m-1}}.
\end{aligned}
\end{equation*}
Finally using \eqref{eq: large period} we complete the proof of Theorem~\ref{thm:main1}.
\medskip
\medskip
\subsection{Proof of Lemma~\ref{lemma: Chat0}} By the definition of the function $\mathscr{C}_{\mathrm{step}}^{(L)}$ (see Definition~\ref{def: C general}) we know 
\begin{equation}
\begin{aligned}
	\hat{\mathscr{C}}_{\mathrm{step}}^{(L)}(\theta_0,\theta_1,\cdots,\theta_{m-1}) &= \left[\prod_{\ell=1}^{m-1}\frac{1}{1-\theta_{\ell}}\right] \left[\prod_{\ell=1}^{m}\frac{\prod_{u\in \mathcal{L}_{z_\ell}}(-u)^N\prod_{v\in \mathcal{R}_{z_{\ell}}}(v+1)^{L-N}}{\Delta(\mathcal{R}_{z_{\ell}};\mathcal{L}_{z_{\ell}})}\right]\\
	&\cdot \left[\prod_{\ell=1}^{m}\frac{E_{\ell}(z_{\ell})}{E_{\ell-1}(z_{\ell})}\right]\left[\prod_{\ell=2}^{m}\frac{\Delta(\mathcal{R}_{z_{\ell}};\mathcal{L}_{z_{\ell-1}})}{\prod_{u\in \mathcal{L}_{z_{\ell-1}}}(-u)^N\prod_{v\in \mathcal{R}_{z_{\ell}}}(v+1)^{L-N}}\right]\\
	&:=\left[\prod_{\ell=1}^{m-1}\frac{1}{1-\theta_{\ell}}\right] \mathscr{A}_1(z_1,\cdots,z_m)\mathscr{A}_2(z_1,\cdots,z_m)\mathscr{A}_3(z_1,\cdots,z_m).
\end{aligned}
\end{equation}
where $E_{\ell}(z)= \prod_{u\in \mathcal{L}_z} (-u)^{-k_\ell}\prod_{v\in \mathcal{R}_z}(v+1)^{-a_\ell-k_\ell}(pv+1)^{t_\ell-k_\ell}$ and the functions $\mathscr{A}_i$ represents the products inside each brackets. The parameters $\theta_i$ and $z_i$ are related by the equations $z_i^L=\prod_{j=0}^{i-1}\theta_j$ for $1\leq i\leq m$. Now if we send $\theta_0\to 0$, then all the $z_i$'s will converge to $0$. As a result, all the roots in $\mathcal{L}_{z_i}$  converge to $-1$ while all the roots in $\mathcal{R}_{z_i}$ converge to $0$, for all $1\leq i\leq m$. It is then easy to see that for $j=1,2,3$
\begin{equation}
	\lim_{\substack{z_i\to 0\\i=1,\cdots,m}} \mathscr{A}_j(z_1,\cdots,z_m) = 1
\end{equation}
is well-defined and hence $\hat{\mathscr{C}}_{\mathrm{step}}^{(L)}(\theta_0,\theta_1,\cdots,\theta_{m-1})$ can be analytically continued to the point $\theta_0=0$ with 
\begin{equation}
	\hat{\mathscr{C}}_{\mathrm{step}}^{(L)}(0,\theta_1,\cdots,\theta_{m-1}) = \prod_{\ell=1}^{m-1} \frac{1}{1-\theta_{\ell}}.
\end{equation}
\medskip

\subsection{Proof of Lemma~\ref{lemma: Dhat0}} This Lemma is an analogue of Lemma 5.2 of \cite{liu2019} and the proof is similar so we omit most of the details. What we want to prove is an identity between two Fredholm determinants, with the first one acting on discrete sets consisting of roots of certain polynomials and the other one acting on continuous contours. By looking at the standard series expansions of both Fredholm determinants, it suffices to show that every term in the two series expansions matches. These type of identities between multiple sums over roots of certain algebraic equations and multiple contour integrals can be understood as a highly nontrivial consequence of the residue theorem, which is stated and proved with enough generality for our purpose in Proposition 4.4 of \cite{liu2019}. We take $q(w)=w^{N}(w+1)^{L-N}(1+pw)^{-N}$ in the proposition and follow the same argument as in Lemma 5.2 of \cite{liu2019}.

\medskip
\subsection{KPZ scaling limit} In this final section we briefly discuss the large-time asymptotics for the multi-time distribution of discrete time parallel TASEP on $\mathbb{Z}$ under the $1:2:3$ KPZ scaling. For completeness we first recall the definition of the limiting distribution function $\mathbb{F}_{\mathrm{step}}^{\mathrm{kpz}}$ which essentially agrees with the function $F_{\mathrm{step}}$ appearing in Theorem 2.20 of \cite{liu2019} with a slight different choices of the parameters. The structure of $\mathbb{F}_{\mathrm{step}}^{\mathrm{kpz}}$ is similar to the finite-time distribution as contour integrals of a Fredholm determinant $\mathrm{D}_{\mathrm{step}}^{\mathrm{kpz}}(\theta_{1},\cdots,\theta_{m-1})$.

\subsubsection{Spaces of the operators}
Given integer $m\geq 1$. We first define the contour in the complex plane
\begin{equation*}
     \Gamma_{1,\RR} :=\{w = m+re^{\frac{\pi \ii}{3}}: r\geq 0\}\cup \{w = m+re^{\frac{-\pi \ii}{3}}: r\geq 0\}, 
\end{equation*}
oriented from $e^{-\frac{\pi\ii}{3}}\infty$ to $e^{\frac{\pi\ii}{3}}\infty$. Then for $2\leq j\leq m$, we define the contours 
\begin{equation*}
	\Gamma_{j,\RR}^{\pm}:= \Gamma_{1,\RR}\pm (j-1)=\{w\in \mathbb{C}: w= \tilde{w}\pm (j-1) \text{ for some } \tilde{w}\in \Gamma_{1,\RR}\}.
\end{equation*}
The left contours $\Gamma_{j,\LL}^{\pm}$, $j=2,\cdots,m$ and $\Gamma_{1,\LL}$ are defined simply as the reflections of the corresponding right contours about the $y$-axis, with orientation from $e^{-\frac{2\pi \ii}{3}}\infty$ to $e^{\frac{2\pi \ii}{3}}\infty$. We further denote 
\begin{equation*}
	\Gamma_{j,\LL}:=\Sigma_{j,\LL}^+\cup \Gamma_{\ell,\LL}^-, \qquad \Gamma_{j,\RR}:=\Gamma_{j,\RR}^+\cup \Gamma_{j,\RR}^-,\qquad \Gamma_{j}= \Gamma_{j,\LL}\cup \Gamma_{j,\RR},\qquad j=1,\cdots,m,
\end{equation*}
and
\begin{equation*}
	\mathbb{S}_1:= \Gamma_{1,\LL} \cup \Gamma_{2,\RR} \cup \cdots \cup \begin{dcases}
		\Gamma_{m,\LL}, & \text{ if $m$ is odd},\\
		\Gamma_{m,\RR}, & \text{ if $m$ is even},
	\end{dcases}
\end{equation*}
and 
\begin{equation*}
	\mathbb{S}_2:= \Gamma_{1,\RR} \cup \Gamma_{2,\LL} \cup \cdots \cup \begin{dcases}
		\Gamma_{m,\RR}, & \text{ if $m$ is odd},\\
		\Gamma_{m,\LL}, & \text{ if $m$ is even}.
	\end{dcases}
\end{equation*}

\subsubsection{Operators $\mathrm{K}_1^{\mathrm{kpz}}$ and $\mathrm{K}_{\mathrm{step}}^{\mathrm{kpz}}$}

Now we introduce the operators $\mathrm{K}_1^{\mathrm{kpz}}$ and $\mathrm{K}_{\mathrm{step}}^{\mathrm{kpz}}$ to define $	\mathrm{D}_{\mathrm{step}}^{\mathrm{kpz}}$ in Theorem~\ref{thm:main2}.

\begin{definition}\label{def: limiting infinite}
	We define
	\begin{equation*}
		\mathrm{D}_{\mathrm{step}}^{\mathrm{kpz}}(\theta_1,\cdots,\theta_{m-1})=\det\left( I - \mathrm{K}_1^{\mathrm{kpz}} \mathrm{K}_{\mathrm{step}}^{\mathrm{kpz}} \right),
	\end{equation*}
	where the two operators
	\begin{equation*}
		\mathrm{K}_1^{\mathrm{kpz}}: L^2(\mathbb{S}_2) \to L^2(\mathbb{S}_1),\qquad \mathrm{K}_{\mathrm{step}}^{\mathrm{kpz}}: L^2(\mathbb{S}_1)\to L^2(\mathbb{S}_2)
	\end{equation*}
	are defined by the kernels
	\begin{equation}
		\label{eq:K1}
		\mathrm{K}_1^{\mathrm{kpz}}(\zeta,\zeta'):= \left(\delta_i(j) + \delta_i( j+ (-1)^i)\right) \frac{ \mathrm{f}_i(\zeta) }{\zeta-\zeta'} Q_1^{(\infty)}(j)P_j(\zeta')\quad \text{for }\zeta'\in \Gamma_j\cap\mathbb{S}_2,\ \zeta\in \Gamma_i\cap\mathbb{S}_1,
	\end{equation}
	and
	\begin{equation}
		\label{eq:KY}
		\mathrm{K}_{\mathrm{step}}^{\mathrm{kpz}}(\zeta',\zeta):= 
		\left(\delta_j (i) + \delta_j(i - (-1)^j)\right) \frac{ \mathrm{f}_j(\zeta') }{\zeta'-\zeta} Q_2^{(\infty)}(i)P_i(\zeta)\quad \text{for }\zeta\in \Gamma_i\cap\mathbb{S}_1,\ \zeta'\in \Gamma_j\cap\mathbb{S}_2,
	\end{equation}
	where $1\leq i,j\leq m$. Here the function $\mathrm{f}_i(\zeta)$ is the same as the one appearing in the kernels for periodic case defined in \eqref{eq: limitingf} and $Q_1^{(\infty)}$ and $Q_2^{(\infty)}$ are defined in \eqref{eq: Q1} and \eqref{eq: Q2}. We also set
	\begin{equation}
		P_j(\zeta):= \begin{cases}
			\frac{1}{1-\theta_{j-1}},\quad & \zeta\in \Gamma_{j,\LL}^{+}\cup\Gamma_{j,\RR}^{+}, \quad j=2,\cdots,m\\
			\frac{-\theta_{j-1}}{1-\theta_{j-1}},\quad & w\in \Gamma_{j,\LL}^{-}\cup\Gamma_{j,\RR}^{-}, \quad j=2,\cdots,m\\
			1,\quad & w\in \Gamma_{1,\LL}\cup\Gamma_{1,\RR}. 
		\end{cases}
	\end{equation}
\end{definition}

\begin{proof}[Proof of Theorem~\ref{thm:main2}]
	The proof is a standard steepest descent analysis and most of the key ingredients are similar to  Lemma~\ref{lemma: seriesconvergence} so we only provide the essential calculations.  First it is not difficult to see that for any nonzero constants $c_1,\cdots,c_m$, the Fredholm determinant part $\mathcal{D}_{\mathrm{step}}^{(\infty)}$ is invariant if we replace the functions $F_i$ appearing in the kernels by $\tilde{F}_i:= c_i F_i$ for $1\leq i\leq m$. For our purpose we set $c_i=\frac{1}{F_i(w_c)}$ where $w_c = -\frac{1}{1+\sqrt{q}}$. We point out that this agrees with the critical point for periodic cases defined in \eqref{eq: critical point} when $\rho=1/2$ so the notation is consistent.  To ensure the kernels have sufficiently fast decay on each variable we make the same conjugation as in \eqref{eq: conjugated kernel} by setting 
	\begin{align*}
		&\tilde{\mathcal{K}}_{1}^{(\infty)}(w,w'):= \left( \delta_i(j) +\delta_i( j + (-1)^i )\right) 
		\frac{ \sqrt{\tilde{f}_i(w)}\sqrt{\tilde{f}_j(w')} }
		{w-w'} Q_1^{(\infty)}(j)P_j(w'),\\
		&\tilde{\mathcal{K}}_{\mathrm{step}}^{(\infty)}(w,w'):= \left( \delta_j(i) +\delta_j( i + (-1)^j )\right)
		\frac{\sqrt{\tilde{f}_j(w')}\sqrt{\tilde{f}_i(w)} }
		{w-w'} Q_1^{(\infty)}(j)P_i(w)		,
	\end{align*}
	for $w'\in \Sigma_j\cap\mathcal{S}_2,\ w\in \Sigma_i\cap\mathcal{S}_1$. Here $\tilde{f}_i$ is obtained by replacing $F_i$ with $\tilde{F}_i$ in \eqref{eq: f_i}. We also conjugate the limiting kernels in a similar way:
	\begin{align*}
		&\tilde{\mathrm{K}}_{1}^{\mathrm{kpz}}(\zeta,\zeta'):= \left( \delta_i(j) +\delta_i( j + (-1)^i )\right) 
		\frac{ \sqrt{\mathrm{f}_i(\zeta)}\sqrt{\mathrm{f}_j(\zeta')} }
		{  \zeta-\zeta'} Q_1^{(\infty)}(j)P_j(\zeta'),\\
		&\tilde{\mathrm{K}}_{\mathrm{step}}^{\mathrm{kpz}}(\zeta,\zeta'):= \left(\delta_j(i) +\delta_j( i + (-1)^j )\right) 
		\frac{ \sqrt{\mathrm{f}_j(\zeta')}\sqrt{\mathrm{f}_i(\zeta)} }
		{  \zeta-\zeta'} Q_1^{(\infty)}(j)P_i(\zeta),
	\end{align*}
    for  $\zeta'\in \Gamma_j\cap\mathbb{S}_2,\ \zeta\in \Gamma_i\cap\mathbb{S}_1$. These conjugations do not change the Fredholm determinants. 
	 
	Now a straightforward Taylor expansion shows that for $w=w_c+ \frac{q^{1/4}}{1+\sqrt{q}}\zeta T^{-1/3}$ with $|\zeta|\leq T^{\epsilon/4}$ for $0<\epsilon<1/3$ we have 
	\begin{equation}
	\begin{aligned}
		\frac{F_i(w)}{F_i(w_c)} &=\exp\left(k_i \log\left(\frac{w}{w_c}\right)-(a_i+k_i)\log\left(\frac{w+1}{w_c+1}\right)+(t_i-k_i)\log\left(\frac{pw+1}{pw_c+1}\right)\right)\\
		& = \exp\left(-\frac{1}{3}\tau_i\zeta^3 + \frac{1}{2}\gamma_i\zeta^2+ \mathrm{x}_i\zeta + O(T^{\epsilon-\frac{1}{3}})\right),
	\end{aligned}
	\end{equation}
where the parameters $a_i, k_i$ and $t_i$ satisfies \eqref{eq: rescale infinite}. On the other for $w$ far away from $w_c$, $\frac{F_i(w)}{F_i(w_c)}$ decays or grows exponentially fast. Hence it is straightforward to show that for $w= -\frac{1}{1+\sqrt{q}}+ \frac{q^{1/4}}{1+\sqrt{q}}\zeta T^{-1/3}$ and $0<\epsilon<1/3$ we have 
\begin{equation}\label{eq: estimates for f_i}
	\tilde{f}_{j}(w) = \begin{cases}
		\mathrm{f}_j(\zeta)(1+O(T^{\epsilon-1/3}))\quad &\text{if }|\zeta|\leq T^{\epsilon/4},\\
		O(e^{-cT^{\epsilon/2}})\quad &\text{if }|\zeta|\geq T^{\epsilon/4}.
	\end{cases}
\end{equation}
Now we deform the contours $\Sigma_{j}$ to be sufficiently close to the critical points $w_c=-\frac{1}{1+\sqrt{q}}$ such that locally they look like the limiting contours $\Gamma_j$. More concretely we deform $\Sigma_{j,\RR}^{+}$ such that
\begin{align*}
 \Sigma_{j,\RR}^{+}\cap \{w\in \mathbb{C}: |w-w_c|\leq \frac{1+\sqrt{q}}{q^{1/4}}T^{\epsilon/4-1/3}\}&= \{w_c+T^{-1/3}(m+j-1+re^{\frac{\pi \ii}{3}}): 0\leq r\leq T^{\epsilon/4}\}\\
 &\cup \{w_c+T^{-1/3}(m+j-1+re^{-\frac{\pi \ii}{3}}): 0\leq r\leq T^{\epsilon/4}\},
 \end{align*}
and similar for other contours. Then by \eqref{eq: estimates for f_i} it is straightforward to show that 
\begin{enumerate}
\item For each $n\in \mathbb{N}$ and $0<|\theta_i|<1$ we have 
\begin{equation*}
	\lim_{T\to \infty} \mathrm{Tr}\left(\tilde{\mathcal{K}}_{1}^{(\infty)}\tilde{\mathcal{K}}_{\mathrm{step}}^{(\infty)}\right)^n = \mathrm{Tr}\left(\tilde{\mathrm{K}}_{1}^{\mathrm{kpz}}\tilde{\mathrm{K}}_{\mathrm{step}}^{\mathrm{kpz}}\right)^n.
\end{equation*}
\item There exists constant $C>0$ independent of $T$ and $n$ such that for all $n\in \mathbb{N}$
\begin{equation*}
	\left|\oint_{\mathcal{S}_1}\cdots\oint_{\mathcal{S}_1}\det\left[\left(\tilde{\mathcal{K}}_{1}^{(\infty)}\tilde{\mathcal{K}}_{\mathrm{step}}^{(\infty)}\right)(w_i,w_j)\right]_{i,j=1}^n \frac{\dd w_1}{2\pi \ii}\cdots \frac{\dd w_n}{2\pi \ii}\right|\leq C^n.
\end{equation*}
\end{enumerate}
Now (1) and (2) immediately implies $\mathcal{D}_{\mathrm{step}}^{(\infty)}(\theta_1,\cdots,\theta_{m-1}) \text{ converges to } \mathrm{D}_{\mathrm{step}}^{\mathrm{kpz}}(\theta_1,\cdots,\theta_{m-1})$ locally uniformly for $0<|\theta_i|<1$ as $T\to \infty$, thus proving Theorem~\ref{thm:main2}.
\end{proof}
\bibliographystyle{alpha}

\medskip

\textsc{Yuchen Liao, Department of Mathematics, University of Warwick, Coventry CV4 7AL, UK}

Email address: \texttt{Yuchen.Liao@warwick.ac.uk}

\end{document}